%% file: main.tex
\documentclass[leqno]{amsart}
\usepackage[a4paper,vmargin={3cm,3cm},hmargin={3cm,3cm}]{geometry}
\usepackage[T1]{fontenc}
\usepackage[utf8]{inputenc}
\usepackage[english]{babel}
\usepackage{csquotes}
\usepackage{lmodern}
\usepackage{amsmath,amssymb,amsthm,mathtools}
\usepackage{microtype}
\usepackage{cite}
\usepackage[bookmarksdepth=1,hyperfootnotes=false]{hyperref}
\hypersetup{colorlinks=true,linkcolor=blue,urlcolor=blue,citecolor=blue}
\usepackage{fancyhdr}
\usepackage{indentfirst}
\usepackage{float}
\usepackage{placeins}

\renewcommand{\!}{}

\theoremstyle{plain}
\newtheorem{theorem}{Theorem}[section]
\newtheorem{proposition}[theorem]{Proposition}
\newtheorem{lemma}[theorem]{Lemma}
\newtheorem{corollary}[theorem]{Corollary}

\theoremstyle{definition}
\newtheorem{definition}[theorem]{Definition}

\newtheorem{assumption}[theorem]{Assumption}
\newtheorem{example}[theorem]{Example}
\newtheorem{problem}{Problem}

\newcommand{\R}{\mathbb{R}}
\newcommand{\C}{\mathbb{C}}
\newcommand{\Z}{\mathbb{Z}}
\newcommand{\Q}{\mathbb{Q}}
\newcommand{\N}{\mathbb{N}}
\newcommand{\abs}[1]{\left\lvert #1 \right\rvert}

\newcommand{\evalG}{\operatorname{eval}_{\Gamma}}
\newcommand{\evalX}{\operatorname{eval}_{\mathbf X}}
\newcommand{\spG}{\operatorname{sp}_{\Gamma}}
\newcommand{\evalH}[2]{\left[#1\right]_{#2}}
\newcommand{\len}{\operatorname{len}}
\newcommand{\TBe}[1]{\mathrm{TB}_{#1}}
\newcommand{\TB}{\mathrm{TB}}
\newcommand{\vol}{\operatorname{vol}}
\newcommand{\covol}{\operatorname{covol}}

\newcommand{\dgt}{\operatorname{dgt}}
\newcommand{\ord}{\operatorname{ord}}
\newcommand{\freq}{\operatorname{freq}}
\newcommand{\revsig}{\overset{\leftarrow}{\sigma}}

\title[Terminal blocks of primes in Pisot numeration systems]{Terminal blocks of primes in Pisot numeration systems}
\author{Sungkon Chang and Johann Verwee}
\date{}

\numberwithin{equation}{section}

\begin{document}

\raggedbottom

\begin{abstract}
We prove a prime number theorem for fixed terminal block words in a family of
Pisot integer numeration systems defined by canonical block expansions.
Suppose that the dominant root $\varphi$ is a Pisot number and that the
characteristic polynomial $P_h$ is its minimal polynomial. Then every terminal
block word of total digit-length $m$ occurs among the primes with asymptotic
frequency $\varphi^{-m}$. The result recovers the known Zeckendorf and generalized Fibonacci cases and extends the terminal-prime problem to irreducible Pisot-unit systems without assuming that the recurrence coefficients are monotone, that the canonical expansions are greedy, or that the finiteness property (F) holds. Explicit non-$\beta$ families and Pisot units failing (F) illustrate this extension.

The main structural ingredient is a canonical block language whose expansion
of the companion substitution fixed word records the corresponding integer
expansions in natural order. This connects terminal conditions with Rauzy
cylinder windows and then with a linear orbit on a compact torus. The required
multiplicity-one geometry follows from Barge's pure-discreteness theorem after
reversal of the substitution words. The same toral model yields prime
genericity for finite factors of the substitution fixed word, polynomial
sampling laws, asymptotic independence from fixed congruence classes,
fixed-shift correlation formulas, and M\"obius orthogonality. Combined with
established prime theorems, it also gives a terminal refinement of Chebotarev
and arithmetic progressions inside fixed terminal prime classes.
\end{abstract}

\maketitle

\tableofcontents

\input{sections/01-introduction}
\input{sections/02-general-data}

\input{sections/03-structural-coding}
\input{sections/04-rauzy-windows}
\input{sections/05-full-terminal-partition}
\input{sections/06-toral-model}
\input{sections/07-prime-terminal-blocks}
\input{sections/08-rauzy-verification}
\input{sections/09-outlook-and-further-directions}

\bibliographystyle{abbrv}
\bibliography{refs}

\end{document}

%% file: sections/01-introduction.tex
\section{Introduction}\label{sec:introduction}

The distribution of primes in arithmetic progressions is a central
theme of analytic number theory. In 1837,
Dirichlet~\cite{Dirichlet1837} proved that every arithmetic progression
$a\bmod q$ with $\left(a,q\right)=1$ contains infinitely many primes. In its
modern quantitative form, the prime number theorem in arithmetic progressions
says the following.
Fix integers $q\geqslant 2$, $k\geqslant 1$, and $a$ with
$\left(a,q\right)=1$. Let $\pi\left(X\right)$ denote the number of
primes at most $X$, and let $\phi$ denote Euler's totient function
\cite[Chapters~4 and~20]{MontgomeryVaughanMNT1}. Then, as
$X\to+\infty$,
$$
\#\left\{p\leqslant X:p\ \mathrm{prime},\
p\equiv a\pmod{q^k}\right\}
\sim \frac{\pi\left(X\right)}{\phi\left(q^k\right)}.
$$
This analytic statement has a concrete digit interpretation. In base $q$,
prescribing the last $k$ digits of an integer is exactly the same as
prescribing one residue class modulo $q^k$. Thus, in ordinary bases,
terminal-digit conditions are congruence conditions. The problem considered
in this paper asks for an analogue of this elementary picture in recurrence
numeration systems, where terminal blocks still make sense but are no longer
described by congruences modulo powers of a base.

The first model beyond ordinary bases is Zeckendorf numeration. Recall
the Fibonacci sequence $\left(F_n\right)_{n\geqslant1}$ defined by
$$
F_1=F_2=1,\qquad
F_{n+2}=F_{n+1}+F_n\quad\left(n\geqslant1\right).
$$
The classical theorem of Lekkerkerker and
Zeckendorf~\cite{Lekkerkerker1952,Zeckendorf1972} gives every
non-negative integer a unique expansion as a sum of non-consecutive
Fibonacci numbers. Generalized representation theorems for related
Fibonacci-type sequences were developed early on by Daykin and Hoggatt
\cite{Daykin1960,Hoggatt1972}. Since $F_1=F_2$, we use the convention $a_1=0$.
Thus, if
$
p=a_1F_1+a_2F_2+\cdots+a_nF_n
$
is the Zeckendorf expansion of a prime $p$, then its first terminal
digit is $a_2$.

\begin{theorem}\label{thm:intro-zeckendorf}
Let $\varphi$ be the golden ratio. For each $c\in\left\{0,1\right\}$, as
$X\to+\infty$,
$$
\#\left\{p\leqslant X:p\ \mathrm{prime},\ a_2=c\right\}
\sim \frac{\pi\left(X\right)}{\varphi^{c+1}}.
$$
\end{theorem}

This particular asymptotic is contained in the more general terminal-block
result of Davletyarova, Zhukova, and Shutov~\cite{DZS2014}; it also follows
from \cite[Lemma~4.1]{DrmotaMuellnerSpiegelhofer2018} together with
Vinogradov's equidistribution theorem. This latter derivation was one of
the motivations for the present generalization.

We consider the same type of question with the Fibonacci sequence replaced by a
sequence defined by a linear recurrence. The basic recurrence data are as follows.

\begin{definition}\label{def:intro-recurrence-data}
Fix an order $L\geqslant 2$ and a coefficient vector
$$
g=\left(1,g_2,\ldots,g_L\right),
$$
where $g_2,\ldots,g_L$ are positive integers. We use the reversed coefficient vector
$$
h=(h_1,\ldots,h_L),\qquad
h_r:=g_{L+1-r}\quad \left(1\leqslant r\leqslant L\right),
$$
so that $h_L=g_1=1$. The coefficients $h_1,\ldots,h_L$ determine a
sequence of weights $\left(H_n\right)_{n\geqslant1}$, defined by $H_1=1$,
$$
H_{r+1}=h_1H_r+h_2H_{r-1}+\cdots+h_rH_1
\qquad \left(1\leqslant r<L\right),
$$
and
$$
H_{n+L}=h_1H_{n+L-1}+h_2H_{n+L-2}+\cdots+h_LH_n
\qquad \left(n\geqslant1\right).
$$
\end{definition}

A slightly different choice of initial values is used in the standard
references
\cite{Fraenkel1985,GrabnerTichy1990,HamlinWebb2012,MillerWang2012}.
One reason for our choice of initial values is combinatorial: for every
$m\geqslant0$, $H_{m+1}$ is the number of admissible coefficient words
of length $m$. Another reason comes from ordinary bases. Although our
standing assumptions require $L\geqslant2$, the same normalization
extends naturally to order one: for
$$
H_1=1,\qquad H_{n+1}=qH_n\quad\left(n\geqslant1\right),
$$
with $q\geqslant2$, one obtains $H_n=q^{n-1}$, which is exactly the
usual base-$q$ weight sequence, and the corresponding reduced
$H$-expansions are the usual base-$q$ expansions.

The associated characteristic polynomial is
\begin{equation}\label{eq:intro-characteristic-polynomial}
\begin{aligned}
P_h(X)
&=X^L-h_1X^{L-1}-h_2X^{L-2}-\cdots-h_{L-1}X-1\\
&=X^L-g_LX^{L-1}-g_{L-1}X^{L-2}-\cdots-g_2X-1.
\end{aligned}
\end{equation}
Since $h_1,\ldots,h_L$ are positive, $P_h$ has a unique positive
dominant root $\varphi>1$, which is simple, by \cite[Proposition~27]{Chang2021} and \cite[Theorem~12.2]{Ostrowski1973}.

Since $H_1=1$, every non-negative integer can be written as a finite
non-negative integer combination of the weights $H_1,H_2,\ldots$ If
arbitrary coefficients are allowed, however, such a representation is not
unique, precisely because the weights satisfy a recurrence relation. For
uniqueness, we choose a certain admissible coefficient sequence
\begin{equation}\label{eq:intro-reduced-H-expansion}
n=\sum_{k\geqslant 1}a_kH_k,
\end{equation}
where only finitely many coefficients are nonzero and the word
$\left(a_1,a_2,\ldots\right)$ satisfies the admissibility constraints imposed by the
recurrence. We call this the \emph{canonical infinite $H$-expansion}
of $n$. A finite reduced representative will be defined in Section~\ref{sec:general-data}.
This admissible representative need not coincide with the usual greedy
expansion. Thus our objects are canonical infinite expansions of
non-negative integers in the basis $H_1,H_2,\ldots$, not beta-expansions
of real numbers in powers of $\varphi^{-1}$.

We introduce the relevant notions through the following example.

\begin{example}\label{ex:intro-order-two}
For $g=\left(1,2\right)$, we have $h=\left(2,1\right)$, so
$H_1=1$, $H_2=2$, $H_3=5$, $H_4=12$, and for $n\geqslant 1$,
$H_{n+2}=2H_{n+1}+H_n$. We write coefficient words from the small-weight end, starting with the
coefficient of $H_1$, then the coefficient of $H_2$, and so on. In this
example, the admissible representatives coincide with the usual greedy
expansions. For instance,
$$
14=H_2+H_4,\qquad 15=H_1+H_2+H_4,\qquad
21=2H_2+H_3+H_4.
$$
Thus, omitting trailing zeros, the corresponding coefficient words are
$$
14:\ \left(0,1,0,1\right),\qquad
15:\ \left(1,1,0,1\right),\qquad
21:\ \left(0,2,1,1\right).
$$
The pair $\left(0,2\right)$ may occur as one admissible package rather
than as two independent digits. Every finite reduced representative in this example
decomposes uniquely into a concatenation of
$$
\left(0\right),\qquad \left(1\right),\qquad \left(0,2\right).
$$
These packages are called the \emph{semi-proper blocks} of this system:
$$
b_{1,0}=\left(0\right),\qquad
b_{1,1}=\left(1\right),\qquad
b_{2,0}=\left(0,2\right).
$$
The semi-proper block at the small-weight end of the canonical infinite
block expansion of $n$ is called the \emph{terminal block} of $n$ and is
denoted by $\TB(n)$. Hence
$$
\TB(14)=b_{1,0},\qquad
\TB(15)=b_{1,1},\qquad
\TB(21)=b_{2,0}.
$$
\end{example}

In general, the canonical infinite block expansion of every non-negative
integer is a concatenation of semi-proper blocks. For a semi-proper block
$b$, let $\len(b)$ denote its number of digits. The condition
$\TB(n)=b$ prescribes its first semi-proper block. Prescribing several
consecutive terminal blocks is a closer analogue of fixing a finite string
of terminal digits in an ordinary base.

We use the theory of Pisot substitutions to generalize
Theorem~\ref{thm:intro-zeckendorf}. The symbolic object associated with the
recurrence is the following.

\begin{definition}
\label{def:intro-associated-substitution}
Let $\mathcal A:=\left\{A_1,\ldots,A_L\right\}$. The \emph{associated
substitution} $\sigma$ on $\mathcal A$ is given by
$$
\sigma(A_i)=A_1^{h_i}A_{i+1}\quad\left(1\leqslant i<L\right),
\qquad
\sigma(A_L)=A_1.
$$
Here a substitution is a rule replacing each letter by a finite word and
extending to finite and one-sided infinite words by concatenation. We say
that a finite word $U$ is a \emph{prefix} of a finite or one-sided infinite
word $V$ if there exists a finite or one-sided infinite word $W$ such that
$V=UW$. Because $\sigma(A_1)$ begins with $A_1$, the word
$\sigma^k(A_1)$ is a prefix of $\sigma^{k+1}(A_1)$ for every integer
$k\geqslant0$. Hence there is a unique right-infinite word $A_\infty$
having every $\sigma^k(A_1)$ as a prefix. Equivalently,
$$
A_\infty:=\lim_{k\to\infty}\sigma^k(A_1).
$$
It satisfies $\sigma(A_\infty)=A_\infty$, and we call it the
\emph{fixed word} of $\sigma$.
\end{definition}

For the system $g=(1,3,2)$, equivalently $h=(2,3,1)$, which will
serve as a recurring example throughout the paper, the associated
substitution is
\[
\sigma(A_1)=A_1A_1A_2,\qquad
\sigma(A_2)=A_1A_1A_1A_3,\qquad
\sigma(A_3)=A_1.
\]

Since $h_1,\ldots,h_L$ are positive, this substitution is primitive,
meaning that some iterate of $\sigma$ sends every letter to a word
containing every letter of the alphabet. The relations between linear numeration systems, finite beta-expansions,
and substitution dynamics are studied in \cite{FrougnySolomyak1992,Hollander1996}; see also
Chang's recent preprint on the Zeckendorf shift and finite beta-expansions
\cite{ChangZeckendorfShift}.

The incidence matrix of $\sigma$ has characteristic polynomial
$P_h$. Hence, when the dominant root
$\varphi$ is a Pisot number and $P_h$ is its minimal polynomial, this
characteristic polynomial is irreducible, its dominant eigenvalue is
$\varphi$, and all its other eigenvalues have modulus strictly less than
$1$. A primitive substitution with these properties is called an
irreducible Pisot substitution.

The classical $\beta$-numeration framework provides useful context for our
algebraic assumptions. A number $\beta>1$ has the finiteness
property~(F) if every nonnegative element of $\Z[\beta^{-1}]$ has
a finite greedy $\beta$-expansion. The monotone and dominant coefficient
conditions appearing in Theorem~\ref{thm:intro-unconditional-cases}
are classical sufficient conditions for~(F)
\cite{FrougnySolomyak1992,Hollander1996}. In the classical
$\beta$-substitution setting, property~(F) is a standard sufficient
condition for the pure-discrete tiling geometry; see, for example,
\cite{AkiyamaBargeBertheLeeSiegel2015}.
Our main results, however, require neither~(F) nor agreement of the
canonical $H$-expansions with greedy expansions. Instead, they assume
that the dominant root $\varphi$ is a Pisot number and that $P_h$ is
its minimal polynomial.

The Pisot Substitution Conjecture predicts pure discrete spectrum for
every irreducible Pisot substitution, both for the symbolic substitution
system and for the associated one-dimensional tiling system
\cite[Conjectures~4.1 and~4.2]{AkiyamaBargeBertheLeeSiegel2015}. For
irreducible Pisot substitutions these two formulations are equivalent by
Clark--Sadun; see the discussion following Conjecture~4.2 in
\cite{AkiyamaBargeBertheLeeSiegel2015} and
\cite[Theorem~3.1]{ClarkSadun2003}.

For the companion substitutions arising here, the required geometric
input follows from existing pure-discreteness theory. After reversing each
substitution word, the substitution is injective on initial letters and
constant on final letters, so Barge's theorem
\cite[Theorem~3.13]{Barge2016} gives pure discrete spectrum for the
associated tiling system. Reversal exchanges the two symbolic languages
and conjugates one shift with the inverse of the other; the
symbolic--tiling equivalence therefore transfers Barge's conclusion back
to our original substitution. In the irreducible unimodular setting
considered here, pure discrete spectrum is equivalent, by
\cite[Theorem~4.8]{BaratBertheLiardetThuswaldner2006}, to the associated
Rauzy fractal generating the corresponding lattice tiling.
Section~\ref{sec:rauzy-verification} identifies this geometry with the
stable coordinates and lattice used below. Thus no instance of the Pisot
Substitution Conjecture is used as an assumption in our argument: for the
companion family considered here, the required pure discreteness follows
from Barge's theorem after reversal.

We can now state the main result without developing the geometric
machinery used in its proof.

\begin{theorem}\label{thm:intro-main}
Suppose that the dominant root $\varphi$ is a Pisot number and that
$P_h$ is its minimal polynomial. Then, for every semi-proper block $b$
of the associated numeration system, as $X\to+\infty$,
$$
\#\left\{
p\leqslant X:
p\ \mathrm{prime},\
\TB(p)=b
\right\}
\sim
\frac{\pi\left(X\right)}{\varphi^{\len(b)}}.
$$
\end{theorem}

In the Zeckendorf case, the terminal blocks are $\left(0\right)$ and
$\left(0,1\right)$. They have lengths $1$ and $2$, corresponding
respectively to $a_2=0$ and $a_2=1$. Thus
Theorem~\ref{thm:intro-main} recovers
Theorem~\ref{thm:intro-zeckendorf}.

The single-block statement is the first case of the following stronger result, in which any fixed number of consecutive terminal blocks may be prescribed.

\begin{theorem}
\label{thm:intro-several-blocks}
Under the hypotheses of Theorem~\ref{thm:intro-main}, let
$e\geqslant1$ be an integer and let $B$ be any concatenation of $e$
semi-proper blocks. Denote by $\TBe{e}(n)$ the concatenation, starting
from the small-weight end, of the first $e$ semi-proper blocks in the
canonical infinite block expansion of $n$, and let $\len(B)$ be the total number of digits
in $B$. Then, as $X\to+\infty$,
$$
\#\left\{
p\leqslant X:
p\ \mathrm{prime},\
\TBe{e}(p)=B
\right\}
\sim
\frac{\pi\left(X\right)}{\varphi^{\len(B)}}.
$$
\end{theorem}

For Zeckendorf numeration, Theorem~\ref{thm:intro-several-blocks}
recovers, up to the harmless indexing shift between the two standard
Fibonacci conventions, the prime asymptotic for prescribed terminal
blocks established by Davletyarova, Zhukova, and Shutov~\cite{DZS2014}.

The algebraic hypothesis in these theorems is itself automatic in
several concrete families.

\begin{theorem}\label{thm:intro-unconditional-cases}
The conclusions of Theorems~\ref{thm:intro-main} and
\ref{thm:intro-several-blocks} hold unconditionally in the monotone
branch $g_2\leqslant g_3\leqslant\cdots\leqslant g_L$, in the
dominant branch $\sum_{k=1}^{L-1}g_k<g_L$, and for the two families
\[
g=\left(1,\underbrace{c+1,\ldots,c+1}_{L-2\text{ terms}},c\right),
\qquad c\geqslant1,
\]
and
\[
g=\left(1,c+1,\underbrace{c,\ldots,c}_{L-2\text{ terms}}\right),
\qquad L\geqslant4,\quad c\geqslant3.
\]
For $L\geqslant3$, the first family is non-$\beta$ in the sense that its
companion substitution is not the standard simple-Parry
$\beta$-substitution associated with $\varphi$; it contains the recurring
example $g=(1,3,2)$. For the second family, the associated Pisot unit fails
the finiteness property~(F).
\end{theorem}

The failure of~(F) for the second family is established in
Section~\ref{sec:nonfinite-beta-family}.

The Rauzy torus also gives an application in the opposite direction,
which no longer refers to terminal blocks. Write
$A_\infty=u_0u_1u_2\cdots$ for the substitution fixed word. For every
finite factor $W$ of $A_\infty$, the frequency of occurrences of $W$
starting at prime positions is equal to its ordinary factor frequency:
\[
\frac1{\pi(X)}
\#\left\{p\leqslant X:
 u_pu_{p+1}\cdots u_{p+|W|-1}=W\right\}
\longrightarrow
\operatorname{freq}(W).
\]
Equivalently, the same statement has a von Mangoldt weighted form. This
is proved in Theorem~\ref{thm:prime-fixed-word} by lifting the fixed word
to its Rauzy torus, applying Vinogradov's theorem there, and projecting
back to the symbolic coding.

The toral model also gives several further consequences. Fixed terminal
words have the expected frequencies along polynomial samples and, by
Rhin's theorem, along polynomial values at prime arguments. Terminal
blocks are asymptotically independent of every fixed reduced congruence
class, successive terminal blocks of primes satisfy an asymptotic
Bernoulli law, fixed shifts admit an exact toral correlation formula,
and every fixed terminal cylinder is orthogonal to the M\"obius
function.

The same coding can be combined with established prime theorems. We
obtain a terminal Chebotarev theorem, showing asymptotic independence
between terminal blocks and Frobenius conjugacy classes. Moreover, the
positive relative density of every terminal prime class, together with
Shao's refinement of the Green--Tao theorem~\cite{GreenTao2008,ShaoNarrow2017}, implies that every such
class contains arbitrarily long arithmetic progressions with
polylogarithmically bounded common difference.

Theorems~\ref{thm:intro-main} and
\ref{thm:intro-several-blocks} complement the existing literature on
digital functions along primes. Mauduit and
Rivat~\cite{MauduitRivat2010} solved the Gelfond problem for the sum of
digits of primes in a fixed integer base. For Zeckendorf expansions and
more general morphic numeration systems, Drmota, M\"ullner and
Spiegelhofer~\cite{DrmotaMuellnerSpiegelhofer2018,DrmotaMuellnerSpiegelhofer2025}
developed methods for studying digit functions along primes. In their
discussion of the Tribonacci substitution, they explicitly suggested a
prime number theorem based on pure discrete spectrum and the classical
Rauzy fractal
\cite[Section~9.2]{DrmotaMuellnerSpiegelhofer2025}.
Theorem~\ref{thm:prime-fixed-word} establishes this fixed-word statement
and, more generally, gives prime genericity for every finite factor of the
companion fixed words considered here. This fixed-word statement is distinct from the broader Tribonacci
sum-of-digits problem posed in
\cite[Problem~8]{DrmotaMuellnerSpiegelhofer2025}. An explicit computation
for Tribonacci factors of length two is given in Section~7.3. Drmota and
Steiner~\cite{DrmotaSteiner2002} treated Zeckendorf digit statistics along
polynomial sequences and along polynomial values at primes.

For prescribed terminal blocks in Zeckendorf numeration, Davletyarova, Zhukova, and Shutov~\cite{DZS2014} established the corresponding prime asymptotic, and subsequently extended this result to a family of order-two generalized Fibonacci numeration systems~\cite{DZS2016}. Shutov later established a prime number theorem for prescribed endings of greedy expansions in linear recurrence systems with nonincreasing coefficients~\cite{Shutov2021}, and studied densities and first differences for numbers with fixed last digits in linear recurrent bases~\cite{Shutov2024}. Chang~\cite{Chang2025TerminalDigits} later formulated the Zeckendorf asymptotic independently as a conjecture.

The present work was developed independently of these earlier results
and shares with Shutov's approach the use of Rauzy geometry and toral
dynamics. Its main distinction lies in an arithmetic--substitution
duality: Theorem~\ref{thm:terminal-coding} relates the entire
canonical $H$-expansion of an integer to the decomposition of a prefix
of the companion substitution fixed word, not only to its terminal
block. This duality provides the arithmetic coding used in our prime
asymptotics. Our main theorems apply to the whole companion family
under the irreducible Pisot-unit hypotheses, without imposing
monotonicity on the recurrence coefficients or requiring greedy
canonical expansions or property~(F). In contrast, the terminal-block
results of Shutov~\cite{Shutov2021,Shutov2024} are stated for
nonincreasing recurrence coefficients. The explicit families of
Theorem~\ref{thm:intro-unconditional-cases} give nonmonotone and
non-$\beta$ examples, while Section~\ref{sec:nonfinite-beta-family}
gives an infinite family of Pisot units failing~(F).

\subsection*{Strategy of the proof}

The geometric construction starts from the two algebraic hypotheses of
Theorem~\ref{thm:intro-main}: the dominant root $\varphi$ is a Pisot
number and $P_h$ is its minimal polynomial. For modularity of the argument,
Section~\ref{sec:rauzy} isolates the precise multiplicity-one statement
needed later as Theorem~\ref{thm:rauzy-fundamental-domain}, whose proof
is given in Section~\ref{sec:rauzy-verification} under the same
algebraic hypotheses.

The Pisot algebraic hypotheses allow us to associate with each canonical
infinite $H$-expansion the vector of its evaluations at the non-dominant
conjugates of $\varphi$. Since these conjugates have modulus strictly
less than $1$, the resulting sets of evaluation vectors are bounded. For each terminal block $b$, we call the closure of the
corresponding set of evaluation vectors the \emph{Rauzy piece associated
with $b$}, following the Rauzy-fractal framework initiated by
Rauzy~\cite{Rauzy1982} and developed by Arnoux and
Ito~\cite{ArnouxIto2001}. This construction is illustrated in
Figure~\ref{fig:intro-expanded-rauzy-132} for the example
$g=\left(1,3,2\right)$.

\begin{figure}[htbp]
\centering
\includegraphics[width=0.94\textwidth]{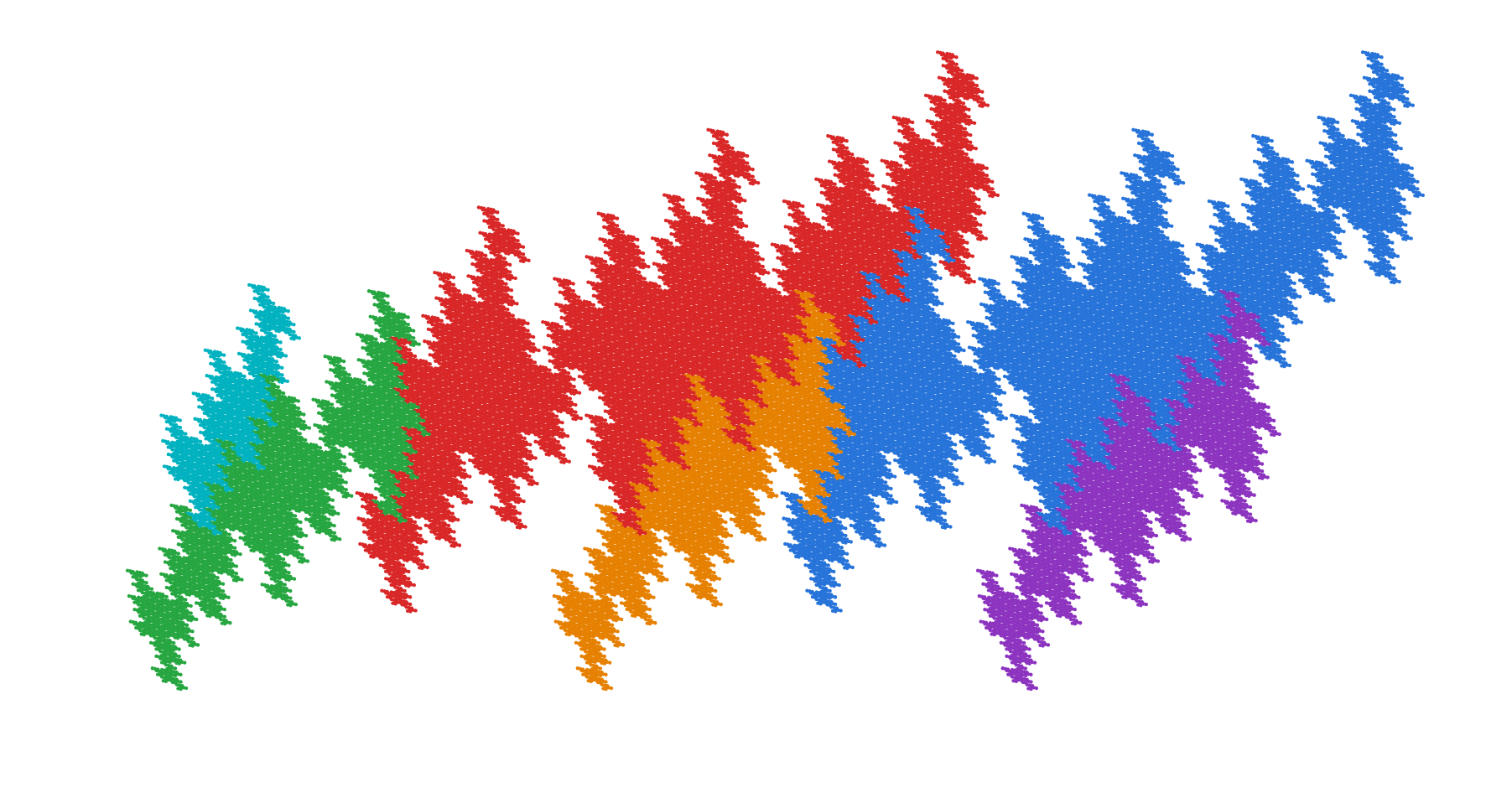}
\caption{A numerical approximation of the union of the Rauzy pieces for
$g=\left(1,3,2\right)$, partitioned according to the terminal block.
The colors red, blue, green, orange, purple, and cyan correspond,
respectively, to the semi-proper blocks $\left(0\right)$,
$\left(1\right)$, $\left(0,2\right)$, $\left(1,2\right)$,
$\left(2,2\right)$, and $\left(0,3,2\right)$.}
\label{fig:intro-expanded-rauzy-132}
\end{figure}

Theorems~\ref{thm:intro-main} and
\ref{thm:intro-several-blocks} have analytic conclusions, but their main
new input is the arithmetic coding of canonical terminal blocks by the
companion substitution, together with its Rauzy-geometric realization. The
proof proceeds by replacing the terminal-block condition successively by
descriptions adapted to substitution dynamics, Rauzy geometry, and
finally prime equidistribution:
$$
\begin{aligned}
\TB(n)=b_{i,j}
&\Longrightarrow
\text{a symbolic prefix condition}\\
&\Longrightarrow
\text{membership in a translated Rauzy piece}\\
&\Longrightarrow
\text{a linear orbit condition on a compact torus}.
\end{aligned}
$$

The symbolic step is the main new coding ingredient. After expanding the
fixed word of the associated substitution into semi-proper blocks, we prove
that the resulting infinite word lists
\[
\TB(0),\TB(1),\TB(2),\ldots
\]
in their natural integer order. For example, when
$g=\left(1,3,2\right)$, the fixed word of $\sigma$ begins with
$$
A_1A_1A_2A_1A_1A_2A_1A_1A_1A_3\cdots.
$$
The letters $A_1,A_2,A_3$ correspond respectively to
$$
b_{1,0}=\left(0\right),\qquad
b_{2,0}=\left(0,2\right),\qquad
b_{3,0}=\left(0,3,2\right).
$$
We expand them according to
$$
\begin{aligned}
A_1&\longmapsto \left(0\right)\left(1\right),\\
A_2&\longmapsto
\left(0,2\right)\left(1,2\right)\left(2,2\right),\\
A_3&\longmapsto \left(0,3,2\right).
\end{aligned}
$$
The expanded fixed word therefore begins
$$
\left(0\right)\left(1\right)
\left(0\right)\left(1\right)
\left(0,2\right)\left(1,2\right)\left(2,2\right)\left(0\right)\left(1\right)
\left(0\right)\left(1\right)
\left(0,2\right)\left(1,2\right)\left(2,2\right)\cdots,
$$
and these are precisely the terminal blocks of the consecutive integers
$n\geqslant0$.

The classes with terminal block $b_{i,0}$ occur at the boundaries of
the complete expanded images of the letters $A_i$. Their arithmetic positions
are therefore encoded by prefix-counting vectors of the unexpanded fixed word,
and the remaining terminal classes are obtained from these base classes by
explicit translations. Projecting the corresponding vectors onto the
non-dominant eigenspaces produces the Rauzy pieces associated with the
terminal classes. In Figure~\ref{fig:intro-expanded-rauzy-132}, the red and
blue pieces differ by one unit translation in the stable plane, while the
green, orange, and purple pieces are successive unit translates.

The multiplicity-one fundamental-domain property, proved in
Section~\ref{sec:rauzy-verification} from the algebraic hypotheses,
allows us to quotient by the natural difference lattice and turn
membership in such a piece into a linear orbit condition on a compact
torus. Rational independence and Vinogradov's theorem then give
equidistribution along the primes, while the volumes of the Rauzy pieces
determine the asymptotic constants.

A main structural contribution of the paper is the direct arithmetic
interpretation of the substitution fixed word. Rather than coding
beta-expansions of real numbers, we interpret its block expansion through
the canonical infinite block expansions of non-negative integers, whose
terminal blocks it records in order. Iterating the same affine coding yields the joint
distribution of any fixed number of consecutive terminal blocks. This
interpretation remains valid when the canonical infinite $H$-expansion is
non-greedy, as in the system $g=\left(1,3,2\right)$.

Further developments of this geometric framework for Pisot substitutions
and numeration systems include the self-affine and beta-numeration tiling
work of Akiyama, Berth\'e--Siegel, Canterini--Siegel, and
Siegel--Thuswaldner
\cite{Akiyama1999SelfAffine,BertheSiegel2005,CanteriniSiegel2001,SiegelThuswaldner2009}.
Very recently, Carton, Sudbery, and Yassawi constructed topological groups
associated with Pisot numerations satisfying a zero-preservation condition
and showed that, in the unimodular case, the resulting group is
continuously isomorphic to a torus \cite{CartonSudberyYassawi2026}.
Their setting is based on normalized greedy expansions with standard
initial conditions, whereas the present paper uses the canonical block
expansions above, which need not be greedy.

For the present companion family, Barge's theorem supplies the geometric
input after reversal of the substitution words. Reversal changes
\[
A_i\mapsto A_1^{h_i}A_{i+1},\qquad A_L\mapsto A_1,
\]
into a substitution that is injective on initial letters and constant on
final letters. Barge's theorem
\cite[Theorem~3.13]{Barge2016} therefore gives pure discrete spectrum under
Assumption~2.8. Under the same assumption, our substitution is unimodular
and irreducible and, by Lemma~\ref{lem:rauzy-hypotheses}, satisfies strong
coincidence. The standard Rauzy-fractal equivalence
\cite[Theorem~4.8]{BaratBertheLiardetThuswaldner2006} then identifies pure
discreteness with the multiplicity-one lattice-tiling property.
Section~\ref{sec:rauzy-verification} checks this identification in our stable
coordinates and with our lattice $\mathcal L$.

The monotone and dominant conditions, together with the two explicit families in Theorem~\ref{thm:intro-unconditional-cases}, therefore provide concrete situations in which the algebraic Assumption~\ref{ass:pisot} can be verified directly. The first family contains the recurring example $g=(1,3,2)$, while the second provides Pisot units failing the finiteness property~(F).

The organization of the paper follows this strategy.
Section~\ref{sec:general-data} develops the block language, infinite and
reduced expansions, substitution, and stable coordinates.
Section~\ref{sec:coding} establishes the symbolic coding of terminal
blocks by projected prefix vectors.
Sections~\ref{sec:rauzy} and~\ref{sec:recombination} construct the Rauzy
pieces and recover all terminal classes.
Section~\ref{sec:volume} builds the arithmetic toral model and proves
the sampling principle.
Section~\ref{sec:prime-blocks} proves prime genericity for the
substitution fixed word and the prime terminal-block theorem using
Theorem~\ref{thm:rauzy-fundamental-domain}, and records selected direct
consequences and transfers from established prime theorems.
Section~\ref{sec:rauzy-verification} applies Barge's theorem and the
standard Rauzy-fractal equivalence to prove
Theorem~\ref{thm:rauzy-fundamental-domain} under
Assumption~\ref{ass:pisot}; it then records elementary coefficient
conditions and two explicit infinite families for which that algebraic
assumption is automatic. The first contains $g=(1,3,2)$ and is non-$\beta$ for $L\geqslant3$, while the second consists of Pisot units failing the finiteness property~(F).
Finally, Section~\ref{sec:conclusion} discusses quantitative extensions and some
non-unimodular and reducible cases beyond the present Pisot-unit setting.

%% file: sections/02-general-data.tex
\section{Numeration, substitution, and stable coordinates}\label{sec:general-data}

Section~\ref{sec:introduction} introduced the recurrence data and the
associated substitution used throughout the paper. In this section, we
develop the block language used to represent non-negative integers and
obtain their canonical infinite expansions, finite reduced representatives,
and terminal blocks. We
then represent both the arithmetic and substitutional data in a common
contracting space. This provides the common language needed before the
two descriptions can be related.

\subsection{The block language}

We introduce a block language for the recurrence data of
Definition~\ref{def:intro-recurrence-data}, designed to encode
non-negative integers at each fixed digit-length. We first fix the order
in which coefficient words are read. A finite
\emph{coefficient word} of non-negative integers
$
a=\left(a_1,a_2,\ldots,a_m\right)
$
is read from left to right in increasing order of the weights: $a_1$ is
the coefficient of $H_1$, $a_2$ is the coefficient of $H_2$, and so on.
We define its \emph{$H$-value} by
$$
\evalH{a}{H}:=\sum_{k=1}^{m}a_kH_k.
$$
The empty word has $H$-value $0$. At this point, coefficient words are unrestricted, and distinct
coefficient words may have the same $H$-value.

Next, we group individual digits into finite blocks adapted to the coefficients
of the recurrence. For $1\leqslant i\leqslant L$ and
$j\in\{0,\ldots,h_i-1\}$, set
$$
b_{i,j}:=\left(j,h_{i-1},h_{i-2},\ldots,h_1\right),
$$
with the convention that $b_{1,j}=(j)$. Adapting the block terminology of \cite{Chang2021}, we call $b_{i,j}$ a
\emph{semi-proper block}. Its \emph{digit-length} is $i$, and we write
$\len(b_{i,j})=i$. Its first digit $j$ may vary, whereas its remaining
$i-1$ digits are prescribed by $h_{i-1},\ldots,h_1$. There are exactly $h_i$ semi-proper blocks of digit-length $i$. Since
$h_L=1$, the unique semi-proper block of digit-length $L$ is
$$
b_{L,0}
=\left(0,h_{L-1},h_{L-2},\ldots,h_1\right).
$$
We put
$
\rho_i:=h_i-1.
$
Thus the semi-proper blocks of digit-length $i$ are
$b_{i,0},\ldots,b_{i,\rho_i}$. If $h_i=1$, then $b_{i,0}$ is the unique
semi-proper block of digit-length $i$.

For $m\geqslant0$, let $\mathcal W_m$ be the set of all words of the
form
$$
W=b_{i_1,j_1}b_{i_2,j_2}\cdots b_{i_t,j_t},
$$
where $t\geqslant0$, the indices satisfy
$1\leqslant i_r\leqslant L$ and
$0\leqslant j_r\leqslant h_{i_r}-1$ for every
$1\leqslant r\leqslant t$, and the total digit-length is
$$
i_1+i_2+\cdots+i_t=m.
$$
Thus $t$ is the number of blocks in $W$, whereas $m$ is the total
number of digits contained in these blocks. We call the elements of
$\mathcal W_m$ \emph{semi-proper block words}. When $t=0$, the
concatenation is the empty word and the sum is understood to be $0$:
hence the empty word is the unique element of $\mathcal W_0$.

Every word in $\mathcal W_m$ has a unique decomposition into
semi-proper blocks. Indeed, its block on the high-weight side is uniquely
determined by its final digits, and removing that block determines the
preceding block recursively. Thus an element of $\mathcal W_m$ may be
identified unambiguously either with its digit word or with its ordered
sequence of semi-proper blocks.

Let $W\in\mathcal W_m$, and let $a_1,\ldots,a_m$ be its digits, read
from the small-weight end. Then
$$
\evalH{W}{H}=\sum_{k=1}^{m}a_kH_k.
$$

The positivity of the recurrence coefficients implies that
$(H_k)_{k\geqslant1}$ is nondecreasing and unbounded. Hence, as $m$
ranges over the non-negative integers, the finite sets
$$
\left\{0,1,\ldots,H_{m+1}-1\right\}
$$
exhaust the non-negative integers.

The finite-level bijection will follow from a more fundamental order
property of the global block language. We therefore pass to canonical
infinite products of semi-proper blocks.

\subsection{Infinite expansions, reduced representatives, and terminal blocks}
\label{subsec:reduced-expansions}

The finite sets $\mathcal W_m$ are useful for fixed digit-lengths, but
terminal blocks are most naturally read from an infinite block expansion.
Let $\mathcal B$ denote the set of all semi-proper blocks, and let
$\mathbf W$ be the set of all infinite block words
\[
c=c_1:c_2:c_3:\cdots,
\qquad c_k\in\mathcal B,
\]
subject only to the condition that
$c_k=b_{1,0}$ for all sufficiently large $k$. If $c\in\mathbf W$,
let $\operatorname{dgt}(c)=(a_k)_{k\geqslant1}$ denote the infinite digit
sequence obtained by concatenating its blocks. We then define
\[
\mathcal W:=\{\operatorname{dgt}(c):c\in\mathbf W\}
\subset\N_0^{\mathbb N}.
\]
Let
\[
z:=b_{1,0}^{\infty}\in\mathbf W
\qquad\text{and}\qquad
Z:=\operatorname{dgt}(z)=(0,0,\ldots)\in\mathcal W.
\]
For $W=(a_k)_{k\geqslant1}\in\mathcal W$, set
\[
\evalH{W}{H}:=\sum_{k\geqslant1}a_kH_k.
\]
The sum is finite because every block word in $\mathbf W$ is eventually
$b_{1,0}$.

We also define
\[
\operatorname{ord}(W)
:=
\begin{cases}
\max\{k\geqslant1:a_k>0\}&\textrm{ if }W\neq Z\\
0&\textrm{ if }W=Z.
\end{cases}
\]
If $c\in\mathbf W$, put
$\operatorname{ord}(c):=\operatorname{ord}(\operatorname{dgt}(c))$.
Thus, if $c\neq z$, then $\operatorname{ord}(c)$ records the highest
nonzero digit position of $c$.

\begin{example}
Consider again the order-two system $g=(1,2)$, for which $h=(2,1)$.
Its semi-proper blocks are
\[
\mathcal B=\{(0),(1),(0,2)\}.
\]
Thus an element of $\mathbf W$ is obtained by choosing any sequence of
these blocks, subject only to the condition that all sufficiently far
blocks are $b_{1,0}=(0)$. For instance,
\[
c=b_{1,1}:b_{2,0}:z
=(1):(0,2):(0):(0):\cdots
\]
belongs to $\mathbf W$, and
\[
\dgt(c)=(1,0,2,0,0,\ldots)\in\mathcal W.
\]
In this way, $\mathcal W$ consists of the infinite digit sequences
obtained by concatenating the blocks of words in $\mathbf W$ and
forgetting the block boundaries. Moreover,
\[
\ord(c)=\ord(\dgt(c))=3,
\]
and, since $H_1=1$, $H_2=2$, and $H_3=5$,
\[
[\dgt(c)]_H=H_1+2H_3=11.
\]
\end{example}

The preceding example shows how an infinite block word determines both
its digit sequence and its finite $H$-value. We record the converse
structural fact: every element of $\mathcal W$ comes from a unique
infinite block word, and its nonzero part has a unique finite block
representative.

\begin{lemma}\label{lem:infinite-block-decomposition}
Let $W\in\mathcal W$ and put $r:=\operatorname{ord}(W)$. If $r>0$,
there is a unique finite semi-proper block word
\[
C=c_1:\cdots:c_t
\]
of total digit-length $r$ such that
\[
W=\operatorname{dgt}(C:z).
\]
Its high-weight block $c_t$ is not $b_{1,0}$. If $r=0$, then $z$ is the
unique element of $\mathbf W$ whose digit sequence is $W=Z$.
\end{lemma}

\begin{proof}
Existence follows from the definition of $\mathcal W$: choose
$c\in\mathbf W$ with $\operatorname{dgt}(c)=W$ and remove all trailing
copies of $b_{1,0}$. If $W\neq Z$, the last remaining block cannot be
$b_{1,0}$. Every semi-proper block of digit-length at least $2$ ends
with the positive digit $h_1$, while a block of digit-length $1$ is
nonzero unless it is $b_{1,0}$. Hence the last digit of the remaining
finite block word is nonzero, so its total digit-length is exactly
$r$.

For uniqueness, the resulting digit word has length $r$, and by the
unique finite block decomposition established in the preceding
subsection it has a unique decomposition into semi-proper blocks. If
$W=Z$, every block must have zero digit word, and the only such
semi-proper block is $b_{1,0}$; hence the unique infinite block word is
$z$.
\end{proof}

Lemma~\ref{lem:infinite-block-decomposition} shows that the digit map is
bijective from $\mathbf W$ onto $\mathcal W$. We therefore write
\[
\operatorname{dgt}^{-1}(W)\in\mathbf W
\qquad (W\in\mathcal W)
\]
for the unique infinite block word with digit sequence $W$.

We will also use finite reduced representatives. If
$W=(a_k)_{k\geqslant1}\in\mathcal W$ and
$r:=\operatorname{ord}(W)>0$, set
\[
\operatorname{red}(W):=(a_1,\ldots,a_r),
\]
and set $\operatorname{red}(Z):=(0)$. If
$c=C:z\in\mathbf W$ with $C$ finite and with high-weight block different
from $b_{1,0}$, set $\operatorname{red}(c):=C$, and set
$\operatorname{red}(z):=\varepsilon$. Thus $\operatorname{red}(W)$ is a finite digit word, whereas
$\operatorname{red}(c)$ is a finite block word; the type of the argument
will make clear which reduction is being used. In particular, the zero
digit sequence has reduced representative $(0)$, while the corresponding
reduced block word is empty.

We will also use the following order on infinite digit sequences. For two distinct elements
$W=(a_k)_{k\geqslant1}$ and $W'=(a'_k)_{k\geqslant1}$ of $\mathcal W$,
write $W\prec W'$ if, at the largest index $r$ for which
$a_r\neq a'_r$, one has $a_r<a'_r$. This is the ascending
lexicographical order used in \cite{Chang2021}.

For $m\geqslant1$, let
\begin{equation}\label{def:U}
U_m:=(0,\ldots,0,1,0,0,\ldots)\in\N_0^{\mathbb N}
\end{equation}
denote the infinite coefficient sequence whose only nonzero entry is
the $m$th one. For every $m\geqslant2$, one has $U_m\in\mathcal W$:
if $h_1>1$, then
\[
\operatorname{dgt}^{-1}(U_m)
=b_{1,0}^{m-1}:b_{1,1}:z,
\]
whereas if $h_1=1$, then
\[
\operatorname{dgt}^{-1}(U_m)
=b_{1,0}^{m-2}:b_{2,0}:z.
\]
The sequences $U_m$ will serve as convenient comparison points for the
order $\prec$.

\begin{theorem}\label{thm:block-language}
The $H$-value map
\[
W\longmapsto\evalH{W}{H}
\]
is a strictly increasing bijection from $\mathcal W$, ordered by
$\prec$, onto $\N_0$ with its usual order. Moreover, if
$W\neq Z$, $n=\evalH{W}{H}$, and
$s:=\operatorname{ord}(W)$, then
\[
H_s\leqslant n<H_{s+1}.
\]
Conversely, if $n\geqslant1$ and $s\geqslant1$ satisfy
\[
H_s\leqslant n<H_{s+1},
\]
then the unique $W\in\mathcal W$ satisfying
$\evalH{W}{H}=n$ has $\operatorname{ord}(W)=s$.
\end{theorem}

\begin{proof}
The first statement is an instance of the general Zeckendorf-collection
formalism of \cite[Theorem~16 and Lemma~37]{Chang2021}. Since our
normalization of the initial weights is tailored to the present block
language, we verify the required order property directly.

We first prove by induction on $m$ that the $H$-value map is a strictly
increasing bijection
\[
\mathcal W_m\longrightarrow\{0,1,\ldots,H_{m+1}-1\},
\]
where $\mathcal W_m$ is ordered by the restriction of $\prec$.
The claim is immediate for $m=0$. Suppose that it holds at every
smaller digit-length. If the high-weight block of
$W\in\mathcal W_m$ is $b_{i,j}$, then
\[
W=V:b_{i,j},
\qquad
V\in\mathcal W_{m-i},
\]
and
\[
\evalH{W}{H}
=
\evalH{V}{H}+C_{m,i,j},
\]
where
\[
C_{m,i,j}
:=
jH_{m-i+1}
+\sum_{r=1}^{i-1}h_rH_{m+1-r}.
\]
When the blocks are read from the high-weight end, one first encounters
the fixed digits $h_1,\ldots,h_{i-1}$ and then the digit
$j<h_i$. Hence the possible high-weight blocks occur in the order
\[
b_{1,0},\ldots,b_{1,h_1-1},
b_{2,0},\ldots,b_{2,h_2-1},\ldots,
\]
with $i\leqslant\min\{L,m\}$, and within each such group the order is
the order of $V$.

By the induction hypothesis, the group with fixed high-weight block
$b_{i,j}$ is therefore mapped increasingly onto the consecutive
interval
\[
\{C_{m,i,j},\ldots,C_{m,i,j}+H_{m-i+1}-1\}.
\]
Consecutive intervals with the same $i$ are adjacent, and the last
interval of length $i$ is adjacent to the first interval of length
$i+1$, since
\[
C_{m,i,h_i-1}+H_{m-i+1}=C_{m,i+1,0}.
\]
The first interval starts at $0$. The right endpoint of the final
interval is $H_{m+1}-1$: if $m<L$, this is exactly the defining
initial relation for $H_{m+1}$, while if $m\geqslant L$, it is the
order-$L$ recurrence together with $h_L=1$. This proves the
finite-level claim.

Now let $W,W'\in\mathcal W$. After appending sufficiently many
copies of $b_{1,0}$ to their finite nonzero parts, both may be viewed
as elements of the same $\mathcal W_m$. The finite-level result then
shows that
\[
W\prec W'
\quad\Longrightarrow\quad
\evalH{W}{H}<\evalH{W'}{H},
\]
so the global $H$-value map is strictly increasing and injective.
Conversely, given $n\in\N_0$, choose $m$ with $n<H_{m+1}$.
The finite-level bijection gives a unique $V\in\mathcal W_m$ of
$H$-value $n$, and $\operatorname{dgt}(V:z)\in\mathcal W$ has the
same value. Thus the global map is surjective.

Let $W=(a_k)_{k\geqslant1}\in\mathcal W\setminus\{Z\}$, put
$s:=\operatorname{ord}(W)$ and $n:=\evalH{W}{H}$. Since $a_s>0$,
\(
n\geqslant H_s.
\)
All coefficients above position $s$ vanish, so $W\prec U_{s+1}$.
Since $U_{s+1}\in\mathcal W$ and
$\evalH{U_{s+1}}{H}=H_{s+1}$, monotonicity gives
\(
n<H_{s+1}.
\)

Conversely, suppose that $H_s\leqslant n<H_{s+1}$. By the first
statement, let $W\in\mathcal W$ be the unique word satisfying
$\evalH{W}{H}=n$, and put $m:=\operatorname{ord}(W)$. The first part
of the proof gives
\[
H_m\leqslant n<H_{m+1}.
\]
We have $H_1=1$, $H_2=h_1\geqslant1$, and, by positivity of the
recurrence coefficients,
\[
H_k<H_{k+1}\qquad (k\geqslant2).
\]
Thus the only possible equality between consecutive weights is
$H_1=H_2=1$. Consequently, the intervals
\[
[H_k,H_{k+1}),\qquad k\geqslant1,
\]
are pairwise disjoint and cover $[1,+\infty)$; the only one that can
be empty is the first, when $H_1=H_2=1$. Since $n$ belongs to both
the $m$th and the $s$th intervals, neither of these two intervals is
empty, and hence $m=s$.
\end{proof}

By Theorem~\ref{thm:block-language}, for every $n\in\N_0$ there
is a unique $W\in\mathcal W$ such that $\evalH{W}{H}=n$. We call this
sequence the \emph{canonical infinite $H$-expansion} of $n$ and write
\begin{equation}\label{def:inverse-expansion}
[n]_H^{-1}:=W
\qquad\text{when}\qquad
\evalH{W}{H}=n.
\end{equation}
The map $n\mapsto[n]_H^{-1}$ is a strictly increasing bijection from
$\N_0$ onto $\mathcal W$. Its finite digit representative
$\operatorname{red}([n]_H^{-1})$ is the reduced $H$-expansion, while
$\operatorname{dgt}^{-1}([n]_H^{-1})$ is its canonical infinite
semi-proper block decomposition.

The finite-level claim proved above gives, for every $m\geqslant0$,
a bijection from $\mathcal W_m$ onto
\[
\{0,1,\ldots,H_{m+1}-1\}.
\]
In particular,
\[
\#\mathcal W_m=H_{m+1},
\]
and consequently
\[
\sum_{m\geqslant0}\#\mathcal W_m x^m
=
\sum_{m\geqslant0}H_{m+1}x^m
=
\frac{1}{1-h_1x-h_2x^2-\cdots-h_Lx^L}.
\]

We illustrate both the finite-level bijection and the reduced
representatives in the order-two case $g=(1,2)$.

\begin{example}
For the order-two system $g=(1,2)$ introduced in
Example~\ref{ex:intro-order-two}, at digit-length $2$ the five block
words are
\[
b_{1,0}\,b_{1,0},\quad
b_{1,1}\,b_{1,0},\quad
b_{1,0}\,b_{1,1},\quad
b_{1,1}\,b_{1,1},\quad
b_{2,0},
\]
and their respective $H$-values are $0$, $1$, $2$, $3$, and $4$.
Thus the $H$-value map is a bijection from $\mathcal W_2$ onto
$\{0,1,2,3,4\}$.

Moreover,
\[
\operatorname{red}([4]_H^{-1})=(0,2),
\qquad
\operatorname{red}\!\left(\operatorname{dgt}^{-1}([4]_H^{-1})\right)
=b_{2,0}.
\]
The infinite block expansion of $4$ is $b_{2,0}:z$, and its finite
block representatives of digit-length at least $2$ are
\[
b_{2,0}b_{1,0}^q,\qquad q\geqslant0.
\]
\end{example}

We introduce notation for the canonical block expansion of an
arbitrary integer and for its reduced block representative.

\begin{definition}\label{def:reduced-expansion-terminal-block}
For $n\in\N_0$, define the infinite block expansion
\[
C_n:=\operatorname{dgt}^{-1}([n]_H^{-1})\in\mathbf W
\]
and its finite reduced block word
\[
W_n:=\operatorname{red}(C_n).
\]
Thus $W_0=\varepsilon$, whereas $C_0=z=b_{1,0}^{\infty}$.

The first block of $C_n$ is called the \emph{terminal block} of $n$
and is denoted by $\TB(n)$. For $n>0$, this is also the first block of
$W_n$, while $\TB(0)=b_{1,0}$.
\end{definition}

\begin{definition}\label{def:terminal-block-words}
For $e\geqslant1$, let $\mathcal B_e$ be the finite set of
concatenations of exactly $e$ semi-proper blocks, and put
$\mathcal B_0:=\{\varepsilon\}$. We define $\TBe{e}(n)$ to be the
concatenation of the first $e$ blocks of $C_n$, and set
$\TBe{0}(n):=\varepsilon$. Thus $\TBe{e}(n)$ is defined for every
$n\in\N_0$ and every $e\geqslant0$; in particular,
\[
\TBe{e}(0)=b_{1,0}^e.
\]
We call $\TBe{e}(n)$ the \emph{$e$-terminal block word} of $n$.
\end{definition}

For $e\geqslant1$ and
$B=b_{i_1,j_1}\cdots b_{i_e,j_e}\in\mathcal B_e$, we extend the
digit-length function $\len$ additively to terminal block words by
\[
\len(B):=\sum_{r=1}^e\len(b_{i_r,j_r})
=i_1+\cdots+i_e,
\]
and set $\len(\varepsilon)=0$. We call $\len(B)$ the
\emph{total digit-length} of $B$.

\begin{example}
For the order-two system $g=(1,2)$ of
Example~\ref{ex:intro-order-two}, the first reduced block words are
\[
W_1=b_{1,1},\qquad
W_2=b_{1,0}\,b_{1,1},\qquad
W_3=b_{1,1}\,b_{1,1},\qquad
W_4=b_{2,0}.
\]
Their terminal blocks are
\[
\TB(1)=b_{1,1},\qquad
\TB(2)=b_{1,0},\qquad
\TB(3)=b_{1,1},\qquad
\TB(4)=b_{2,0}.
\]
Since $C_2=b_{1,0}:b_{1,1}:z$, one has
\[
\TBe{2}(2)=b_{1,0}b_{1,1},
\qquad
\len(\TBe{2}(2))=2.
\]
Also $C_0=z$, so, for example,
$\TBe{3}(0)=b_{1,0}^3$.
\end{example}

Before turning to substitutions, we note that the positivity assumption
on $g_2,\ldots,g_L$ is not needed for the block language itself. If
$h_i=0$ for some $i>1$, there is simply no semi-proper block of
digit-length $i$, and the finite-level induction above adapts verbatim
after omitting the corresponding empty block groups.

\subsection{Substitution and stable coordinates}

We turn from the block language to the substitution associated with
the recurrence data. The symbolic construction below requires no
assumptions beyond those already imposed on the recurrence data. The
stable geometric realization, however, requires an additional algebraic
hypothesis.

Let $\sigma$ be the substitution on $\mathcal A$ and let $A_\infty$
be its fixed word, as defined in
Definition~\ref{def:intro-associated-substitution}. For a finite word
$W$, let $\lvert W\rvert_{A_j}$ denote the number of occurrences of
$A_j$ in $W$. We use the row-incidence matrix
\[
M_{i,j}:=\lvert\sigma(A_i)\rvert_{A_j}.
\]
Thus
\[
M=
\begin{pmatrix}
h_1&1&0&\cdots&0\\
h_2&0&1&\cdots&0\\
\vdots&\vdots&\ddots&\ddots&\vdots\\
h_{L-1}&0&0&\cdots&1\\
h_L&0&0&\cdots&0
\end{pmatrix}.
\]
This convention is the transpose of the usual column-incidence
convention for substitutions.

Under our standing positivity assumption, the substitution $\sigma$ is
primitive. Indeed, every image $\sigma(A_i)$ contains $A_1$, while
$\sigma^{L-1}(A_1)$ contains every letter of $\mathcal A$. Consequently,
every letter occurs in $\sigma^L(A_i)$ for every
$1\leqslant i\leqslant L$. Moreover,
\[
\lvert\det M\rvert=h_L=1,
\]
so $M$ is unimodular, and a direct computation gives
\[
\det(XI-M)=P_h(X).
\]

Up to this point, the construction has depended only on the recurrence
data. To obtain contracting coordinates, we now impose the following
algebraic assumption.

\begin{assumption}\label{ass:pisot}
The dominant root $\varphi$ is a Pisot number and $P_h$ is its minimal
polynomial over $\mathbb Q$.
\end{assumption}

Since $P_h$ is the minimal polynomial of $\varphi$, it is irreducible
and $\varphi$ has degree $L$. Let
$\gamma_1,\ldots,\gamma_{L-1}$ be the remaining roots of $P_h$; these
are precisely the other algebraic conjugates of $\varphi$. Since
$\varphi$ is Pisot, we have
\[
\abs{\gamma_j}<1
\qquad \left(1\leqslant j\leqslant L-1\right).
\]
Moreover, $h_L=1$ gives $P_h(0)=-1$, so $\varphi$ is an algebraic
unit. The matrix $M$ therefore has the expanding eigenvalue $\varphi$
and the contracting eigenvalues
$\gamma_1,\ldots,\gamma_{L-1}$. In particular, $\sigma$ is a primitive
unimodular irreducible Pisot substitution.

Set $\omega:=\varphi^{-1}$. Dividing $P_h(\varphi)=0$ by
$\varphi^L$ gives
\[
1=h_1\omega+h_2\omega^2+\cdots+h_L\omega^L.
\]
We will use this form of the characteristic identity throughout the
paper.

Next, we construct the stable coordinates associated with the contracting
eigenvalues and relate them to the letter-count vectors of finite
prefixes of $A_\infty$. Define
\[
T_1:=A_1,\qquad T_k:=\sigma(T_{k-1})\quad\text{for }k\geqslant2.
\]
Equivalently, $T_k=\sigma^{k-1}(A_1)$, and the words $T_k$ form a nested
sequence of prefixes of $A_\infty$. For a finite word $W$, let
\[
\ell(W):=
\begin{pmatrix}
\lvert W\rvert_{A_1}\\
\vdots\\
\lvert W\rvert_{A_L}
\end{pmatrix}
\in\mathbb Z^L
\]
be its column vector of letter counts. Let $e_1,\ldots,e_L$ be the
standard basis of $\R^L$. With the row-incidence convention used above,
one has
\[
\ell(\sigma(W))=M^{\mathsf T}\ell(W)
\]
for every finite word $W$. Therefore
\[
\ell(T_k)=\left(M^{\mathsf T}\right)^{k-1}e_1.
\]

To pass to contracting coordinates, consider the spectrum of
$M^{\mathsf T}$. The eigenvalue $\varphi$ determines the expanding
direction, while the conjugate eigenvalues
$\gamma_1,\ldots,\gamma_{L-1}$ determine the contracting directions.
Let $a$ be the number of real contracting conjugates and $b$ the number
of conjugate pairs of non-real contracting conjugates, so that
$a+2b=L-1$. We introduce the \emph{contracting coordinate space}
\[
E_s:=\R^a\times\C^b,
\]
viewed as a real vector space of dimension $L-1$. Let
\[
\Gamma=\left(\eta_1,\ldots,\eta_{a+b}\right)
\]
consist of all real contracting conjugates and one representative of
each non-real conjugate pair. For each
$j\in\{1,\ldots,a+b\}$, choose a right eigenvector
$w_j\in\C^L$ of $M$ associated with $\eta_j$, chosen real when
$\eta_j$ is real, and normalize it so that its last coordinate is $1$.
For this companion matrix the last coordinate of a nonzero eigenvector
cannot vanish, so this normalization is well defined. Thus
$Mw_j=\eta_jw_j$. Taking the last coordinate in this identity gives
$(w_j)_1=\eta_j$.

We use the following complex-bilinear pairing on $\C^L$:
\[
u\cdot v:=u^{\mathsf T}v.
\]
Each $w_j$ therefore defines a scalar coordinate functional
$v\mapsto w_j\cdot v$ on $\R^L$. Collecting these coordinates, we
define the \emph{real-linear stable projection}
\[
\begin{array}{rcll}
\pi_s: & \R^L & \longrightarrow & E_s\\
       & v    & \longmapsto     & \left(w_j\cdot v\right)_{1\leqslant j\leqslant a+b}.
\end{array}
\]

Figure~\ref{fig:stable-projection-132} illustrates this construction
for the recurring three-letter example $g=(1,3,2)$. The ambient space is
$\R^3$, with coordinate directions $e_1,e_2,e_3$ corresponding to the
three letters. Starting from the origin, a word is represented by a
stepped path: reading the letter $A_i$ produces a unit step in the
direction $e_i$. Consequently, after reading a finite prefix $W$, the
corresponding vertex of the path is precisely its letter-count vector
$\ell(W)$.

In this three-dimensional situation, the expanding eigenspace is a
line, while the real contracting invariant subspace is a plane. The
coordinate space $E_s=\C$ is another real two-dimensional realization of
the same stable quotient. In Figure~\ref{fig:stable-projection-132}, we identify $E_s$
with the displayed orange stable plane by a fixed real-linear
isomorphism. The blue stepped path records prefix letter-count vectors,
the red lines indicate projection along the expanding direction onto
that plane, and the black points represent the resulting stable
coordinates. Thus, under this identification, the black point attached
to a prefix $W$ represents $\pi_s(\ell(W))$. The left panel gives a
local view of several such projections, while the right panel shows the
same geometry on a larger scale.

\begin{figure}[!ht]
\centering
\includegraphics[width=0.98\textwidth,trim=48 0 0 0,clip]{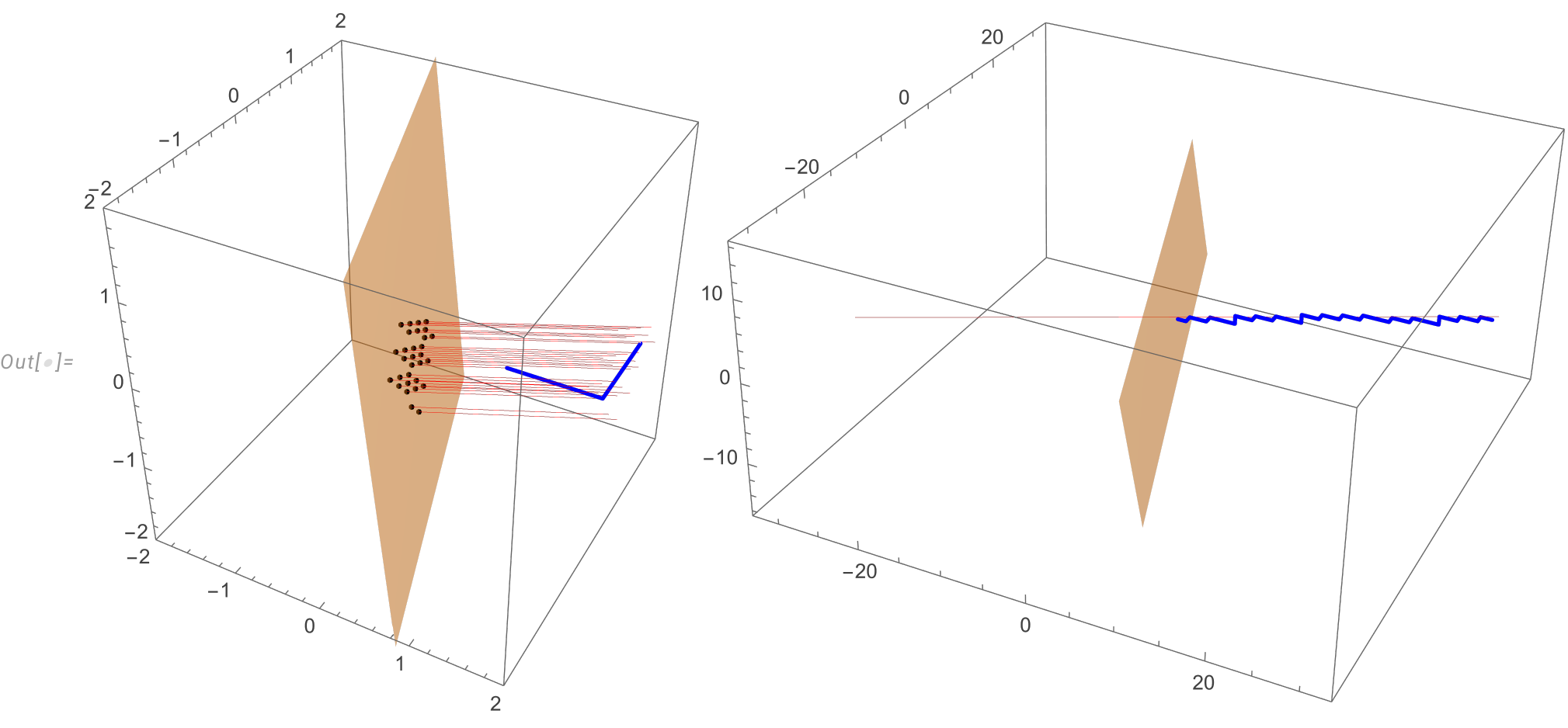}
\caption{Stable projection for $g=(1,3,2)$. The orange plane is the
real-plane realization of $E_s$ fixed above. The blue stepped path
records prefix letter-count vectors; red lines indicate projection
along the expanding direction onto this plane, and black points
represent the corresponding stable coordinates. The left panel is a
local view and the right panel shows the geometry on a larger scale.}
\label{fig:stable-projection-132}
\end{figure}

For each $m\geqslant1$, the vector $\ell(T_m)\in\R^L$ records the
numbers of occurrences of the letters in $T_m$, and
$\pi_s(\ell(T_m))\in E_s$ gives its contracting coordinates. For every
$j\in\{1,\ldots,a+b\}$,
\[
\begin{aligned}
w_j\cdot\ell(T_m)
&=w_j^{\mathsf T}\left(M^{\mathsf T}\right)^{m-1}e_1\\
&=\left(M^{m-1}w_j\right)\cdot e_1\\
&=\eta_j^{m-1}\left(w_j\cdot e_1\right)\\
&=\eta_j^m.
\end{aligned}
\]
Hence
\[
\pi_s\left(\ell(T_m)\right)
=
\left(\eta_j^m\right)_{1\leqslant j\leqslant a+b}.
\]

For the recurring example $g=(1,3,2)$, one has $E_s=\C$, and there is
only one contracting conjugate $\eta_1$ to consider. The first three
iterates have letter-count vectors
\[
\ell(T_1)=
\begin{pmatrix}1\\0\\0\end{pmatrix},
\qquad
\ell(T_2)=
\begin{pmatrix}2\\1\\0\end{pmatrix},
\qquad
\ell(T_3)=
\begin{pmatrix}7\\2\\1\end{pmatrix}.
\]
Their stable projections are therefore
\[
\pi_s\left(\ell(T_1)\right)=\eta_1,
\qquad
\pi_s\left(\ell(T_2)\right)=\eta_1^2,
\qquad
\pi_s\left(\ell(T_3)\right)=\eta_1^3
\qquad\text{in }\C.
\]

The preceding computation identifies the contracting vector attached
to each digit position. Set
\[
s_0:=\pi_s(e_L)=(1,\ldots,1)\in E_s,
\]
where the second equality follows from the normalization of the
eigenvectors $w_j$. Thus the first digit position is represented by
$s_0$, while, for $k\geqslant2$, the $k$-th digit position is
represented by
\[
\pi_s\left(\ell(T_{k-1})\right)
=
\left(\eta_j^{k-1}\right)_{1\leqslant j\leqslant a+b}.
\]
These vectors allow us to encode the canonical infinite $H$-expansion
of an integer.

Let $n\in\N_0$ and write
\[
[n]_H^{-1}=(a_k)_{k\geqslant1}.
\]
Thus $(a_k)_{k\geqslant1}$ is the canonical infinite $H$-expansion of
$n$, and it is eventually zero. For each $\eta_j\in\Gamma$, set
\[
\operatorname{eval}_{\eta_j}(n)
:=\sum_{k\geqslant1}a_k\eta_j^{k-1}.
\]
Collecting these scalar evaluations, we define the \emph{stable
evaluation} of $n$ by
\[
\evalG(n):=
\left(\operatorname{eval}_{\eta_j}(n)\right)_{1\leqslant j\leqslant a+b}
=
\left(
\sum_{k\geqslant1}a_k\eta_j^{k-1}
\right)_{1\leqslant j\leqslant a+b}
\in E_s.
\]
The sums are finite because $(a_k)_{k\geqslant1}$ is eventually zero.
Equivalently,
\[
\evalG(n)
=
a_1s_0
+
\sum_{k\geqslant2}
a_k\pi_s\left(\ell(T_{k-1})\right)
=
\pi_s\left(
a_1e_L+\sum_{k\geqslant2}a_k\ell(T_{k-1})
\right).
\]
Thus $\evalG(n)$ is obtained by assigning a contracting vector to each
position of the infinite $H$-expansion, weighting it by the
corresponding digit, and summing the contributions. In particular,
$\evalG(0)=0$ follows directly from the definition.

For the recurring example $g=(1,3,2)$, one has
\[
h=(2,3,1)
\]
and
\[
H_1=1,\qquad H_2=2,\qquad H_3=7,\qquad H_4=21.
\]
Here $E_s=\C$ and $\Gamma=(\eta_1)$. For example, the reduced
$H$-expansions
\[
6=2H_1+2H_2,
\qquad
20=3H_2+2H_3
\]
give
\[
\evalG(6)=2+2\eta_1,
\qquad
\evalG(20)=3\eta_1+2\eta_1^2.
\]
In particular, the canonical expansion of $6$ is not greedy, since the
usual greedy algorithm gives $6=3H_2$. The colored point cloud in
Figure~\ref{fig:intro-expanded-rauzy-132} is obtained by plotting these
stable evaluations in the complex plane and grouping them according to
their terminal blocks. This makes concrete how the arithmetic coding by
terminal blocks is realized in the contracting space underlying the
Rauzy geometry illustrated in Figure~\ref{fig:intro-expanded-rauzy-132}.

At this point, the arithmetic and substitutional descriptions live in
the same contracting space $E_s$. An integer $n$, through its infinite
$H$-expansion, is represented by the stable evaluation
$\evalG(n)\in E_s$, whereas a finite prefix $W$ of $A_\infty$ is
represented by the projected letter-count vector
$\pi_s(\ell(W))\in E_s$. The structural coding developed in the next
section relates these two descriptions.

%% file: sections/03-structural-coding.tex
\section{Structural coding of terminal blocks}\label{sec:coding}

Section~\ref{sec:general-data} placed two kinds of data in the same stable
space $E_s$. An integer $n$ is represented by its stable evaluation
$\evalG(n)$, while a finite prefix $W$ of $A_\infty$ is represented by
the projected letter-count vector $\pi_s(\ell(W))$. The purpose of the present section is to determine how each possible terminal block is encoded by a prefix of $A_\infty$. This is the symbolic step that connects the small-weight end of the canonical infinite $H$-expansion with the substitution structure.

The section is divided into four subsections. First, we construct the
expanded fixed word by replacing each letter $A_i$ with the complete
list of semi-proper blocks of digit-length $i$, and we establish the
recursive and length properties of the resulting finite words. Second,
we prove that the expanded fixed word lists
$\TB(0),\TB(1),\TB(2),\ldots$ in this order. Third, we isolate the
boundaries corresponding to the base terminal blocks $b_{i,0}$ and
identify the associated substitution prefixes; this gives the identity
between the stable evaluation of the corresponding integer and the
stable projection of its prefix-count vector. Finally, we recover the
remaining terminal-block conditions $b_{i,j}$ by translation.

\subsection{The expanded fixed word}
\label{subsec:expanded-fixed-word}

Recall that
$$
\rho_i=h_i-1
\qquad \left(1\leqslant i\leqslant L\right).
$$
For each $i\in\{1,\ldots,L\}$, the letter $A_i$ is associated with
the ordered list of all semi-proper blocks of digit-length $i$. We
encode this correspondence by the \emph{block-expansion morphism}
$\mathcal E$, defined by
$$
\mathcal E(A_i)
:=b_{i,0}b_{i,1}\cdots b_{i,\rho_i},
$$
where the blocks are listed in increasing order of their first digit.
We extend $\mathcal E$ to finite and one-sided infinite words by
letterwise concatenation. In particular, since $h_L=1$, one has $\mathcal E(A_L)=b_{L,0}$.

For a word $V$ of semi-proper blocks, let
$\lVert V\rVert$ denote the number of semi-proper
blocks in $V$, rather than the total number of digits contained in them.
By construction,
$$
\left\lVert \mathcal E(A_i)\right\rVert
=\rho_i+1=h_i.
$$

Recall from Section~\ref{sec:general-data} that
$
T_n=\sigma^{n-1}(A_1)
$
for $n\geqslant1$. By Definition~\ref{def:intro-associated-substitution},
these words are successive finite prefixes of the fixed word $A_\infty$.
We therefore set
$
T_\infty:=A_\infty
$
and write
$$
\widehat{T}_n:=\mathcal E(T_n),
\qquad
\widehat{T}_\infty:=\mathcal E(T_\infty).
$$
These are the block expansions of the finite words $T_n$ and of the
fixed word $T_\infty$, respectively.

\begin{example}\label{ex:block-expansion-12}
Consider again the system $g=(1,2)$, for which $h=(2,1)$. Then
$\rho_1=1$, $\rho_2=0$, and
$$
\sigma(A_1)=A_1A_1A_2,
\qquad
\sigma(A_2)=A_1.
$$
The block-expansion morphism is given by
$$
\mathcal E(A_1)=b_{1,0}\,b_{1,1},
\qquad
\mathcal E(A_2)=b_{2,0}.
$$
The first substitution words are
$$
T_1=A_1,
\qquad
T_2=A_1A_1A_2,
\qquad
T_3=A_1A_1A_2A_1A_1A_2A_1.
$$
Applying $\mathcal E$ gives
$$
\widehat T_1=b_{1,0}\,b_{1,1} \qquad \text{and} \qquad
\widehat T_2
=
b_{1,0}\,b_{1,1}\,
b_{1,0}\,b_{1,1}\,
b_{2,0}.
$$
Hence
$$
\widehat T_\infty
=
b_{1,0}\,b_{1,1}\,
b_{1,0}\,b_{1,1}\,
b_{2,0}\,
b_{1,0}\,b_{1,1}\,
b_{1,0}\,b_{1,1}\,
b_{2,0}\,
b_{1,0}\,b_{1,1}\cdots.
$$
\end{example}

We use juxtaposition or a colon to denote
concatenation. For a word $W$, the notation $W^q$ denotes the
concatenation of $q$ copies of $W$, with $W^0$ equal to the empty word.
If $W$ is finite or one-sided infinite, then $W[n]$ denotes the word
formed by its first $n$ entries. For $r\geqslant1$, the notation $W(r)$
denotes the $r$th entry of $W$. Thus $T_\infty(r)$ is a letter $A_i$,
whereas $\widehat{T}_\infty$ is a block word and
$\widehat{T}_\infty(r)$ is a semi-proper block. In particular,
$\widehat{T}_\infty[n]$ consists of the first $n$ semi-proper blocks of
$\widehat{T}_\infty$, whereas $\widehat{T}_\infty(n+1)$ is the single
block in position $n+1$.

The recursive structure of the words $T_s$ mirrors the recurrence
satisfied by the weights $H_s$.

\begin{lemma}\label{lem:recursive-T}
For $2\leqslant s\leqslant L$, one has
\begin{equation}\label{eq:recursive-T-initial}
T_s
=
T_{s-1}^{h_1}:T_{s-2}^{h_2}:\cdots:
T_1^{h_{s-1}}:A_s.
\end{equation}
For $s\geqslant L+1$, one has
\begin{equation}\label{eq:recursive-T}
T_s
=
T_{s-1}^{h_1}:T_{s-2}^{h_2}:\cdots:
T_{s-L}^{h_L}.
\end{equation}
\end{lemma}

\begin{proof}
For $1\leqslant i<L$ and $r\geqslant1$, applying
$\sigma^{r-1}$ to
$$
\sigma(A_i)=A_1^{h_i}:A_{i+1}
$$
gives
$$
\sigma^r(A_i)
=
\left(\sigma^{r-1}(A_1)\right)^{h_i}
:\sigma^{r-1}(A_{i+1})
=
T_r^{h_i}:\sigma^{r-1}(A_{i+1}).
$$
Let $2\leqslant s\leqslant L$. Applying this identity successively to
$T_s=\sigma^{s-1}(A_1)$ until the letter $A_s$ is reached yields
$$
T_s
=
T_{s-1}^{h_1}:T_{s-2}^{h_2}:\cdots:
T_1^{h_{s-1}}:A_s.
$$
If $s\geqslant L+1$, the same iteration gives
$$
T_s
=
T_{s-1}^{h_1}:T_{s-2}^{h_2}:\cdots:
T_{s-L+1}^{h_{L-1}}:\sigma^{s-L}(A_L).
$$
Since $\sigma(A_L)=A_1$, one has
$$
\sigma^{s-L}(A_L)
=
\sigma^{s-L-1}(A_1)
=
T_{s-L}.
$$
Finally, $h_L=1$, and therefore
$$
T_s
=
T_{s-1}^{h_1}:T_{s-2}^{h_2}:\cdots:
T_{s-L}^{h_L}.
$$
\end{proof}

The recursive decomposition of $T_s$ induces, after applying
$\mathcal E$, the same recurrence for the number of semi-proper blocks.
This gives the following identification with the weights $H_s$.

\begin{corollary}\label{cor:length-expanded-word}
For every $s\geqslant1$, one has
$$
\left\lVert \widehat{T}_s\right\rVert=H_{s+1}.
$$
\end{corollary}

\begin{proof}
We proceed by induction on $s$. For $s=1$, we have
$$
\left\lVert \widehat{T}_1\right\rVert
=
\left\lVert \mathcal E(A_1)\right\rVert
=
h_1
=
H_2.
$$

Assume that the result holds for all positive integers smaller than $s$.
If $2\leqslant s\leqslant L$, Lemma~\ref{lem:recursive-T} gives
$$
\left\lVert \widehat{T}_s\right\rVert
=
\sum_{r=1}^{s-1}
h_r\left\lVert \widehat{T}_{s-r}\right\rVert
+
\left\lVert \mathcal E(A_s)\right\rVert.
$$
By the induction hypothesis and
$\left\lVert \mathcal E(A_s)\right\rVert=h_s$, this becomes
$$
\left\lVert \widehat{T}_s\right\rVert
=
\sum_{r=1}^{s-1}h_rH_{s+1-r}+h_s
=
\sum_{r=1}^{s}h_rH_{s+1-r}
=
H_{s+1}.
$$

If $s\geqslant L+1$, Lemma~\ref{lem:recursive-T} similarly gives
$$
\left\lVert \widehat{T}_s\right\rVert
=
\sum_{r=1}^{L}
h_r\left\lVert \widehat{T}_{s-r}\right\rVert
=
\sum_{r=1}^{L}h_rH_{s+1-r}
=
H_{s+1}.
$$
This completes the induction.
\end{proof}

This prepares the comparison between the expanded fixed word and the
terminal blocks of consecutive integers.

\subsection{Terminal blocks in natural order}

By Corollary~\ref{cor:length-expanded-word}, the word $\widehat T_s$
contains exactly $H_{s+1}$ semi-proper blocks, indexed by the positions
$1,\ldots,H_{s+1}$. We identify not only the block in position
$n+1$, but the whole prefix ending at that block. This combines the
terminal-block statement with the prefix decomposition needed later for
the base terminal classes.

\begin{theorem}\label{thm:terminal-coding}
Let $n>0$ and put
$$
s:=\operatorname{ord}([n]_H^{-1}).
$$
By Theorem~\ref{thm:block-language}, equivalently
$$
H_s\leqslant n<H_{s+1}.
$$
Write
$$
\operatorname{red}([n]_H^{-1})=(a_1,\ldots,a_s)
$$
for the reduced digit word of $n$, and let
$$
C:=\operatorname{red}\!\left(\operatorname{dgt}^{-1}([n]_H^{-1})\right)
$$
be its reduced semi-proper block word. Write
$$
C=C_0:b_{i,j},
$$
where $b_{i,j}$ is the semi-proper block at the high-weight end and
$C_0$ has total digit-length $s-i$, with $C_0$ empty when $i=s$. Let
$
b_{k,\ell}=\TB(n).
$

If $i<s$, then
$$
\widehat T_s[n+1]
=
\left(\widehat T_{s-1}\right)^{a_s}
:\left(\widehat T_{s-2}\right)^{a_{s-1}}
:\cdots:
\left(\widehat T_1\right)^{a_2}
:b_{k,0}b_{k,1}\cdots b_{k,\ell}.
$$
If $i=s$, then necessarily $s\leqslant L$ and
$$
\widehat T_s[n+1]
=
\left(\widehat T_{s-1}\right)^{h_1}
:\left(\widehat T_{s-2}\right)^{h_2}
:\cdots:
\left(\widehat T_1\right)^{h_{s-1}}
:b_{s,0}b_{s,1}\cdots b_{s,j},
$$
where the string of $\widehat T$-factors is empty when $s=1$.

In particular, for every $r\geqslant1$ and every
$0\leqslant n<H_{r+1}$,
$$
\widehat T_r(n+1)=\TB(n),
$$
and consequently, for every $n\in\N_0$,
$$
\widehat T_\infty(n+1)=\TB(n).
$$
\end{theorem}

\begin{proof}
We use induction on the number $\lVert C\rVert$ of semi-proper blocks
in the reduced block expansion $C$.

Suppose first that $\lVert C\rVert=1$. Then $i=s$, so
$C=b_{s,j}$ and necessarily $s\leqslant L$. If $s=1$, then
$$
\widehat T_1
=
b_{1,0}b_{1,1}\cdots b_{1,h_1-1},
$$
and the asserted prefix formula is immediate. Assume that
$2\leqslant s\leqslant L$. Applying $\mathcal E$ to
\eqref{eq:recursive-T-initial} gives
$$
\widehat T_s
=
\left(\widehat T_{s-1}\right)^{h_1}
:\cdots:
\left(\widehat T_1\right)^{h_{s-1}}
:b_{s,0}b_{s,1}\cdots b_{s,h_s-1}.
$$
The digit word of $b_{s,j}$ is
$
(j,h_{s-1},\ldots,h_1),
$
and therefore
$$
n
=
j+\sum_{r=1}^{s-1}h_rH_{s+1-r}.
$$
By Corollary~\ref{cor:length-expanded-word}, the complete
$\widehat T$-factors preceding $\mathcal E(A_s)$ contain
$
\sum_{r=1}^{s-1}h_rH_{s+1-r}
$
blocks. Hence the first $n+1$ blocks of $\widehat T_s$ are exactly
$$
\left(\widehat T_{s-1}\right)^{h_1}
:\cdots:
\left(\widehat T_1\right)^{h_{s-1}}
:b_{s,0}b_{s,1}\cdots b_{s,j},
$$
which proves the initial case.

Assume now that $\lVert C\rVert>1$. Then $i<s$ and $C_0$ is nonempty.
Set
$$
n_0
:=
\evalH{\operatorname{dgt}(C_0:z)}{H}
$$
and
$$
n_1
:=
\sum_{r=1}^{i-1}h_rH_{s+1-r}
+jH_{s+1-i}.
$$
The two terms are respectively the contributions of the low-weight
word $C_0$ and of the high-weight block $b_{i,j}$, so
$n=n_0+n_1$. By the global bijectivity in
Theorem~\ref{thm:block-language},
$$
[n_0]_H^{-1}=\operatorname{dgt}(C_0:z).
$$
Hence
$\operatorname{ord}([n_0]_H^{-1})\leqslant s-i$. If $n_0>0$, the
order characterization in Theorem~\ref{thm:block-language}
gives
$$
n_0<H_{\operatorname{ord}([n_0]_H^{-1})+1}
\leqslant H_{s-i+1};
$$
for $n_0=0$ the same upper bound is immediate. Thus
$$
0\leqslant n_0<H_{s-i+1}.
$$

The recursive decomposition of $\widehat T_s$ shows that the block word
$$
V
:=
\left(\widehat T_{s-1}\right)^{h_1}
:\cdots:
\left(\widehat T_{s-i+1}\right)^{h_{i-1}}
:\left(\widehat T_{s-i}\right)^j
$$
is the prefix immediately preceding the next copy of
$\widehat T_{s-i}$; when $i=1$, the string of $\widehat T$-factors is empty.
By Corollary~\ref{cor:length-expanded-word},
$$
\lVert V\rVert
=
\sum_{r=1}^{i-1}h_rH_{s+1-r}
+jH_{s+1-i}
=
n_1.
$$
Consequently,
$$
\widehat T_s[n+1]
=
V:\widehat T_{s-i}[n_0+1].
$$

If $n_0=0$, then $C_0$ is a string of copies of $b_{1,0}$, so
$\TB(n)=b_{1,0}$ and
$
\widehat T_{s-i}[1]=b_{1,0}.
$
Assume that $n_0>0$. Put
$$
R:=\operatorname{red}\!\left(
\operatorname{dgt}^{-1}([n_0]_H^{-1})
\right),
$$
and let $r$ be its total digit-length. Since
$[n_0]_H^{-1}=\operatorname{dgt}(C_0:z)$, the uniqueness of the
infinite block decomposition shows that $R$ is obtained from $C_0$ by
removing only copies of $b_{1,0}$ at the high-weight end. Hence
$\lVert R\rVert\leqslant\lVert C_0\rVert<\lVert C\rVert$,
and the induction hypothesis applies to $n_0$. Since
$\widehat T_r$ is a prefix of $\widehat T_{s-i}$, the prefix supplied
by the induction hypothesis is also the prefix
$\widehat T_{s-i}[n_0+1]$. Removing these trailing copies of
$b_{1,0}$ changes only zero coefficients at the high-weight end and
does not alter the first block. Thus, in both cases $n_0=0$ and
$n_0>0$,
$$
\widehat T_{s-i}[n_0+1]
=
\left(\widehat T_{s-i-1}\right)^{a_{s-i}}
:\cdots:
\left(\widehat T_1\right)^{a_2}
:b_{k,0}b_{k,1}\cdots b_{k,\ell},
$$
where $b_{k,\ell}=\TB(n_0)=\TB(n)$; when $n_0=0$, the
$\widehat T$-factors on the right are empty.

Finally, the digits contributed by the high-weight block $b_{i,j}$ are
$$
(a_{s-i+1},\ldots,a_s)
=
(j,h_{i-1},\ldots,h_1).
$$
Combining the last two displayed formulas with the definition of $V$
yields
$$
\widehat T_s[n+1]
=
\left(\widehat T_{s-1}\right)^{a_s}
:\left(\widehat T_{s-2}\right)^{a_{s-1}}
:\cdots:
\left(\widehat T_1\right)^{a_2}
:b_{k,0}b_{k,1}\cdots b_{k,\ell}.
$$
This completes the induction.

The last block in either prefix formula is $\TB(n)$, so
$\widehat T_s(n+1)=\TB(n)$ whenever $H_s\leqslant n<H_{s+1}$.
For $n=0$, one has
$
\widehat T_1(1)=b_{1,0}=\TB(0).
$
Since the words $\widehat T_s$ form an increasing sequence of prefixes,
the finite identity follows for every
$0\leqslant n<H_{r+1}$, and passage to the limit gives the identity in
$\widehat T_\infty$.
\end{proof}

In the Fibonacci case $g=(1,1)$,
Theorem~\ref{thm:terminal-coding} specializes to the
golden-string prefix decomposition of
\cite[Lemma~3.2]{Griffiths2011}. Related correspondences between low
digits of Zeckendorf expansions and letters of the golden string are
given in
\cite[Theorems~1.2 and~1.3]{BustosChuKimLeeShankarTresch2026}.

\subsection{Base terminal classes and prefix boundaries}

Theorem~\ref{thm:terminal-coding} identifies every terminal
block with its position in $\widehat T_\infty$. For a base block
$b_{i,0}$, deleting that final block from the corresponding prefix
leaves a concatenation of complete images under $\mathcal E$. Such a
block prefix can therefore be contracted to a prefix of $T_\infty$,
which is precisely the form needed for the stable Rauzy geometry.

We use the terminal-block coding to define the corresponding
classes of non-negative integers. For $1\leqslant i\leqslant L$,
the \emph{base terminal class} associated with $b_{i,0}$ is
$$
N_i
:=
\left\{
n\in\N_0:\TB(n)=b_{i,0}
\right\},
$$
and we set
$$
N_*:=N_1\sqcup\cdots\sqcup N_L.
$$
More generally, for $1\leqslant i\leqslant L$ and
$0\leqslant j\leqslant\rho_i$, the \emph{terminal class}
associated with $b_{i,j}$ is
$$
N_{i,j}
:=
\left\{
n\in\N_0:\TB(n)=b_{i,j}
\right\}.
$$
Thus $N_i=N_{i,0}$ for every $1\leqslant i\leqslant L$. By the
convention $\TB(0)=b_{1,0}$, one has $0\in N_1$ and
$0\notin N_i$ for every $i>1$.

To pass from a block word written in complete images back to the
corresponding substitution word, we use the inverse of $\mathcal E$
on its image. Let $\mathcal A^*$ denote the set of all finite words
over $\mathcal A$, including the empty word $\varepsilon$. Every word in
$\mathcal E(\mathcal A^*)$ has a unique decomposition into complete
images $\mathcal E(A_i)$. Indeed, the first block of a
nonempty such word is $b_{i,0}$ for a unique $i$, and therefore
determines its first complete image $\mathcal E(A_i)$. Removing this
image and iterating proves uniqueness. We may consequently define
$$
\begin{array}{r@{\;}c@{\quad}c@{\quad}c}
\operatorname{con}:&
\mathcal E(\mathcal A^*)
&\longrightarrow&
\mathcal A^*,\\
&
\mathcal E(A_{i_1}):\cdots:\mathcal E(A_{i_r})
&\longmapsto&
A_{i_1}:\cdots:A_{i_r}.
\end{array}
$$
We also set $\operatorname{con}(\varepsilon)=\varepsilon$.

Since the first digit of $b_{i,j}$ is $j$, an integer
$n\in\N_0$ belongs to $N_*$ exactly when the first digit of its
canonical infinite $H$-expansion vanishes. The prefix formula in
Theorem~\ref{thm:terminal-coding} then gives the required
complete-image decomposition. Let $s\geqslant1$ and let
$0\leqslant n<H_{s+1}$. Write
$$
[n]_H^{-1}=(a_1,a_2,\ldots)\in\mathcal W.
$$
If $n>0$ and $r:=\operatorname{ord}([n]_H^{-1})$, then
Theorem~\ref{thm:block-language} gives
$$
H_r\leqslant n<H_{r+1}.
$$
Since $n<H_{s+1}$, one must have $r\leqslant s$, and therefore
$a_k=0$ for $k>s$. The same conclusion is immediate when $n=0$.

Assume that $n\in N_*$, or equivalently that $a_1=0$. The case
$n=0$ is immediate. For $n>0$, apply
Theorem~\ref{thm:terminal-coding} with the order $r$ above.
Since $\widehat T_r$ is a prefix of $\widehat T_s$ and
$a_{r+1}=\cdots=a_s=0$, the same prefix formula gives
\begin{equation}\label{eq:base-prefix-decomposition}
\widehat T_s[n]
=
\left(\widehat T_{s-1}\right)^{a_s}
:\left(\widehat T_{s-2}\right)^{a_{s-1}}
:\cdots:
\left(\widehat T_1\right)^{a_2}.
\end{equation}
Indeed, the terminal block is of the form $b_{k,0}$, so deleting the
last block from the corresponding prefix $\widehat T_s[n+1]$ leaves
only complete expanded images. Since each factor
$\widehat T_t$ is the complete image $\mathcal E(T_t)$, applying
$\operatorname{con}$ to \eqref{eq:base-prefix-decomposition} yields
\begin{equation}\label{eq:base-prefix-contraction}
\operatorname{con}\left(\widehat T_s[n]\right)
=
T_{s-1}^{a_s}:T_{s-2}^{a_{s-1}}:\cdots:T_1^{a_2}.
\end{equation}

We can now connect the two stable-space constructions introduced in
Section~\ref{sec:general-data}. For $n\in N_*$, the corresponding
position in $\widehat T_\infty$ is a boundary between complete images
$\mathcal E(A_i)$, and the preceding block prefix contracts to a
prefix of $T_\infty$. The next theorem identifies the stable projection
of its letter-count vector with the stable evaluation of $n$.

Write the infinite fixed word as
$$
T_\infty=u_0u_1u_2\cdots.
$$
Recall that, for a finite word $W$ over $\mathcal A$,
$$
\ell(W)
=
\bigl(|W|_{A_1},\ldots,|W|_{A_L}\bigr)^{\mathsf T}
$$
denotes its column vector of letter counts. For $m\geqslant0$, set
$$
p_m
:=
\ell\left(u_0u_1\cdots u_{m-1}\right)\in\N_0^L,
\qquad
p_0=0.
$$

\begin{theorem}\label{thm:fixed-word-bridge}
For $m\geqslant0$, let
$$
c_m
:=
\left\lVert 
\mathcal E(u_0u_1\cdots u_{m-1})
\right\rVert,
$$
where $c_0=0$. Then $(c_m)_{m\geqslant0}$ is the increasing
enumeration of $N_*$. More precisely, if $u_m=A_i$, then
$c_m\in N_i$, and
$$
\evalG(c_m)=\pi_s(p_m).
$$
\end{theorem}

\begin{proof}
By definition, $\widehat{T}_\infty=\mathcal E(T_\infty)$ is obtained by
replacing each letter $A_i$ by the block word
$$
 \mathcal E(A_i)=b_{i,0}b_{i,1}\cdots b_{i,h_i-1}.
$$
Hence the block in position $c_m+1$ of $\widehat{T}_\infty$ is the
first semi-proper block of $\mathcal E(u_m)$. If $u_m=A_i$, this block
is $b_{i,0}$. By Theorem~\ref{thm:terminal-coding}, it is also
$\TB(c_m)$. Thus $c_m\in N_i$.

Conversely, let $n\in N_i$. By terminal coding, the block in position
$n+1$ of $\widehat{T}_\infty$ is $b_{i,0}$. Among the words
$\mathcal E(A_r)$, the block $b_{i,0}$ occurs only in
$\mathcal E(A_i)$, and there only as the first block. Therefore position
$n+1$ is the beginning of a unique image $\mathcal E(u_m)$ with
$u_m=A_i$, so $n=c_m$. Moreover, if $u_m=A_i$, then
$$
c_{m+1}-c_m
=
\left\lVert \mathcal E(A_i)\right\rVert
=
h_i>0.
$$
Hence $(c_m)_{m\geqslant0}$ is strictly increasing. This proves that it is
the increasing enumeration of $N_*$, with the stated letter class.

It remains to identify the stable evaluation. For $m=0$, one has
$c_0=0$, $p_0=0$, and $\evalG(0)=0$, so the identity is immediate.
Assume now that $m\geqslant1$ and put $n=c_m$. Since $(c_m)$ is
strictly increasing and $c_0=0$, one has $n>0$. Write
$$
[n]_H^{-1}=(a_1,a_2,\ldots)\in\mathcal W.
$$
Since $n\in N_*$, one has $a_1=0$. Choose
$q\geqslant2$ such that $n<H_{q+1}$. The word $\widehat T_q$ is a
prefix of $\widehat T_\infty$ and contains exactly $H_{q+1}$
semi-proper blocks by Corollary~\ref{cor:length-expanded-word}. Hence
\eqref{eq:base-prefix-decomposition} gives
$$
\widehat T_\infty[n]
=
\left(\widehat T_{q-1}\right)^{a_q}
:\left(\widehat T_{q-2}\right)^{a_{q-1}}
:\cdots:
\left(\widehat T_1\right)^{a_2}.
$$
By the definition of $c_m$, the left-hand side is
$\mathcal E(u_0u_1\cdots u_{m-1})$. Applying
$\operatorname{con}$ to both sides gives
$$
u_0u_1\cdots u_{m-1}
=
T_{q-1}^{a_q}:T_{q-2}^{a_{q-1}}:\cdots:T_1^{a_2}.
$$
Taking letter-count vectors yields
$$
p_m
=
\sum_{k=2}^q a_k\,\ell(T_{k-1}).
$$
Applying $\pi_s$ and using the stable-coordinate formula from
Section~\ref{sec:general-data} gives
$$
\pi_s(p_m)
=
\left(
\sum_{k=2}^q a_k\eta_j^{k-1}
\right)_{1\leqslant j\leqslant a+b}.
$$
Since $a_1=0$, adjoining the $k=1$ term does not change the sum.
Therefore
$$
\pi_s(p_m)=\evalG(c_m),
$$
as claimed.
\end{proof}

Theorem~\ref{thm:fixed-word-bridge} completes the bridge for the
base classes: the arithmetic evaluations of their elements are exactly
the stable projections of the corresponding prefix-count vectors. It
remains only to pass from the base block $b_{i,0}$ to an arbitrary
block $b_{i,j}$. Since the varying digit $j$ occupies the smallest
weight $H_1=1$, this final step is a translation by $j$ in the integers
and by $js_0$ in the stable space.

\subsection{Translation to arbitrary terminal classes}

We make this translation explicit at the level of infinite
$H$-expansions and stable evaluations.

\begin{lemma}\label{lem:translated-classes}
Let $1\leqslant i\leqslant L$, let
$0\leqslant j\leqslant\rho_i$, and let $n\in N_{i,j}$. Write the
canonical infinite block expansion of $n$ as
$$
C_n=b_{i,j}:W',
\qquad W'\in\mathbf W,
$$
and set
$$
C_n^*:=b_{i,0}:W'.
$$
Then $C_n^*\in\mathbf W$. Put
$$
n^*:=\evalH{\operatorname{dgt}(C_n^*)}{H}.
$$
Then
$$
n=n^*+j,
\qquad
n^*\in N_i,
$$
and
$$
\evalG(n)=\evalG(n^*)+js_0.
$$
Conversely, for every $m\in N_i$, one has
$$
m+j\in N_{i,j}
\qquad\text{and}\qquad
\evalG(m+j)=\evalG(m)+js_0.
$$
Consequently,
$$
N_{i,j}=N_i+j,
$$
where $N_i+j:=\{m+j:m\in N_i\}$.
\end{lemma}

\begin{proof}
The digit sequences $\operatorname{dgt}(C_n)$ and
$\operatorname{dgt}(C_n^*)$ differ only in the smallest-weight digit of
their first blocks. Since $H_1=1$,
$$
n
=
\evalH{\operatorname{dgt}(C_n)}{H}
=
jH_1+\evalH{\operatorname{dgt}(C_n^*)}{H}
=
n^*+j.
$$
By the uniqueness in Theorem~\ref{thm:block-language},
$$
C_{n^*}
=
\operatorname{dgt}^{-1}([n^*]_H^{-1})
=
C_n^*
=
b_{i,0}:W'.
$$
Hence $\TB(n^*)=b_{i,0}$ and therefore $n^*\in N_i$.

The same change in the first digit changes the stable evaluation by
$$
j(1,\ldots,1)=js_0,
$$
and therefore
$$
\evalG(n)=\evalG(n^*)+js_0.
$$

Conversely, let $m\in N_i$. Its canonical infinite block expansion
begins with $b_{i,0}$, so write
$$
C_m=b_{i,0}:W
$$
with $W\in\mathbf W$, and set
$$
C':=b_{i,j}:W.
$$
Then $C'\in\mathbf W$, and
$$
\evalH{\operatorname{dgt}(C')}{H}
=
jH_1+\evalH{\operatorname{dgt}(C_m)}{H}
=
m+j.
$$
Again by uniqueness,
$$
C_{m+j}
=
\operatorname{dgt}^{-1}([m+j]_H^{-1})
=
C'
=
b_{i,j}:W.
$$
Thus $\TB(m+j)=b_{i,j}$, so $m+j\in N_{i,j}$. The same first-digit
change gives
$$
\evalG(m+j)=\evalG(m)+js_0.
$$
This proves the converse and hence the translated-class identity.
\end{proof}

Since every non-negative integer has exactly one terminal block, the
terminal classes form the disjoint partition
$$
\N_0
=
\bigsqcup_{i=1}^{L}
\bigsqcup_{j=0}^{\rho_i}
N_{i,j}.
$$

The results of this section provide the symbolic bridge needed for the
geometric part of the argument. Theorem~\ref{thm:terminal-coding}
identifies terminal blocks with positions in the expanded fixed word,
Theorem~\ref{thm:fixed-word-bridge} identifies the base classes
with boundaries between complete images $\mathcal E(A_i)$ and realizes
their stable evaluations as projected prefix-count vectors, and
Lemma~\ref{lem:translated-classes} reduces every terminal class
to a translate of a base class. The next section places these
prefix-count vectors in the contracting Rauzy geometry.

%% file: sections/04-rauzy-windows.tex
\section{Rauzy windows and base terminal classes}\label{sec:rauzy}

Section~\ref{sec:coding} identified the stable evaluations of the base
terminal classes with projected prefix-count vectors of the substitution
fixed word. We pass from this symbolic-arithmetic correspondence to
its geometric realization.

In the classical Rauzy-fractal construction, prefix-count vectors are
projected onto the contracting space of a Pisot substitution, which in
our setting is denoted by $E_s$, and the closures of the resulting
projected sets give the central Rauzy set and its pieces indexed by the
letters of the substitution
\cite{Rauzy1982,ArnouxIto2001,CanteriniSiegel2001,SiegelThuswaldner2009}.
Here, Theorem~\ref{thm:fixed-word-bridge} shows that the
points associated with the base terminal class $N_i$ are precisely the
projected prefixes immediately preceding an occurrence of $A_i$.
Their closures therefore provide geometric sets encoding the base
terminal classes. We refer to these sets as \emph{Rauzy windows}. They play,
for the present numeration system, the role that residue classes play
for terminal digits in an ordinary base.

The purpose of this section is twofold. First, we isolate the precise
geometric property needed for the arithmetic argument: the central
Rauzy window must be a multiplicity-one fundamental domain for the
natural prefix lattice, meaning that its lattice translates cover the
stable space and overlap only on a set of Lebesgue measure zero, with
its letter-indexed pieces likewise disjoint up to null sets.
Second, assuming this property, we relate the projected prefix orbit to
the linear toral orbit $(ms_0)_{m\geqslant 0}$ and derive
equidistribution for the base terminal classes. The proof of this
geometric statement under Assumption~\ref{ass:pisot} is postponed to
Section~\ref{sec:rauzy-verification}, so that the arithmetic argument
can be developed modularly from the precise geometric input it requires.

\subsection{Rauzy windows and the tiling theorem}

The set $R_i$ is designed to encode the base condition
$\TB(n)=b_{i,0}$.
Indeed, Theorem~\ref{thm:fixed-word-bridge} identifies the points
$\evalG(n)$ for $n\in N_i$ with the projected prefix-count vectors
$\pi_s(p_m)$ for which $u_m=A_i$. We therefore define the windows by
taking the closures of exactly these sets.

Recall that $T_\infty=u_0u_1u_2\cdots$ is the fixed word of $\sigma$,
and that
$$
p_m=\ell\left(u_0u_1\cdots u_{m-1}\right)\in\Z^L
\qquad (m\geqslant 0)
$$
is the prefix-counting vector, with $p_0=0$. Define
$$
R:=\overline{\left\{\pi_s(p_m):m\geqslant 0\right\}}
$$
and
$$
R_i:=
\overline{\left\{\pi_s(p_m):u_m=A_i\right\}}
\qquad (1\leqslant i\leqslant L).
$$
We call $R$ the \emph{central Rauzy window} and
$R_1,\ldots,R_L$ its \emph{letter-indexed pieces}. Since the alphabet
is finite, these definitions give
$$
R=R_1\cup\cdots\cup R_L.
$$

To obtain a compact dynamical model, we quotient the stable space by
the natural difference lattice of prefix counts. Put
$$
D_0:=\left\{x=(x_1,\ldots,x_L)\in\R^L:
x_1+\cdots+x_L=0\right\}
=
\operatorname{span}_{\R}
\{e_i-e_L:1\leqslant i\leqslant L-1\}.
$$
The next lemma shows that the projected difference vectors generate a
full lattice in the contracting space.

\begin{lemma}\label{lem:rauzy-lattice-full-rank}
The vectors
$$
\pi_s(e_1-e_L),\ldots,\pi_s(e_{L-1}-e_L)
$$
form a real basis of $E_s$. Consequently,
$$
\mathcal L:=
\left\langle
\pi_s(e_i-e_L):1\leqslant i\leqslant L-1
\right\rangle_{\Z}
$$
is a full lattice in $E_s$.
\end{lemma}

\begin{proof}
Since $P_h$ is irreducible over $\Q$, it is separable in characteristic
zero. Hence $M^T$ has $L$ distinct eigenvalues and is diagonalizable over
$\C$. By the spectral decomposition used in
Section~\ref{sec:general-data} to define $\pi_s$, the coordinate
functionals of $\pi_s$ correspond to all contracting eigenvalues of
$M^T$. Therefore
$
\ker(\pi_s)
$
is the one-dimensional expanding eigenspace of $M^T$.

Since $M^T$ is primitive and non-negative, the Perron--Frobenius
theorem gives a generator of this eigenspace with strictly positive
coordinates. Such a vector cannot belong to
$$
D_0=\{x\in\R^L:x_1+\cdots+x_L=0\}
$$
unless it is zero. Therefore
$$
D_0\cap\ker(\pi_s)=\{0\},
$$
so the restriction
$$
\pi_s|_{D_0}:D_0\longrightarrow E_s
$$
is injective.

Both $D_0$ and $E_s$ have real dimension $L-1$, hence this restriction
is an isomorphism. Since
$$
e_1-e_L,\ldots,e_{L-1}-e_L
$$
is a real basis of $D_0$, their images form a real basis of $E_s$.
Their integer span is therefore a full lattice in $E_s$.
\end{proof}

Let
$$
\mathbb T_{\mathcal L}:=E_s/\mathcal L
$$
and denote by
$q_{\mathcal L}:E_s\to\mathbb T_{\mathcal L}$
the quotient map. By
Lemma~\ref{lem:rauzy-lattice-full-rank},
$\mathbb T_{\mathcal L}$ is a compact real torus of dimension $L-1$.

The algebraic properties needed for the Rauzy-fractal construction have
already been established in Section~\ref{sec:general-data}: under
Assumption~\ref{ass:pisot}, the substitution $\sigma$ is primitive,
unimodular, and irreducible Pisot. We only need to record one additional
coincidence property.

Recall that positive strong coincidence means that for every pair of
letters $A_i,A_j$, there exist an integer $k\geqslant 1$, a letter
$A_r$, and factorizations
$$
\sigma^k(A_i)=P A_r U,
\qquad
\sigma^k(A_j)=Q A_r V,
$$
such that
$$
\ell(P)=\ell(Q).
$$

\begin{lemma}\label{lem:rauzy-hypotheses}
Under Assumption~\ref{ass:pisot}, the substitution $\sigma$ satisfies
positive strong coincidence.
\end{lemma}

\begin{proof}
Let $A_i$ and $A_j$ be two letters. By the definition of $\sigma$, both
$\sigma(A_i)$ and $\sigma(A_j)$ begin with $A_1$. We may therefore take
$k=1$, $A_r=A_1$, and choose both preceding prefixes to be the empty
word. Their letter-count vectors are both equal to $0$. Hence positive
strong coincidence holds already at the first iterate.
\end{proof}

Together with the algebraic properties recalled above, this coincidence
condition places $\sigma$ in the standard Rauzy-fractal framework.
However, the statement needed here is sharper: the central Rauzy window
must tile the stable space by the lattice $\mathcal L$ with multiplicity
one, and the possible ambiguities must be confined to null boundaries.
This is precisely the form that allows toral equidistribution to be
lifted back to the Rauzy window. The following theorem is proved in
Section~\ref{sec:rauzy-verification}; its proof does not depend on the
applications in Sections~\ref{sec:rauzy}--\ref{sec:prime-blocks}.
Here and below, $\lambda_{L-1}$ denotes $(L-1)$-dimensional Lebesgue
measure on $E_s$.

\begin{theorem}\label{thm:rauzy-fundamental-domain}\label{hyp:rauzy-fundamental-domain}
Under Assumption~\ref{ass:pisot}, the central Rauzy window $R$ is compact, has positive
$(L-1)$-dimensional Lebesgue measure, and satisfies
$$
\lambda_{L-1}(\partial R)=0.
$$
For every $1\leqslant i\leqslant L$, the piece $R_i$ has positive
$(L-1)$-dimensional Lebesgue measure and
$$
\lambda_{L-1}(\partial R_i)=0,
$$
and the interiors of $R_1,\ldots,R_L$ are pairwise disjoint. Moreover, the
translates of $R$ by $\mathcal L$ form a multiplicity-one tiling of $E_s$:
one has
$$
R+\mathcal L=E_s,
$$
and the interiors of $R+\ell$, $\ell\in\mathcal L$, are pairwise
disjoint. Equivalently, the quotient map
$$
q_{\mathcal L}:R\longrightarrow\mathbb T_{\mathcal L}
$$
is onto and is one-to-one on $R\setminus\partial R$.
\end{theorem}

The proof is given in Proposition~\ref{prop:automatic-rauzy-hypothesis}
of Section~\ref{sec:rauzy-verification}. In the arguments through
Section~\ref{sec:prime-blocks}, we use the conclusions of
Theorem~\ref{thm:rauzy-fundamental-domain} without invoking its proof.

For the recurring example $g=(1,3,2)$,
Figure~\ref{fig:rauzy-lattice-translates-132} illustrates the
multiplicity-one tiling in
Theorem~\ref{thm:rauzy-fundamental-domain}.

\begin{figure}[!ht]
\centering
\includegraphics[width=0.8\textwidth]{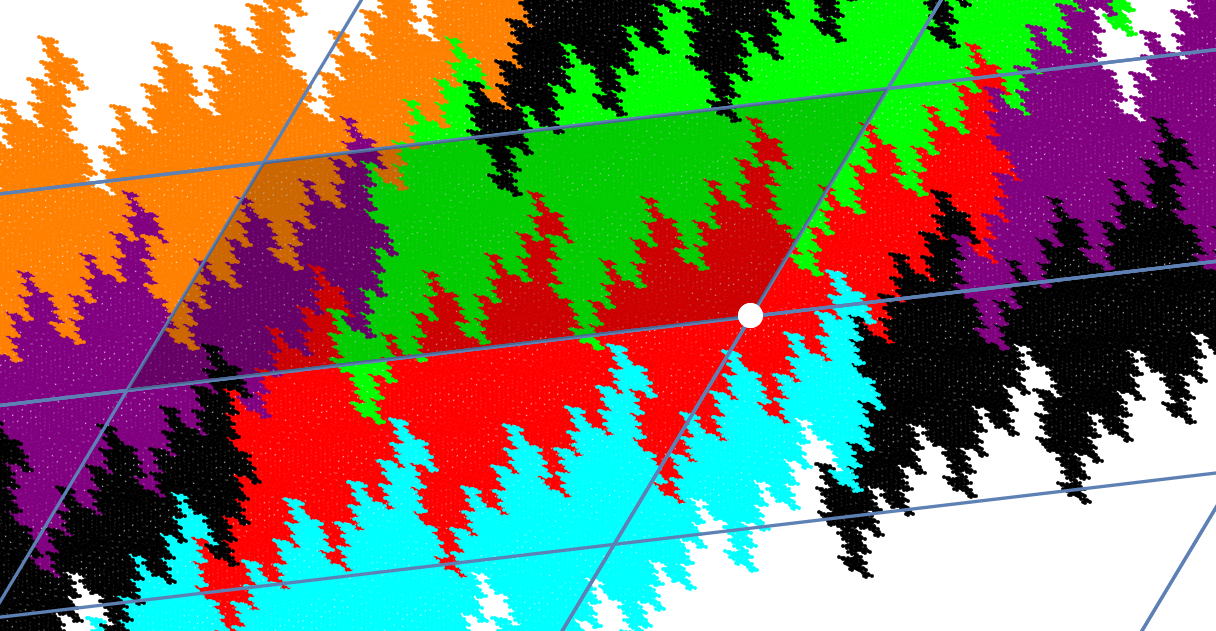}
\caption{Lattice translates of the central Rauzy window $R$ and a
fundamental parallelogram for $g=(1,3,2)$.}
\label{fig:rauzy-lattice-translates-132}
\end{figure}

For this example, $E_s=\C$ and
\[
 R=R_1\cup R_2\cup R_3,
\]
where $R_1,R_2,R_3$ are the red, green, and cyan base pieces,
respectively, in Figure~\ref{fig:intro-expanded-rauzy-132}.
Thus the central Rauzy window $R$ is exactly the union of these three
pieces. In Figure~\ref{fig:rauzy-lattice-translates-132}, this same
union is shown entirely in red, since the colors there have a different
meaning.

The lattice $\mathcal L$ is generated by
\[
 \pi_s(e_1-e_3)=\eta_1-1,
 \qquad
 \pi_s(e_2-e_3)=\eta_1^2-2\eta_1-1.
\]
The corresponding fundamental parallelogram $D$ is shown slightly
darker, and the white point marks the origin. If the central red copy
of $R$ is cut according to the lattice cells $D+\ell$,
$\ell\in\mathcal L$, then each resulting piece can be translated by a
lattice vector back into $D$. In this way, the translated pieces fill
$D$ up to their boundaries. In particular, the green region in $D$ is
obtained from the large red part of $R$ lying below $D$, the purple
region from the red part lying to the right of $D$, and the orange
region from the small red part lying on the other side of the origin
relative to $D$. This illustrates geometrically the identifications
induced by the quotient map
\[
 E_s\to E_s/\mathcal L.
\]

Theorem~\ref{thm:rauzy-fundamental-domain} provides the only
geometric input in the arguments that follow. Its role is precise: it identifies normalized Lebesgue
measure on $R$ with Haar measure on the torus $E_s/\mathcal L$, while
the null-boundary and disjointness properties make the correspondence
between the torus, the central Rauzy window, and its letter-indexed
pieces unambiguous almost everywhere. We next derive the consequences
needed for the base terminal classes before turning to the proof of
Theorem~\ref{thm:rauzy-fundamental-domain} in
Section~\ref{sec:rauzy-verification}.

\subsection{The prefix orbit on the Rauzy torus}

The lattice $\mathcal L$ was chosen so that each prefix-count vector is
congruent, after stable projection, to a vector depending only on the
length of the prefix. This turns the prefix orbit $\pi_s(p_m)$ into the
linear toral orbit $ms_0$, where $s_0=(1,\ldots,1)\in E_s$.

\begin{lemma}\label{lem:prefix-toral-congruence}
For every $m\geqslant 0$ one has
$$
\pi_s(p_m)\equiv ms_0\pmod{\mathcal L}.
$$
\end{lemma}

\begin{proof}
Write
$$
p_m=v_1(m)e_1+\cdots+v_L(m)e_L,
$$
where $v_i(m)$ is the number of occurrences of $A_i$ in
$u_0\cdots u_{m-1}$. Since
$$
v_1(m)+\cdots+v_L(m)=m,
$$
we have
$$
p_m-me_L=\sum_{i=1}^{L-1}v_i(m)(e_i-e_L).
$$
Applying $\pi_s$ gives
$$
\pi_s(p_m)-m\pi_s(e_L)\in\mathcal L.
$$
Since $\pi_s(e_L)=s_0$, this proves the congruence.
\end{proof}

Thus, after passing to the quotient by $\mathcal L$, the projected
prefix orbit is exactly the cyclic orbit generated by $s_0$. We first
establish the equidistribution of this toral orbit.

\begin{proposition}\label{prop:torus-equidistribution}
The sequence $(ms_0)_{m\geqslant 0}$ is equidistributed in
$\mathbb T_{\mathcal L}$ with respect to Haar measure.
\end{proposition}

\begin{proof}
By Lemma~\ref{lem:prefix-toral-congruence},
$$
q_{\mathcal L}\bigl(\pi_s(p_m)\bigr)
=
q_{\mathcal L}(ms_0)
$$
for every $m\geqslant 0$. Since
$$
R=
\overline{\{\pi_s(p_m):m\geqslant 0\}}
$$
and $q_{\mathcal L}$ is continuous, we have
$$
q_{\mathcal L}(R)
\subset
\overline{
\{q_{\mathcal L}(ms_0):m\geqslant 0\}
}.
$$
By Theorem~\ref{thm:rauzy-fundamental-domain},
$q_{\mathcal L}(R)=\mathbb T_{\mathcal L}$. Hence the cyclic orbit
$$
\{q_{\mathcal L}(ms_0):m\geqslant 0\}
$$
is dense in $\mathbb T_{\mathcal L}$. For a translation on a compact torus, density of the cyclic orbit is
equivalent, by Weyl's criterion, to equidistribution with respect to
Haar measure \cite[Chapter~1]{KuipersNiederreiter1974}.
\end{proof}

For $g=(1,3,2)$, Figure~\ref{fig:prefix-orbit-quotient-132}
illustrates the prefix orbit before and after quotienting by $\mathcal L$.

\begin{figure}[!htbp]
\centering
\includegraphics[width=0.96\textwidth]{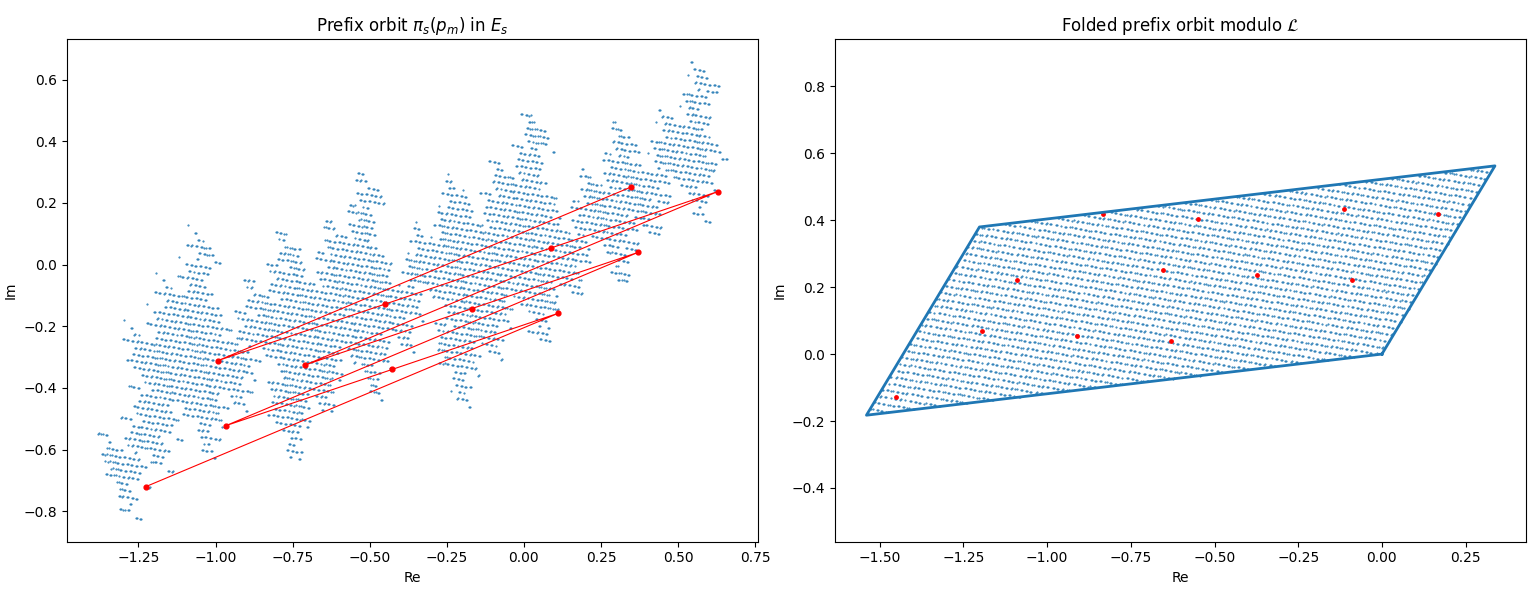}
\caption{The prefix orbit before and after quotienting by $\mathcal L$
for $g=(1,3,2)$.}
\label{fig:prefix-orbit-quotient-132}
\end{figure}
\FloatBarrier

The left-hand panel shows the projected prefix points $\pi_s(p_m)$ in
the contracting space $E_s$, with a short run of consecutive indices
marked in red. The right-hand panel shows the same points folded into
a fundamental parallelogram of $\mathcal L$. By
Lemma~\ref{lem:prefix-toral-congruence}, the linear orbit
$(ms_0\bmod\mathcal L)_{m\geqslant0}$ gives exactly the same points on
the quotient as the right-hand panel.

Next, we transfer this equidistribution statement from the torus back to
the central Rauzy window. The multiplicity-one property in
Theorem~\ref{thm:rauzy-fundamental-domain} provides an
almost-everywhere inverse to the quotient map on $R$, while the
null-boundary condition controls the exceptional points.

\begin{proposition}\label{prop:prefix-rauzy-equidistribution}
The sequence $(\pi_s(p_m))_{m\geqslant 0}$ is equidistributed in $R$
with respect to normalized Lebesgue measure on $R$.
\end{proposition}

\begin{proof}
Let $B:=q_{\mathcal L}(\partial R)\subset\mathbb T_{\mathcal L}$.
Since $\partial R$ is compact and has Lebesgue measure zero, and since
$q_{\mathcal L}$ is a local isometry, the set $B$ is compact and has
Haar measure zero. By Proposition~\ref{prop:torus-equidistribution},
$$
\frac{1}{M}
\#\left\{
0\leqslant m<M:
q_{\mathcal L}(ms_0)\in B
\right\}
\longrightarrow 0.
$$

Let $f$ be a continuous function on $R$. For $t\in\mathbb T_{\mathcal L}\setminus B$,
Theorem~\ref{thm:rauzy-fundamental-domain} gives a unique
point $x_t\in R\setminus\partial R$ such that
$q_{\mathcal L}(x_t)=t$. Define
$$
\widetilde f(t):=f(x_t)
\qquad (t\notin B),
$$
and set $\widetilde f(t):=0$ for $t\in B$.

The map $t\mapsto x_t$ is continuous on
$\mathbb T_{\mathcal L}\setminus B$. Indeed, locally around
$x_t\in\operatorname{int}(R)$, the quotient map $q_{\mathcal L}$ is a
homeomorphism onto its image, and uniqueness of the representative in
$R\setminus\partial R$ shows that $t\mapsto x_t$ is its local inverse.
Hence $\widetilde f$ is continuous on
$\mathbb T_{\mathcal L}\setminus B$. Since $B$ has Haar measure zero,
$\widetilde f$ is Riemann integrable on
$\mathbb T_{\mathcal L}$. Therefore
Proposition~\ref{prop:torus-equidistribution} gives
$$
\frac{1}{M}
\sum_{m<M}
\widetilde f\bigl(q_{\mathcal L}(ms_0)\bigr)
\longrightarrow
\int_{\mathbb T_{\mathcal L}}
\widetilde f\,d\nu_{\mathcal L},
$$
where $\nu_{\mathcal L}$ denotes normalized Haar measure.

Now suppose that $\pi_s(p_m)\notin\partial R$. By
Lemma~\ref{lem:prefix-toral-congruence},
$$
q_{\mathcal L}\bigl(\pi_s(p_m)\bigr)
=
q_{\mathcal L}(ms_0),
$$
and the representative of this torus point in
$R\setminus\partial R$ is unique. Hence
$$
\widetilde f\bigl(q_{\mathcal L}(ms_0)\bigr)
=
f\bigl(\pi_s(p_m)\bigr).
$$
If instead $\pi_s(p_m)\in\partial R$, then
$q_{\mathcal L}(ms_0)\in B$ by the same congruence. Thus the exceptional
indices have asymptotic density zero. Since both $f$ and
$\widetilde f$ are bounded, it follows that
$$
\frac{1}{M}\sum_{m<M}f\bigl(\pi_s(p_m)\bigr)
-
\frac{1}{M}\sum_{m<M}
\widetilde f\bigl(q_{\mathcal L}(ms_0)\bigr)
\longrightarrow 0.
$$

Finally, because $R$ is a multiplicity-one fundamental domain for
$\mathcal L$, the quotient map sends normalized Lebesgue measure on
$R$ to normalized Haar measure on $\mathbb T_{\mathcal L}$. Therefore
$$
\int_{\mathbb T_{\mathcal L}}
\widetilde f\,d\nu_{\mathcal L}
=
\frac{1}{\operatorname{vol}(R)}
\int_R f(x)\,dx.
$$
Combining the preceding limits gives
$$
\frac{1}{M}
\sum_{m<M}
f\bigl(\pi_s(p_m)\bigr)
\longrightarrow
\frac{1}{\operatorname{vol}(R)}
\int_R f(x)\,dx,
$$
which is the asserted equidistribution.
\end{proof}

We use this equidistribution result to recover the distribution of the base terminal classes.

\subsection{Equidistribution of the base terminal classes}

We have so far obtained equidistribution for all prefix points in the
central Rauzy set. The symbolic bridge from Section~\ref{sec:coding}
now turns this dynamical statement into an arithmetic one: the prefix
before $u_m$ corresponds to the integer $c_m$, and the next letter
selects its base terminal class.

The index $m$ always refers to the prefix ending just before the letter
$u_m$: the vector $p_m$ counts $u_0\cdots u_{m-1}$, while the next letter
$u_m$ determines the Rauzy piece. The sequence $c_m$ and the
prefix-evaluation bridge are given by
Theorem~\ref{thm:fixed-word-bridge}.

\begin{proposition}\label{prop:base-class-equidistribution}
The sequence $(\evalG(c_m))_{m\geqslant 0}$ is equidistributed in $R$
with respect to normalized Lebesgue measure. Moreover, for every
$1\leqslant i\leqslant L$, the subsequence
$
\bigl(\evalG(c_m)\bigr)_{m:\,c_m\in N_i}
$
is equidistributed in $R_i$ with respect to the probability measure
$
\operatorname{vol}(R_i)^{-1}\lambda_{L-1}|_{R_i},
$
and
$$
\lim_{M\to\infty}
\frac{1}{M}
\#\{0\leqslant m<M:c_m\in N_i\}
=
\frac{\operatorname{vol}(R_i)}{\operatorname{vol}(R)}.
$$
\end{proposition}

\begin{proof}
The first assertion follows from
Proposition~\ref{prop:prefix-rauzy-equidistribution} and the identity
$$
 \evalG(c_m)=\pi_s(p_m)
$$
from Theorem~\ref{thm:fixed-word-bridge}.

Fix $i$ and put
$$
 B_*:=\partial R_1\cup\cdots\cup\partial R_L.
$$
By Theorem~\ref{thm:rauzy-fundamental-domain}, the set $B_*$
has Lebesgue measure zero. Since the pieces cover $R$ and have pairwise
disjoint interiors, every point of $R\setminus B_*$ belongs to exactly
one $R_i$. On the other hand, Theorem~\ref{thm:fixed-word-bridge}
gives
$$
 c_m\in N_i\quad\Longleftrightarrow\quad u_m=A_i,
$$
and the definition of $R_i$ implies
$$
 u_m=A_i\quad\Longrightarrow\quad \pi_s(p_m)\in R_i.
$$
Consequently, whenever $\pi_s(p_m)\notin B_*$, one has
$$
 c_m\in N_i\quad\Longleftrightarrow\quad \pi_s(p_m)\in R_i.
$$
By Proposition~\ref{prop:prefix-rauzy-equidistribution} and
$\lambda_{L-1}(B_*)=0$, the indices for which
$\pi_s(p_m)\in B_*$ have asymptotic density zero. Thus the index sets defined by $c_m\in N_i$ and by
$\pi_s(p_m)\in R_i$ differ only by a set of asymptotic density zero.
Since $R_i$ is a continuity set for normalized Lebesgue measure on $R$
and $\operatorname{vol}(R_i)>0$, the standard restriction principle
for equidistributed sequences gives equidistribution in $R_i$ with
respect to normalized Lebesgue measure, together with
$$
\lim_{M\to\infty}\frac{1}{M}
\#\left\{0\leqslant m<M:c_m\in N_i\right\}
=
\frac{\vol(R_i)}{\vol(R)}.
$$
\end{proof}

The output of this section concerns only the base classes $N_i$, but it
contains the essential geometric information: their evaluations fill
the Rauzy pieces $R_i$ with the correct normalized measures. The next
section uses the translation identity $N_{i,j}=N_i+j$ and the
self-affine structure of the block language to reconstruct every
terminal class and to assemble a single full window for all integers.

%% file: sections/05-full-terminal-partition.tex
\section{The full terminal partition}\label{sec:recombination}

Section~\ref{sec:rauzy} produced geometric windows only for the base
blocks $b_{i,0}$. The Prime Number Theorem for Terminal Blocks
(Theorem~\ref{thm:intro-several-blocks}), however, must distinguish every
terminal block $b_{i,j}$ and, in its refined form, every fixed finite
word of consecutive terminal blocks. The purpose of this section is to
reconstruct all these classes without repeating the Rauzy analysis.

The translation identity from Section~\ref{sec:coding} first turns
$N_{i,j}$ into a translate of $N_i$, and hence turns $R_i$ into the
translated window $R_{i,j}$. We then use the block-prefixing operation to
show that the base pieces are self-affine copies of one full set $\widetilde R$.
This gives both the volume of every terminal window and a measurable
partition of $\widetilde R$. After proving equidistribution of all integer
evaluations in the full window, we iterate the same affine coding to
construct cylinder windows for arbitrary finite terminal block words.

For $1\leqslant i\leqslant L$ and $0\leqslant j\leqslant \rho_i=h_i-1$, put
$$
 R_{i,j}:=js_0+R_i.
$$

\begin{proposition}\label{prop:translated-classes}
For every $i$ and $j$ with $0\leqslant j\leqslant \rho_i$, the sequence $(\evalG(n))_{n\in N_{i,j}}$, taken in increasing order of $n$, is equidistributed in $R_{i,j}$ with respect to normalized Lebesgue measure.
\end{proposition}

\begin{proof}
By Theorem~\ref{thm:fixed-word-bridge}, the sequence
$(c_m)_{m\geqslant0}$ is the increasing enumeration of $N_*$. Hence, for
$j=0$, the subsequence considered in
Proposition~\ref{prop:base-class-equidistribution} is precisely
$(\evalG(n))_{n\in N_i}$ with $n$ taken in increasing order.

Now let $j>0$. By Lemma~\ref{lem:translated-classes},
$$
 N_{i,j}=N_i+j
$$
and
$$
 \evalG(n+j)=\evalG(n)+js_0\qquad \left(n\in N_i\right).
$$
Since $n\mapsto n+j$ preserves the natural order, the sequence indexed by
$N_{i,j}$ is obtained from the sequence indexed by $N_i$ by the fixed
translation $x\mapsto x+js_0$. This translation sends normalized Lebesgue
measure on $R_i$ to normalized Lebesgue measure on $R_{i,j}$, and the result
follows.
\end{proof}

\subsection{Self-affine structure and volumes}

Translations account for the varying first digit $j$, but they do not
yet determine the relative sizes of terminal classes of different
lengths $i$. For this we return to the block language. Prefixing the
fixed block $b_{i,0}$ shifts every remaining digit by $i$ positions;
in stable coordinates this becomes the contraction $K^i$. This is the
source of the factor $\omega^i$ in the final prime asymptotic.

Let
$$
 \widetilde R:=\overline{\left\{\evalG(n):n\in\N_0\right\}}.
$$

Define the stable contraction
\begin{equation}\label{def:K}
 K\left(\left(x_j\right)_j\right):=\left(\eta_j x_j\right)_j
\end{equation}
and, for $1\leqslant i\leqslant L$, define
$$
 t_i:=\sum_{\ell=1}^{i-1}h_{i-\ell}\left(\eta_j^\ell\right)_j,
$$
with the convention $t_1=0$. This is the stable contribution of the fixed small-weight block $b_{i,0}$.

The digit shift underlying this self-affine structure is useful in its own
right. We record it as an arithmetic function.

\begin{definition}\label{def:arithmetic-shift}
For $n\in\N_0$, write
\[
[n]_H^{-1}=(a_k)_{k\geqslant1}.
\]
We define
\[
\Phi(n):=\sum_{k\geqslant1}a_kH_{k+1}.
\]
Thus $\Phi$ shifts every coefficient of the canonical infinite
$H$-expansion by one weight. The definition is unambiguous because
Theorem~\ref{thm:block-language} gives a unique canonical
expansion, and the sum is finite because this expansion is eventually
zero. In the Zeckendorf case this is the classical Zeckendorf shift
studied in \cite{Kimberling1995}; the definition above is its
recurrence-numeration analogue.
\end{definition}

If $C_n=\operatorname{dgt}^{-1}([n]_H^{-1})$, then
$b_{1,0}:C_n\in\mathbf W$ has digit sequence
$(0,a_1,a_2,\ldots)$ and $H$-value $\Phi(n)$. Hence, by uniqueness,
\[
[\Phi(n)]_H^{-1}=(0,a_1,a_2,\ldots),
\]
so the iterates $\Phi^q$ shift the canonical expansion by exactly $q$
digit positions.

For example, in the order-two system $g=(1,2)$ one has
$\Phi(0)=0$, $\Phi(1)=H_2=2$, and $\Phi(2)=H_3=5$.

The base terminal classes are obtained by combining this shift with a
fixed small-weight block.

\begin{lemma}\label{lem:tail-stability}
For each $1\leqslant i\leqslant L$, the map
\[
\Psi_i:\N_0\longrightarrow N_i,
\qquad
\Psi_i(n):=\evalH{b_{i,0}}{H}+\Phi^i(n),
\]
is a bijection.
\end{lemma}

\begin{proof}
Let $n\in\N_0$ and write its infinite block expansion as
\[
C_n=\operatorname{dgt}^{-1}([n]_H^{-1}).
\]
Prefixing $b_{i,0}$ gives the infinite block word
$b_{i,0}:C_n\in\mathbf W$. Its first block is $b_{i,0}$, so its
$H$-value belongs to $N_i$. Since the prefix occupies the first $i$
digit positions and the remaining digits are shifted by $i$ positions,
this $H$-value is exactly
\[
\evalH{b_{i,0}}{H}+\Phi^i(n)=\Psi_i(n).
\]

Conversely, if $m\in N_i$, then its infinite block expansion begins
with $b_{i,0}$ and can be written uniquely as
\[
C_m=b_{i,0}:C
\]
with $C\in\mathbf W$. By Theorem~\ref{thm:block-language}, the
digit sequence $\operatorname{dgt}(C)$ is the infinite $H$-expansion
of a unique $n\in\N_0$. The preceding construction then gives
$\Psi_i(n)=m$. Hence $\Psi_i$ is bijective.
\end{proof}

We translate this block-prefixing bijection into stable coordinates.

\begin{proposition}\label{prop:self-affine-pieces}
For every $1\leqslant i\leqslant L$,
$$
 R_i=t_i+K^i(\widetilde R).
$$
\end{proposition}

\begin{proof}
By Theorem~\ref{thm:fixed-word-bridge}, $R_i$ is the closure of the stable evaluations of the integers in $N_i$. Indeed, the sequence $(c_m)_{m\geqslant 0}$ enumerates $N_*$, and the condition $u_m=A_i$ is equivalent to $c_m\in N_i$.

Let $m\in\N_0$ and put $n:=\Psi_i(m)\in N_i$. Prefixing
$b_{i,0}$ fixes the first $i$ digits and shifts the infinite
$H$-expansion of $m$ by $i$ positions. Therefore, in stable coordinates,
\[
\evalG(n)=t_i+K^i\evalG(m).
\]
Since $\Psi_i$ is a bijection from $\N_0$ onto $N_i$, taking closures gives
$$
 \overline{\left\{\evalG(n):n\in N_i\right\}}=t_i+K^i\overline{\left\{\evalG(m):m\in\N_0\right\}},
$$
which is the desired identity.
\end{proof}

The self-affine identity immediately gives the corresponding volume scaling.

\begin{corollary}\label{cor:self-affine-volume}
One has
$$
 \vol(R_i)=\omega^i\vol(\widetilde R)\qquad \left(1\leqslant i\leqslant L\right).
$$
\end{corollary}

\begin{proof}
As a real-linear map on $E_s$, one has
$$
 \abs{\det_{\mathbb R} K}=\abs{\gamma_1\cdots\gamma_{L-1}}=\omega.
$$
Indeed, each non-real coordinate contributes $\abs{\eta_j}^2$ to the real determinant, so the first equality is the product over all contracting conjugates. The product of all roots of $P_h$ is $(-1)^{L+1}h_L$, with $h_L=1$, and therefore the product of the contracting roots has absolute value $\varphi^{-1}=\omega$.

Theorem~\ref{thm:rauzy-fundamental-domain} gives measurability of the pieces $R_i$. Since $R_1=K(\widetilde R)$ and $K$ is invertible, $\widetilde R$ is measurable. Applying Proposition~\ref{prop:self-affine-pieces} and the change-of-variables formula gives
$$
 \vol(R_i)=\abs{\det_{\mathbb R} K}^i\vol(\widetilde R)=\omega^i\vol(\widetilde R).
$$
\end{proof}

Next, we assemble the translated pieces into a single full window.

\subsection{The full window and integer equidistribution}

The pieces $R_{i,j}$ now have the expected volumes. To use them as
terminal windows, we must still show that they cover the full set and
overlap only on null sets. The Pisot identity
$\sum_i h_i\omega^i=1$ supplies exactly the volume balance required for
this measurable partition.

\begin{proposition}\label{prop:measurable-partition}
The pieces $R_{i,j}$ form a measurable partition of $\widetilde R$ up to null sets:
$$
 \widetilde R=\bigcup_{i=1}^{L}\bigcup_{j=0}^{\rho_i}R_{i,j},
$$
with pairwise intersections of measure zero. Their boundaries have measure zero.
\end{proposition}

\begin{proof}
The classes $N_{i,j}$ partition $\N_0$. By Lemma~\ref{lem:translated-classes},
$$
 \overline{\left\{\evalG(n):n\in N_{i,j}\right\}}=js_0+\overline{\left\{\evalG(n):n\in N_i\right\}}=R_{i,j}.
$$
Since there are only finitely many classes, taking closures gives
$$
 \widetilde R=\bigcup_{i=1}^{L}\bigcup_{j=0}^{\rho_i}R_{i,j}.
$$

The boundary assertion follows from the boundary-null assumption for the Rauzy pieces $R_i$ and the fact that translations preserve null boundaries. It remains to prove that the overlaps are null. By Corollary~\ref{cor:self-affine-volume},
$$
 \sum_{i=1}^{L}\sum_{j=0}^{\rho_i}\vol(R_{i,j})
 =\sum_{i=1}^{L}h_i\vol(R_i)
 =\sum_{i=1}^{L}h_i\omega^i\vol(\widetilde R).
$$
The Pisot identity gives
$$
 \sum_{i=1}^{L}h_i\omega^i=1.
$$
Therefore the sum of the volumes of the finitely many pieces is exactly $\vol(\widetilde R)$. Since their union is $\widetilde R$, finite subadditivity can be an equality only if the pairwise overlaps have measure zero.
\end{proof}

For the recurring example $g=(1,3,2)$,
Figure~\ref{fig:full-terminal-window-132} illustrates this partition.
Here $R=R_1\cup R_2\cup R_3$, while the remaining terminal pieces are
obtained by translating the relevant base pieces by multiples of $s_0$.

\begin{figure}[!ht]
\centering
\includegraphics[width=\textwidth]{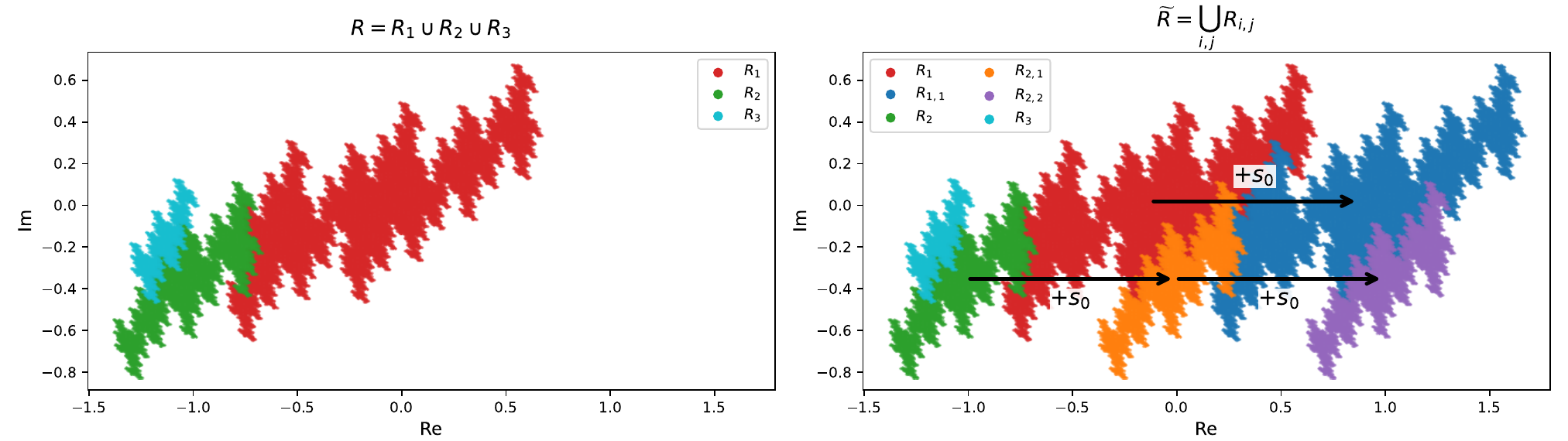}
\caption{The central Rauzy window $R$ and the full terminal window
$\widetilde R$ for $g=(1,3,2)$.}
\label{fig:full-terminal-window-132}
\end{figure}

With this measurable partition in hand, we can combine the classwise
equidistribution statements.

\begin{proposition}\label{prop:full-equidistribution}
The sequence $(\evalG(n))_{n\geqslant0}$ is equidistributed in $\widetilde R$
with respect to normalized Lebesgue measure.
\end{proposition}

\begin{proof}
For integers $X\geqslant1$, set
\[
 M_{i,j}(X):=\#\{n<X:n\in N_{i,j}\},
 \qquad
 M_i(X):=M_{i,0}(X),
\]
and
\[
 M(X):=\sum_{i=1}^{L}M_i(X).
\]
Thus $M(X)$ is the number of elements of $N_*$ below $X$.

Since $(c_m)_{m\geqslant0}$ is the increasing enumeration of $N_*$,
\[
 M_i(X)=\#\{0\leqslant m<M(X):c_m\in N_i\}.
\]
Hence, by Proposition~\ref{prop:base-class-equidistribution},
\[
 M_i(X)=q_iM(X)+o(M(X)),
 \qquad
 q_i:=\frac{\vol(R_i)}{\vol(R)}.
\]

For $0\leqslant j\leqslant \rho_i$, the identity
$N_{i,j}=N_i+j$ gives
\[
 M_{i,j}(X)=M_i(X-j).
\]
Since $j$ is fixed and $M(X-j)=M(X)+O(1)$, we obtain
\[
 M_{i,j}(X)=q_iM(X)+o(M(X))
 \qquad (X\to+\infty).
\]

The classes $N_{i,j}$ partition $\N_0$. Therefore, for every integer
$X\geqslant1$,
\[
 X
 =\sum_{i=1}^{L}\sum_{j=0}^{\rho_i}M_{i,j}(X)
 =\left(\sum_{i=1}^{L}h_iq_i\right)M(X)+o(M(X)).
\]
Consequently, as $X\to+\infty$,
\[
 \frac{M(X)}{X}
 \longrightarrow
 \frac{1}{\sum_{r=1}^{L}h_rq_r},
\]
and therefore, for every $i,j$,
\[
 \frac{M_{i,j}(X)}{X}
 \longrightarrow
 \frac{q_i}{\sum_{r=1}^{L}h_rq_r}
 =
 \frac{\vol(R_i)}
 {\sum_{r=1}^{L}h_r\vol(R_r)}.
\]

By Proposition~\ref{prop:measurable-partition},
\[
 \vol(\widetilde R)
 =\sum_{r=1}^{L}\sum_{j=0}^{\rho_r}\vol(R_{r,j})
 =\sum_{r=1}^{L}h_r\vol(R_r),
\]
while $\vol(R_{i,j})=\vol(R_i)$. Therefore
\[
 \frac{M_{i,j}(X)}{X}
 \longrightarrow
 \frac{\vol(R_{i,j})}{\vol(\widetilde R)}
 \qquad (X\to+\infty).
\]

Let $f\in C(\widetilde R)$. Since the limiting density of every
$N_{i,j}$ is positive, $M_{i,j}(X)>0$ for all sufficiently large $X$.
By Proposition~\ref{prop:translated-classes}, for every $i,j$,
\[
 \frac{1}{M_{i,j}(X)}
 \sum_{\substack{n<X\\ n\in N_{i,j}}}
 f(\evalG(n))
 \longrightarrow
 \frac{1}{\vol(R_{i,j})}
 \int_{R_{i,j}}f\,d\lambda_{L-1}
 \qquad (X\to+\infty).
\]
Since the classes $N_{i,j}$ partition $\N_0$,
\[
 \frac{1}{X}\sum_{0\leqslant n<X}f(\evalG(n))
 =
 \sum_{i=1}^{L}\sum_{j=0}^{\rho_i}
 \frac{M_{i,j}(X)}{X}
 \left(
 \frac{1}{M_{i,j}(X)}
 \sum_{\substack{n<X\\ n\in N_{i,j}}}
 f(\evalG(n))
 \right).
\]
Passing to the limit gives
\[
 \frac{1}{X}\sum_{0\leqslant n<X}f(\evalG(n))
 \longrightarrow
 \frac{1}{\vol(\widetilde R)}
 \sum_{i=1}^{L}\sum_{j=0}^{\rho_i}
 \int_{R_{i,j}}f\,d\lambda_{L-1}.
\]
Finally, the sets $R_{i,j}$ partition $\widetilde R$ up to null sets, so
\[
 \frac{1}{X}\sum_{0\leqslant n<X}f(\evalG(n))
 \longrightarrow
 \frac{1}{\vol(\widetilde R)}
 \int_{\widetilde R}f\,d\lambda_{L-1}
 \qquad (X\to+\infty),
\]
which proves the desired equidistribution.
\end{proof}

We have now treated the classes determined by a single terminal block;
we next extend the construction to longer terminal block words.

\subsection{Longer terminal block words}

We prescribe several terminal blocks at once. For fixed $e\geqslant1$
and $B\in\mathcal B_e$, define
\[
N_B:=\{n\in\N_0:\TBe{e}(n)=B\},
\]
and set $N_{\varepsilon}:=\N_0$. Because $\TBe{e}(n)$ is defined from
the infinite block expansion, these classes are defined for every
$n\in\N_0$, including $n=0$.

Fix $e\geqslant0$ and $B\in\mathcal B_e$, and put $m:=\len(B)$. The
underlying operation is again a digit shift: prefixing $B$ fixes the
first $m$ digit positions and shifts the remaining infinite expansion
by $m$ positions. Equivalently, the map
\[
\Psi_B(r):=\evalH{B}{H}+\Phi^m(r)
\qquad (r\in\N_0)
\]
is a bijection from $\N_0$ onto $N_B$. Indeed, if
$C_r=\operatorname{dgt}^{-1}([r]_H^{-1})$, then the infinite block
expansion of $\Psi_B(r)$ is $B:C_r$; conversely, deleting the first $e$
blocks from the infinite expansion of an element of $N_B$ recovers a
unique tail in $\mathbf W$ and hence a unique $r\in\N_0$.

Moreover, $\Psi_B$ preserves the natural order. Prefixing the same
block word $B$ shifts the largest digit position at which two infinite
coefficient sequences differ by $m$ positions without changing the
comparison at that position. By the order characterization in
Theorem~\ref{thm:block-language}, this implies that $\Psi_B$ is
strictly increasing. This order preservation will allow us to transfer
the integer equidistribution directly to $N_B$.

We record the same shift formally before specialization at the
contracting conjugates. Let
\[
\mathbf X=(X_1,\ldots,X_{a+b})
\]
be a tuple of independent indeterminates. If
$[n]_H^{-1}=(a_k)_{k\geqslant1}$, define the \emph{formal stable coding}
by
\[
\evalX(n)
:=
\left(
\sum_{k\geqslant1}a_kX_j^{k-1}
\right)_{1\leqslant j\leqslant a+b}
\in \prod_{j=1}^{a+b}\R[X_j].
\]

Let
\[
\spG:
\prod_{j=1}^{a+b}\R[X_j]
\longrightarrow E_s
\]
be the coordinatewise specialization
\[
\spG\bigl((F_j(X_j))_{j=1}^{a+b}\bigr)
:=
\bigl(F_j(\eta_j)\bigr)_{j=1}^{a+b}.
\]
Then $\spG(\evalX(n))=\evalG(n)$.

For $\mathbf F=(F_j(X_j))_{j=1}^{a+b}$, define the formal shift
\[
\Phi_{\mathbf X}(\mathbf F)
:=
\bigl(X_jF_j(X_j)\bigr)_{j=1}^{a+b}.
\]
Thus
\[
\evalX(\Phi^q(n))=\Phi_{\mathbf X}^q(\evalX(n))
\]
for $q\geqslant0$, and specialization gives
\[
\spG\!\left(\Phi_{\mathbf X}^q(\mathbf F)\right)=K^q\spG(\mathbf F).
\]

If $(d_1,\ldots,d_m)$ is the digit word of $B$, define its
\emph{formal prefix vector} and \emph{specialized prefix vector} by
\[
V_B(\mathbf X)
:=
\left(
\sum_{r=1}^{m}d_rX_j^{r-1}
\right)_{1\leqslant j\leqslant a+b}
\qquad\text{and}\qquad
v_B:=\spG\left(V_B(\mathbf X)\right).
\]
For $B=\varepsilon$, set $V_{\varepsilon}(\mathbf X):=0$ and
$v_{\varepsilon}:=0$.

Finally, define the formal cylinder and its stable window by
\[
\mathcal C_B:=\{\evalX(n):n\in N_B\}
\qquad\text{and}\qquad
R_B:=\overline{\spG(\mathcal C_B)}.
\]
In particular, $R_{\varepsilon}=\widetilde R$.

Prefixing a terminal block word is represented formally by the
following affine cylinder identity.

\begin{proposition}\label{prop:terminal-cylinder-window}
Let $e\geqslant0$ and $B\in\mathcal B_e$. Then
\[
\mathcal C_B
=
V_B(\mathbf X)+\Phi_{\mathbf X}^{\len(B)}\mathcal C_{\varepsilon}.
\]
Consequently,
\[
R_B=v_B+K^{\len(B)}(\widetilde R).
\]
The sequence $(\evalG(n))_{n\in N_B}$, taken in increasing order of
$n$, is equidistributed in $R_B$ with respect to normalized Lebesgue
measure. Moreover,
\[
\vol(R_B)=\omega^{\len(B)}\vol(\widetilde R),
\]
and $\partial R_B$ has measure zero.
\end{proposition}

\begin{proof}
Put $m:=\len(B)$. The bijection $\Psi_B:\N_0\to N_B$ gives, for every
$r\in\N_0$,
\[
\evalX(\Psi_B(r))
=
V_B(\mathbf X)+\Phi_{\mathbf X}^{m}\evalX(r).
\]
Since every element of $N_B$ is obtained from a unique $r$, this proves
the exact identity
\[
\mathcal C_B
=
V_B(\mathbf X)+\Phi_{\mathbf X}^{m}\mathcal C_{\varepsilon}.
\]
After specialization,
\[
\spG(\mathcal C_B)
=
v_B+K^m\spG(\mathcal C_{\varepsilon}).
\]
The affine map $x\mapsto v_B+K^mx$ is a homeomorphism, so taking
closures gives
\[
R_B=v_B+K^m(\widetilde R).
\]

Because $\Psi_B$ preserves the natural order, the sequence indexed by
$N_B$ is exactly the affine image of the full integer sequence. Hence
Proposition~\ref{prop:full-equidistribution} gives the claimed
equidistribution in $R_B$.

Since $\abs{\det_{\mathbb R}K}=\omega$,
\[
\vol(R_B)
=
\abs{\det_{\mathbb R}K}^{m}\vol(\widetilde R)
=
\omega^{\len(B)}\vol(\widetilde R).
\]
Finally, Proposition~\ref{prop:measurable-partition} gives a
finite covering of $\widetilde R$ by the sets $R_{i,j}$, and hence
\[
\partial\widetilde R
\subseteq
\bigcup_{i=1}^{L}\bigcup_{j=0}^{\rho_i}\partial R_{i,j}.
\]
The right-hand side has measure zero. Since
\[
\partial R_B=v_B+K^m(\partial\widetilde R),
\]
the boundary of $R_B$ also has measure zero.
\end{proof}

For fixed $e$, the cylinder classes now partition the integers exactly,
and the corresponding windows partition the full window up to null
sets.

\begin{proposition}\label{prop:terminal-cylinder-partition}
For every fixed $e\geqslant1$, the family
\[
(R_B)_{B\in\mathcal B_e}
\]
forms a measurable partition of $\widetilde R$ up to null sets.
\end{proposition}

\begin{proof}
The classes $N_B$, with $B\in\mathcal B_e$, form a disjoint partition
of $\N_0$, because every infinite block expansion has exactly one
initial word of $e$ blocks. Therefore
\[
\widetilde R
=
\overline{
\bigcup_{B\in\mathcal B_e}
\{\evalG(n):n\in N_B\}
}.
\]
Since $\mathcal B_e$ is finite, the closure commutes with the union, and
hence
\[
\widetilde R=\bigcup_{B\in\mathcal B_e}R_B.
\]
There are $h_i$ semi-proper blocks of digit-length $i$. Since
$\len$ is additive under concatenation, it follows that
\[
\sum_{B\in\mathcal B_e}\omega^{\len(B)}
=
\left(\sum_{i=1}^{L}h_i\omega^i\right)^e
=1.
\]
By Proposition~\ref{prop:terminal-cylinder-window}, the sum of the
volumes of the finitely many sets $R_B$ is therefore exactly
$\vol(\widetilde R)$. Since their union is $\widetilde R$, all pairwise
overlaps have measure zero. Their boundaries are null by the same
proposition.
\end{proof}

We have now associated every fixed terminal block word with a cylinder
window $R_B$, whose relative volume is $\omega^{\len(B)}$. To pass from
integers to primes, it remains to represent the sequence $\evalG(n)$ by
a linear orbit on a compact torus. The next section constructs the
lattice adapted to the arithmetic weights $H_k$ and proves the required
toral congruence.

%% file: sections/06-toral-model.tex
\section{A toral model for the full window}\label{sec:volume}

Throughout this section, we work under
Assumption~\ref{ass:pisot}. By
Theorem~\ref{thm:rauzy-fundamental-domain}, the required
multiplicity-one Rauzy geometry is automatic under this assumption.
The proof of that theorem is given in
Section~\ref{sec:rauzy-verification}.

The previous section associated with every fixed terminal block word
$B$ a cylinder window $R_B\subset \widetilde R$ and showed that
\[
\frac{\vol(R_B)}{\vol(\widetilde R)}
=
\omega^{\len(B)}.
\]
Thus the relative volume associated with $B$ is already known. To turn
this geometric quantity into the density of primes with terminal block
word $B$, we place the windows on a torus and relate the digit-dependent
point $\evalG(n)$ to a linear orbit.

For this purpose, we seek a lattice $\mathcal T\subset E_s$ such that
\begin{equation}\label{eq:toral-congruence-goal}
\evalG(n)\equiv ns_0\pmod{\mathcal T}.
\end{equation}
If $\widetilde R$ projects almost everywhere injectively onto
$E_s/\mathcal T$, then each cylinder window $R_B$ determines a toral
region with the same relative measure. The congruence above then turns
the terminal-block condition into membership of the linear orbit
$ns_0$ in the corresponding toral region, outside a null exceptional
set. Once the prime-indexed orbit $ps_0$ is shown to be equidistributed,
that exceptional set is negligible and the required prime density
follows.

The Rauzy lattice $\mathcal L$ is natural for substitution prefixes and
linearizes the prefix index $m$, but it does not directly express the
arithmetic value of an arbitrary integer. We therefore introduce a
second lattice $\mathcal T$, adapted to the recurrence weights, for which
the integer parameter $n$ itself becomes linear on the torus.

The argument has three stages. First, we compare the covolumes of
$\mathcal L$ and $\mathcal T$ and show that $\widetilde R$ has exactly the
volume of one fundamental domain for $\mathcal T$. Second, we prove the
arithmetic congruence~\eqref{eq:toral-congruence-goal}. Third, after identifying
$E_s/\mathcal T$ with $\R^d/\Z^d$, we prove the rational independence
needed for the prime orbit and show that the projection of
$\widetilde R$ to the torus is one-to-one almost everywhere.

\subsection{The two lattices and the covolume problem}

We begin with the first step described above: we construct the
arithmetic lattice $\mathcal T$ and compare its covolume with that of
the Rauzy lattice $\mathcal L$. Define
\[
\mathcal T
:=
\sum_{r=1}^d \Z Q_r(\Gamma),
\qquad d:=L-1,
\]
where
\[
Q_r(X)
:=
X^{-r}+g_2X^{-(r-1)}+g_3X^{-(r-2)}+\cdots+g_rX^{-1}
\qquad (1\leqslant r\leqslant d).
\]
Equivalently, in the $h$-notation,
\[
Q_r(X)
=
X^{-r}+h_{L-1}X^{-(r-1)}+h_{L-2}X^{-(r-2)}+\cdots+h_{L-r+1}X^{-1}.
\]
For $r=1$, this is simply $Q_1(X)=X^{-1}$. The generators of
$\mathcal T$ are chosen so that the stable contribution of the
recurrence weights $H_k$ can later be compared with $H_ks_0$ modulo
$\mathcal T$.

The Rauzy lattice gives the starting point for the covolume comparison.
By Theorem~\ref{thm:rauzy-fundamental-domain} and
Corollary~\ref{cor:self-affine-volume},
\[
\covol(\mathcal L)
=
\vol(R)
=
\sum_{i=1}^{L}\vol(R_i)
=
\left(\sum_{i=1}^{L}\omega^i\right)\vol(\widetilde R).
\]
Hence the desired identity
$\vol(\widetilde R)=\covol(\mathcal T)$ will follow from the next
comparison.

\begin{proposition}\label{prop:covol-ratio}
The two lattices $\mathcal L$ and $\mathcal T$ satisfy
\[
\frac{\covol(\mathcal L)}{\covol(\mathcal T)}
=
\omega+\omega^2+\cdots+\omega^L.
\]
\end{proposition}

To prove the proposition, it is convenient to express both covolumes by
determinants in the full list of contracting roots. For a Laurent
polynomial $F$, write
\[
F(\widetilde\Gamma)
=
\left(F(\gamma_1),\ldots,F(\gamma_d)\right),
\]
and, for Laurent polynomials $F_1,\ldots,F_d$, set
\[
\Delta(F_1,\ldots,F_d)
:=
\det\left(F_j(\gamma_i)\right)_{1\leqslant i,j\leqslant d},
\]
where $i$ is the row index and $j$ the column index. Also put
\[
V:=\Delta(1,X,\ldots,X^{d-1}).
\]

The generators of the Rauzy lattice $\mathcal L$ have stable-coordinate
expressions
\[
P_r(\Gamma),\qquad 1\leqslant r\leqslant d,
\]
where, for $1\leqslant r\leqslant L$,
\[
P_r(X)
:=
X^r-h_1X^{r-1}-h_2X^{r-2}-\cdots-h_{r-1}X-1.
\]
For completeness, let us verify this identification. Fix $\eta\in\Gamma$
and let $w(\eta)$ be the normalized right eigenvector of $M$ used in
Section~\ref{sec:general-data}, so that its last coordinate is $1$. The
last row of $Mw(\eta)=\eta w(\eta)$ gives $w_1(\eta)=\eta$, while the
$r$th row, for $1\leqslant r<L$, gives
\[
w_{r+1}(\eta)=\eta w_r(\eta)-h_r\eta.
\]
Induction therefore yields
\[
w_r(\eta)=\eta^r-h_1\eta^{r-1}-\cdots-h_{r-1}\eta
=P_r(\eta)+1.
\]
Since $\pi_s(e_L)$ has every coordinate equal to $1$, the $\eta$-coordinate
of $\pi_s(e_r-e_L)$ is $w_r(\eta)-1=P_r(\eta)$. Hence indeed
\[
\pi_s(e_r-e_L)=P_r(\Gamma).
\]
Thus $P_1(X)=X-1$, while
\[
P_L(X)=X^L-h_1X^{L-1}-\cdots-h_{L-1}X-1=P_h(X)
\]
is the characteristic polynomial introduced in
Section~\ref{sec:introduction}.

Since $P_L$ is separable, the roots $\gamma_1,\ldots,\gamma_d$ are
 distinct, so $V\neq0$.

We use $\Delta$ to express the two covolumes. More precisely, if
$F_1(\Gamma),\ldots,F_d(\Gamma)$ generate a full-rank lattice in $E_s$,
then its covolume is
\[
c_\Gamma\abs{\Delta(F_1,\ldots,F_d)},
\]
where $c_\Gamma>0$ depends only on the choice of real coordinates on
$E_s$, and not on the polynomials $F_1,\ldots,F_d$. Indeed, passing from
the full list of complex conjugates to the real coordinates
$\R^a\times\C^b$ amounts to one fixed real-linear change of coordinates.

Applying this to the two lattices gives
\[
\covol(\mathcal L)
=c_\Gamma\abs{\Delta(P_1,\ldots,P_d)},
\qquad
\covol(\mathcal T)
=c_\Gamma\abs{\Delta(Q_1,\ldots,Q_d)}.
\]
The common factor $c_\Gamma$ therefore cancels, so it remains only to
compute the two determinants.

\subsection{The covolume comparison and the toral volume}

We first compute the determinant associated with the arithmetic lattice $\mathcal T$.

\begin{lemma}\label{lem:Qdet}
One has
$$
 \abs{\Delta(Q_1,\ldots,Q_d)}=\varphi^d\abs{V}.
$$
\end{lemma}

\begin{proof}
Multiply the row corresponding to $\gamma_i$ by $\gamma_i^d$. Since
\[
\varphi\gamma_1\cdots\gamma_d=(-1)^{L+1},
\]
we have
\[
\abs{\gamma_1\cdots\gamma_d}=\varphi^{-1}.
\]
Thus these row multiplications change the absolute value of the
determinant by the factor $\varphi^{-d}$. The transformed columns are
the polynomials
\[
S_r(X):=X^dQ_r(X)
=X^{d-r}+g_2X^{d-r+1}+\cdots+g_rX^{d-1}
\]
for $1\leqslant r\leqslant d$.

Let $A$ be the coefficient matrix of $S_1,\ldots,S_d$ in the monomial
basis
\[
1,X,\ldots,X^{d-1},
\]
with the coefficient vector of $S_r$ as its $r$th column. Then $A$ is
anti-triangular with all anti-diagonal entries equal to $1$, so
$\det(A)=\pm1$. If
\[
\mathcal V:=\left(\gamma_i^{j-1}\right)_{1\leqslant i,j\leqslant d},
\]
then
\[
\left(S_j(\gamma_i)\right)_{1\leqslant i,j\leqslant d}
=\mathcal V A.
\]
Consequently,
\[
\abs{\Delta(S_1,\ldots,S_d)}
=\abs{\det(\mathcal V)}\abs{\det(A)}
=\abs{V}.
\]
Since
\[
\Delta(S_1,\ldots,S_d)
=\left(\gamma_1\cdots\gamma_d\right)^d
\Delta(Q_1,\ldots,Q_d),
\]
the claimed identity follows.
\end{proof}

In particular, the vectors
$Q_1(\Gamma),\ldots,Q_d(\Gamma)$ form a real basis of $E_s$.
Hence $\mathcal T$ is indeed a full-rank lattice in $E_s$.

It remains to measure the Rauzy lattice in the same coordinates. We next compute its corresponding determinant.

Let
$$
 P_L^*(X):=\frac{P_L(X)}{X-\varphi}.
$$
The contracting roots $\gamma_1,\ldots,\gamma_d$ are the roots of $P_L^*$. Write
$$
 P_L^*(X)=X^d+a_1X^{d-1}+a_2X^{d-2}+\cdots+a_d.
$$
A direct division gives
$$
 a_j=\varphi^j-h_1\varphi^{j-1}-h_2\varphi^{j-2}-\cdots-h_j\qquad \left(1\leqslant j\leqslant d\right),
$$
where the last coefficient also satisfies $a_d=\varphi^{-1}$.

\begin{lemma}\label{lem:Pdet}
One has
$$
 \abs{\Delta(P_1,\ldots,P_d)}=
 \frac{\varphi^d+\varphi^{d-1}+\cdots+\varphi+1}{\varphi}\abs{V}.
$$
\end{lemma}

\begin{proof}
All computations take place in the quotient ring
$$
 \C[X]/(P_L^*).
$$
Let $C_1,\ldots,C_d$ be the coefficient columns of $P_1,\ldots,P_d$ in the monomial basis
$$
 1,X,\ldots,X^{d-1},
$$
where $P_d$ is reduced modulo $P_L^*$. For $1\leqslant r<d$, the column $C_r$ is read directly from the polynomial $P_r$. The last column is obtained by using
$$
 X^d\equiv -a_1X^{d-1}-a_2X^{d-2}-\cdots-a_d\pmod{P_L^*}.
$$
The key elementary cancellation is
$$
 P_d+\varphi P_{d-1}+\varphi^2P_{d-2}+\cdots+\varphi^{d-1}P_1
 \equiv
 -\frac{\varphi^d+\varphi^{d-1}+\cdots+\varphi+1}{\varphi}
 \pmod{P_L^*}.
$$
Indeed, after substituting the above reduction of $X^d$, the coefficient of each positive power of $X$ cancels telescopically. The constant term is
$$
 -a_d-1-\varphi-\varphi^2-\cdots-\varphi^{d-1}
 =-\frac{1}{\varphi}-1-\varphi-\cdots-\varphi^{d-1},
$$
which is the displayed constant.

Therefore the column operation
$$
 C_d\longleftarrow C_d+\varphi C_{d-1}+\varphi^2C_{d-2}+\cdots+\varphi^{d-1}C_1
$$
changes only the last column and replaces it by a column whose only nonzero entry is the constant entry
$$
 -\frac{\varphi^d+\varphi^{d-1}+\cdots+\varphi+1}{\varphi}.
$$
The first $d-1$ columns, restricted to the rows corresponding to $X,X^2,\ldots,X^{d-1}$, form an upper triangular matrix with diagonal entries equal to $1$. Expanding the determinant along the modified last column gives
$$
 \det(C_1,\ldots,C_d)=(-1)^d\frac{\varphi^d+\varphi^{d-1}+\cdots+\varphi+1}{\varphi}.
$$
Since
$$
 \Delta(P_1,\ldots,P_d)=\det(C_1,\ldots,C_d)V,
$$
the result follows after taking absolute values.
\end{proof}

We can now prove Proposition~\ref{prop:covol-ratio}.

\begin{proof}[Proof of Proposition~\ref{prop:covol-ratio}]
By Lemmas~\ref{lem:Qdet} and~\ref{lem:Pdet},
$$
 \frac{\covol(\mathcal L)}{\covol(\mathcal T)}
 =
 \frac{\abs{\Delta(P_1,\ldots,P_d)}}{\abs{\Delta(Q_1,\ldots,Q_d)}}
 =
 \frac{\varphi^d+\varphi^{d-1}+\cdots+\varphi+1}{\varphi^{d+1}}.
$$
Since $d=L-1$, this is
$$
 \varphi^{-1}+\varphi^{-2}+\cdots+\varphi^{-L}=\omega+\omega^2+\cdots+\omega^L.
$$
\end{proof}

The determinant comparison fixes the relative covolumes of the two lattices. Next, we normalize the arithmetic lattice to the standard integer lattice. Since $Q_1(\Gamma),\ldots,Q_d(\Gamma)$ form a real basis of $E_s$, there is a unique real-linear map
$$
 \tau:E_s\longrightarrow \R^d
$$
determined by
$$
 \tau\left(Q_r(\Gamma)\right)=e_r\qquad \left(1\leqslant r\leqslant d\right),
$$
where $e_1,\ldots,e_d$ is the standard basis of $\R^d$. Equivalently,
$$
 \mathcal T=\tau^{-1}\left(\Z^d\right).
$$
Therefore, for every measurable $A\subset E_s$,
$$
 \vol(\tau(A))=\frac{\lambda_{L-1}(A)}{\covol(\mathcal T)}.
$$

In these coordinates, the desired covolume identity becomes a unit-volume statement for the full window.

\begin{proposition}\label{prop:tau-volume}
One has
$$
 \vol(\tau(\widetilde R))=1.
$$
\end{proposition}

\begin{proof}
By Corollary~\ref{cor:self-affine-volume},
$$
 \vol(R_i)=\omega^i\vol(\widetilde R).
$$
Since $R$ is a fundamental domain for $\mathcal L$,
$$
 \covol(\mathcal L)=\vol(R)=\sum_{i=1}^{L}\omega^i\vol(\widetilde R).
$$
The covolume ratio gives
$$
 \covol(\mathcal L)=\left(\sum_{i=1}^{L}\omega^i\right)\covol(\mathcal T).
$$
Canceling $\sum_{i=1}^{L}\omega^i>0$ yields
$$
 \vol(\widetilde R)=\covol(\mathcal T).
$$
Since $\mathcal T=\tau^{-1}\left(\Z^d\right)$, this is equivalent to $\vol(\tau(\widetilde R))=1$.
\end{proof}

\subsection{The arithmetic congruence modulo $\mathcal T$}

The preceding subsection shows that $\widetilde R$ has exactly the
volume of one fundamental domain for $\mathcal T$. This identifies
$\mathcal T$ as a lattice of the correct geometric size, but it does
not yet explain why it is the right lattice for the arithmetic
parameter $n$. It remains to establish this second property.

The key point is first checked for a single weight $H_k$: the
contracting vector attached to the $k$th digit position differs from
$H_ks_0$ by an element of $\mathcal T$. Once this is known, summing
these relations with the digits of the canonical infinite
$H$-expansion turns the digit-dependent quantity $\evalG(n)$ into the
linear quantity $ns_0$ modulo $\mathcal T$.

The toral congruence is purely algebraic: it uses only the recurrence
for $H_k$, the relation $P_L(\gamma)=0$, and the definition of
$\mathcal T$.

\begin{lemma}\label{lem:single-weight-congruence}
For every $k\geqslant 1$, one has
$$
 \left(\eta_j^{k-1}\right)_j-H_ks_0\in\mathcal T.
$$
\end{lemma}

\begin{proof}
For $k\geqslant1$, put
\[
D_k(X):=X^{k-1}-H_k.
\]
We prove that the vector
\[
D_k(\Gamma):=\left(D_k(\eta_j)\right)_j
\]
belongs to $\mathcal T$.

For $k=1$, this is immediate because $H_1=1$. Now let
$k=r+1$ with $1\leqslant r\leqslant L-1$. For every contracting root
$\gamma$, dividing $P_L(\gamma)=0$ by $\gamma^{L-r}$ gives
\[
\gamma^r-h_1\gamma^{r-1}-\cdots-h_r
=
h_{r+1}\gamma^{-1}+\cdots+h_L\gamma^{-(L-r)}
=
Q_{L-r}(\gamma).
\]
For $1\leqslant q\leqslant r$, notice that
\[
h_q\gamma^{r-q}
=
h_qD_{r+1-q}(\gamma)+h_qH_{r+1-q}.
\]
Using the initial recurrence
\[
H_{r+1}=\sum_{q=1}^{r}h_qH_{r+1-q},
\]
we therefore obtain
\[
D_{r+1}(\gamma)
=
\sum_{q=1}^{r}h_qD_{r+1-q}(\gamma)
+
Q_{L-r}(\gamma).
\]
Applying this identity simultaneously to all contracting conjugates
gives
\[
D_{r+1}(\Gamma)
=
\sum_{q=1}^{r}h_qD_{r+1-q}(\Gamma)
+
Q_{L-r}(\Gamma).
\]
Induction on $r$ now gives
$D_k(\Gamma)\in\mathcal T$ for $1\leqslant k\leqslant L$.

For $k>L$, each sequence $(\eta_j^{k-1})_{k\geqslant1}$ and the
sequence $(H_k)_{k\geqslant1}$ satisfy the same order-$L$ recurrence.
Hence the vectors $D_k(\Gamma)$ satisfy
\[
D_{n+L}(\Gamma)
=
h_1D_{n+L-1}(\Gamma)+\cdots+h_LD_n(\Gamma)
\qquad (n\geqslant1).
\]
Since the first $L$ vectors lie in the lattice $\mathcal T$, all later
ones do as well. This proves the claim.
\end{proof}

The preceding lemma gives the required congruence separately for each
digit position. Since the canonical infinite $H$-expansion has finite
support, we may now sum these relations with the corresponding digit
coefficients.

\begin{proposition}\label{prop:toral-congruence}
For every $n\in\N_0$,
\[
\evalG(n)\equiv ns_0\pmod{\mathcal T}.
\]
\end{proposition}

\begin{proof}
Let $(a_k)_{k\geqslant1}$ be the digit sequence such that
\[
[n]_H^{-1}=(a_k)_{k\geqslant1}.
\]
By definition,
\[
\evalG(n)-ns_0
=\sum_{k\geqslant1}a_k
\left(\left(\eta_j^{k-1}\right)_j-H_ks_0\right).
\]
The sum is finite, the coefficients $a_k$ are integers, and each summand lies in $\mathcal T$ by the preceding lemma. Therefore the whole difference lies in $\mathcal T$.
\end{proof}

We express this congruence in standard torus coordinates. Recall
that the linear isomorphism
\[
\tau:E_s\longrightarrow\R^d
\]
sends $\mathcal T$ onto $\Z^d$. Put
\[
\mathbf v:=\tau(s_0)=(v_1,\ldots,v_d).
\]
Applying $\tau$ to Proposition~\ref{prop:toral-congruence} gives
\[
\tau(\evalG(n))\equiv n\mathbf v\pmod{\Z^d}.
\]
Thus, after quotienting by $\mathcal T$, all dependence on the digits
of $n$ disappears: the point represented by $\evalG(n)$ is simply the
$n$th point of the linear orbit generated by $\mathbf v$ on
$\R^d/\Z^d$. This is the arithmetic linearization that will allow us
to replace integer evaluations, and later prime evaluations, by a
toral orbit.

Two properties remain to be checked before this coding can be used for
primes.

\subsection{Rational independence and almost-everywhere injectivity}

The next section will apply Vinogradov's theorem through Weyl's
criterion to the prime orbit $p\mathbf v$ modulo $\Z^d$. The required
nondegeneracy is that every nonzero integer frequency
$a\in\Z^d$ gives an irrational number $a\cdot\mathbf v$; equivalently,
$1,v_1,\ldots,v_d$ must be linearly independent over $\Q$.

There is a second issue. Passing from $\widetilde R$ to the quotient
torus must not lose the terminal-block information on a set of
positive measure. We therefore also need the projection
$\widetilde R\to E_s/\mathcal T$ to be one-to-one almost everywhere.
Then a toral point determines, outside a null set, a unique point of
$\widetilde R$ and hence the cylinder window to which it belongs.

These two facts are obtained from the same comparison. On the one hand,
Proposition~\ref{prop:full-equidistribution} describes the
limiting distribution of the integer evaluations inside
$\widetilde R$. On the other hand, the toral congruence identifies
their images with the orbit $n\mathbf v$. Comparing these two
descriptions first forces rational independence of $\mathbf v$, and
then, using $\vol(\tau(\widetilde R))=1$, forces multiplicity one
almost everywhere.

\begin{proposition}\label{prop:rational-independence}
The numbers
$$
 1,v_1,
 \ldots,v_d
$$
are linearly independent over $\Q$.
\end{proposition}

\begin{proof}
Set
\[
A:=\tau(\widetilde R)\subset\R^d
\]
and let
\[
q:\R^d\longrightarrow\R^d/\Z^d
\]
be the quotient map. By Proposition~\ref{prop:tau-volume},
$\vol(A)=1$. Proposition~\ref{prop:full-equidistribution} and
the continuity of $q\circ\tau$ imply that
\[
q\!\left(\tau(\evalG(n))\right)
\]
is equidistributed with respect to the pushforward $\nu$ of Lebesgue
measure on $A$ under $q$. By the toral congruence, this is the same
sequence as $n\mathbf v\pmod{\Z^d}$.

Suppose, to the contrary, that $1,v_1,\ldots,v_d$ satisfy a nontrivial
rational relation. Clearing denominators, there exist $a_0\in\Z$ and
a nonzero vector $a\in\Z^d$ such that
\[
a_0+a\cdot\mathbf v=0.
\]
Then, for every $n\geqslant0$,
\[
a\cdot(n\mathbf v)=-na_0\in\Z,
\]
so the entire toral orbit is contained in the proper closed subgroup
\[
H:=\left\{y\in\R^d/\Z^d:a\cdot y=0\pmod1\right\}.
\]
Consequently, its limiting measure $\nu$ is supported on $H$.

On the other hand, the pushforward of Lebesgue measure on $A$ under
$q$ is absolutely continuous with respect to Haar measure on
$\R^d/\Z^d$: its density is the periodized indicator
\[
m_A(y):=\sum_{z\in\Z^d}\mathbf 1_A(\widetilde y+z),
\]
where $\widetilde y$ is any representative of $y$. Since $H$ has Haar
measure zero, this gives $\nu(H)=0$, contradicting the fact that
$\nu$ is supported on $H$. Therefore $1,v_1,\ldots,v_d$ are linearly
independent over $\Q$.
\end{proof}

We have therefore proved that the linear orbit
$n\mathbf v\pmod{\Z^d}$ is Haar equidistributed. We next compare this
Haar distribution with the distribution obtained by projecting
Lebesgue measure from $\tau(\widetilde R)$. Since
$\vol(\tau(\widetilde R))=1$, equality of the two measures will force
almost every torus point to have exactly one representative in
$\tau(\widetilde R)$.

\begin{proposition}\label{prop:ae-injectivity}
The projection
$$
 \widetilde R\longrightarrow E_s/\mathcal T
$$
is one-to-one almost everywhere.
\end{proposition}

\begin{proof}
Put again $A:=\tau(\widetilde R)$ and let
$q:\R^d\to\R^d/\Z^d$ be the quotient map. By
Proposition~\ref{prop:rational-independence},
the numbers $1,v_1,\ldots,v_d$ are linearly independent over $\Q$.
Weyl's criterion \cite[Chapter~1]{KuipersNiederreiter1974} therefore
gives Haar equidistribution of
\[
n\mathbf v\pmod{\Z^d}.
\]
By Proposition~\ref{prop:toral-congruence}, the same is true for
$q(\tau(\evalG(n)))$.

On the other hand, Proposition~\ref{prop:full-equidistribution}
shows that this projected sequence is equidistributed with respect to
the pushforward of normalized Lebesgue measure on $A$. Hence that
pushforward is Haar measure.

As in the preceding proof, its density with respect to Haar measure is
the multiplicity function
\[
m(y):=\#\left(A\cap q^{-1}(y)\right).
\]
Proposition~\ref{prop:tau-volume} gives $\vol(A)=1$, so no
additional normalization factor occurs. Since the pushforward measure
is Haar measure, we obtain
\[
m(y)=1
\]
for almost every $y\in\R^d/\Z^d$. Hence the set of torus points with
more than one representative in $A$ has Haar measure zero. Its
preimage in $A$ also has Lebesgue measure zero, because the pushforward
of Lebesgue measure on $A$ is Haar measure. After removing this null
set, $q$ is one-to-one on $A$. Transporting the statement back through
the linear isomorphism $\tau$ proves that
\[
\widetilde R\longrightarrow E_s/\mathcal T
\]
is one-to-one almost everywhere.
\end{proof}

\subsection{A general sampling principle}

The preceding arguments do not depend on sampling the orbit at every
integer. Once the toral orbit attached to a sequence of indices is Haar
equidistributed, the same coding transfers that equidistribution to the
full window and hence to every terminal cylinder. We record this
principle because it separates the geometric part of the paper from the
particular arithmetic sequence used to sample it.

To use this toral model along subsequences, we first isolate the small
exceptional set on which the coding may fail to be one-to-one.
Let
\[
\operatorname{rem}:\R^d\longrightarrow\R^d/\Z^d
\]
be the quotient map and set
\[
F:=\operatorname{rem}\circ\tau:
\widetilde R\longrightarrow\R^d/\Z^d.
\]

\begin{lemma}\label{lem:regular-toral-coding}
Suppose Assumption~\ref{ass:pisot} holds. Then
$F(\widetilde R)=\R^d/\Z^d$. Moreover,
\[
E:=F(\partial\widetilde R)
\]
is a compact Haar-null set, and every point of
$(\R^d/\Z^d)\setminus E$ has a unique representative in
$\widetilde R$. The inverse representative map is continuous on
$(\R^d/\Z^d)\setminus E$.
\end{lemma}

\begin{proof}
By Proposition~\ref{prop:terminal-cylinder-window} with
$B=\varepsilon$, the boundary $\partial\widetilde R$ has Lebesgue
measure zero. Since $\widetilde R$ is compact and $F$ is Lipschitz, the
set $E=F(\partial\widetilde R)$ is compact and Haar-null.

Propositions~\ref{prop:full-equidistribution},
\ref{prop:toral-congruence}, and
\ref{prop:rational-independence} show that the pushforward
of normalized Lebesgue measure on $\widetilde R$ under $F$ is Haar
measure on $\R^d/\Z^d$; this is also the measure comparison used in the
proof of Proposition~\ref{prop:ae-injectivity}. Consequently,
$F(\widetilde R)$ has full Haar measure. It is compact, hence closed, so
its complement is an open set of Haar measure zero and must be empty.
Thus $F$ is onto.

Now let $y\notin E$. Any representative of $y$ in $\widetilde R$ lies
in the interior of $\widetilde R$. If $y$ had two distinct
representatives $x,x'\in\operatorname{int}(\widetilde R)$, then, since
$F$ is locally a homeomorphism, a whole neighborhood of $y$ would have
at least two representatives in $\widetilde R$. This would contradict
the almost-everywhere injectivity of
Proposition~\ref{prop:ae-injectivity}. Hence the representative
is unique.

Finally, let $y\notin E$ and let $x$ be its unique representative.
Since $F=\operatorname{rem}\circ\tau$ is a local homeomorphism, there
is an open neighborhood $U\subset\operatorname{int}(\widetilde R)$ of
$x$ such that $F|_U$ is a homeomorphism onto an open neighborhood $V$ of
$y$. As $E$ is closed and $y\notin E$, we may shrink $V$, and $U$
accordingly, so that $V\cap E=\varnothing$. Every point of $V$ then has
a unique representative in $\widetilde R$, so the global inverse
representative map agrees on $V$ with $(F|_U)^{-1}$. It is therefore
continuous at $y$.
\end{proof}

The regular coding lemma immediately yields the following transfer
principle.

\begin{proposition}
\label{prop:toral-sampling-principle}
Suppose Assumption~\ref{ass:pisot} holds. Let
$(K_n)_{n\geqslant1}$ be a sequence in $\N_0$ such that
\[
K_n\mathbf v\pmod{\Z^d}
\]
is Haar equidistributed in $\R^d/\Z^d$. Then
$\bigl(\evalG(K_n)\bigr)_{n\geqslant1}$ is equidistributed in
$\widetilde R$ with respect to normalized Lebesgue measure. Moreover,
for every fixed $e\geqslant1$ and every $B\in\mathcal B_e$,
\[
\frac1N
\#\left\{1\leqslant n\leqslant N:\TBe{e}(K_n)=B\right\}
\longrightarrow
\omega^{\len(B)}.
\]
\end{proposition}

\begin{proof}
Let $F=\operatorname{rem}\circ\tau$ and let
$E=F(\partial\widetilde R)$ be as in
Lemma~\ref{lem:regular-toral-coding}. For
$f\in C(\widetilde R)$, define $\widetilde f$ on the torus by
\[
\widetilde f(y):=f(x_y)\qquad(y\notin E),
\]
where $x_y$ is the unique representative of $y$ in $\widetilde R$, and
put $\widetilde f=0$ on $E$. By the lemma, $\widetilde f$ is bounded and
continuous outside the compact Haar-null set $E$, hence is Riemann
integrable. Proposition~\ref{prop:toral-congruence} gives
\[
F(\evalG(K_n))=K_n\mathbf v\pmod{\Z^d}.
\]
Since an equidistributed sequence visits the Haar-continuity set $E$
with density zero, the averages of $f(\evalG(K_n))$ and
$\widetilde f(K_n\mathbf v)$ differ by $o(1)$. Equidistribution on the
torus and the measure identification in
Lemma~\ref{lem:regular-toral-coding} therefore give
\[
\frac1N\sum_{n\leqslant N}f(\evalG(K_n))
\longrightarrow
\frac1{\vol(\widetilde R)}\int_{\widetilde R}f(x)\,dx.
\]
This proves the first assertion.

For the terminal-block statement, fix $e$ and $B$. By
Proposition~\ref{prop:terminal-cylinder-partition}, the cylinder
windows $R_C$, $C\in\mathcal B_e$, form a measurable partition of
$\widetilde R$ up to null overlaps, while
Proposition~\ref{prop:terminal-cylinder-window} gives null boundary
for every $R_C$. Hence the possible ambiguity between the integer class
$N_B$ and membership in $R_B$ is contained in a finite union of
cylinder boundaries and pairwise intersections. This exceptional set is
compact and Lebesgue-null. The first part therefore gives
\[
\frac1N
\#\{n\leqslant N:\TBe{e}(K_n)=B\}
\longrightarrow
\frac{\vol(R_B)}{\vol(\widetilde R)}
=\omega^{\len(B)}.
\]
\end{proof}

The first application is polynomial sampling. Unlike prime sampling,
this uses only the classical Weyl theorem.

\begin{corollary}
\label{cor:polynomial-terminal-sampling}
Suppose Assumption~\ref{ass:pisot} holds. Let
$P\in\Z[t]$ be nonconstant and suppose that $P(n)\in\N_0$ for all
sufficiently large $n$. Choose $n_0$ such that this holds for every
$n\geqslant n_0$. Then, for every fixed $e\geqslant1$ and every
$B\in\mathcal B_e$,
\[
\frac1N
\#\left\{n_0\leqslant n\leqslant N:\TBe{e}(P(n))=B\right\}
\longrightarrow
\omega^{\len(B)}.
\]
\end{corollary}

\begin{proof}
For every nonzero $\mathbf m\in\Z^d$, Proposition~\ref{prop:rational-independence}
shows that $\mathbf m\cdot\mathbf v$ is irrational. Hence the
nonconstant polynomial
\[
(\mathbf m\cdot\mathbf v)P(t)
\]
has an irrational nonconstant coefficient. Weyl's polynomial
equidistribution theorem therefore gives
\[
P(n)\mathbf v\pmod{\Z^d}
\]
Haar equidistributed on the torus. The tail sequence
$K_n:=P(n+n_0-1)$ takes values in $\N_0$ and remains Haar
equidistributed modulo the torus, so
Proposition~\ref{prop:toral-sampling-principle} applies. Reindexing this
tail proves the stated limit.
\end{proof}

The same argument on a product torus yields joint independence for
polynomial samples that are algebraically independent at the elementary
linear level needed by Weyl's criterion.

\begin{corollary}
\label{cor:joint-polynomial-terminal-sampling}
Suppose Assumption~\ref{ass:pisot} holds. Let
$P_1,\ldots,P_r\in\Z[t]$ be such that each $P_j(n)$ is non-negative for
all sufficiently large $n$, and choose $n_0$ so that this holds for every
$j$ and every $n\geqslant n_0$. Suppose that the nonconstant polynomials
\[
P_1-P_1(0),\ldots,P_r-P_r(0)
\]
are linearly independent over $\Q$. For fixed
$e_1,\ldots,e_r\geqslant1$ and $B_j\in\mathcal B_{e_j}$,
\[
\frac1N\#\left\{n_0\leqslant n\leqslant N:
\TBe{e_j}(P_j(n))=B_j\ \text{for all }1\leqslant j\leqslant r
\right\}
\longrightarrow
\prod_{j=1}^r\omega^{\len(B_j)}.
\]
\end{corollary}

\begin{proof}
Consider the orbit
\[
n\longmapsto
\bigl(P_1(n)\mathbf v,\ldots,P_r(n)\mathbf v\bigr)
\pmod{(\Z^d)^r}.
\]
For a nonzero frequency
$(\mathbf m_1,\ldots,\mathbf m_r)\in(\Z^d)^r$, the corresponding
scalar phase is
\[
Q(t):=\sum_{j=1}^r(\mathbf m_j\cdot\mathbf v)P_j(t).
\]
Write
\[
P_j(t)-P_j(0)=\sum_{k\geqslant1}c_{j,k}t^k,
\qquad c_{j,k}\in\Z.
\]
The coefficient of $t^k$ in the nonconstant part of $Q$ is
\[
\left(\sum_{j=1}^r c_{j,k}\mathbf m_j\right)\cdot\mathbf v.
\]
If all the vectors $\sum_j c_{j,k}\mathbf m_j$ were zero, then,
after taking any fixed coordinate in $\Z^d$, the resulting coefficients
would give a $\Q$-linear relation among
$P_1-P_1(0),\ldots,P_r-P_r(0)$. Their assumed independence would force
every coordinate of every $\mathbf m_j$ to vanish, contrary to the
choice of a nonzero frequency. Hence, for some $k$, the vector
\[
\mathbf a:=\sum_{j=1}^r c_{j,k}\mathbf m_j
\]
is nonzero. Proposition~\ref{prop:rational-independence}
then makes the corresponding coefficient $\mathbf a\cdot\mathbf v$
irrational. Weyl's theorem therefore gives Haar equidistribution on the
product torus. Applying the proof of
Proposition~\ref{prop:toral-sampling-principle} coordinatewise and then
to the product of the cylinder windows gives the asserted product
measure.
\end{proof}

In the Zeckendorf system, Corollary~\ref{cor:polynomial-terminal-sampling}
settles the polynomial terminal-block question suggested by numerical
experiments in \cite[Section~8]{Chang2025TerminalDigits}. Earlier work
of Drmota and Steiner \cite{DrmotaSteiner2002} established strong
limit laws for Zeckendorf digit statistics along polynomial sequences;
the present statement concerns the local terminal-cylinder law instead.

We have now completed the reduction from terminal blocks to a linear
toral orbit and isolated the resulting transfer mechanism. The
covolume comparison shows that the full window $\widetilde R$ has the
volume of one fundamental domain for $\mathcal T$; the arithmetic
congruence identifies $\evalG(n)$ with $n\mathbf v$ on the standard
torus; and the final results show that any equidistributed sampling of
this toral orbit inherits the cylinder-volume law. The next section
specializes the principle to the primes and derives its further
prime-specific consequences.

%% file: sections/07-prime-terminal-blocks.tex
\section{Prime sampling and terminal blocks}\label{sec:prime-blocks}

Throughout this section, we work under
Assumption~\ref{ass:pisot}, and $p$ denotes a prime number.
Theorem~\ref{thm:rauzy-fundamental-domain}, proved in
Section~\ref{sec:rauzy-verification}, supplies the geometric input.

Two prime-sampling problems are governed by the Rauzy geometry developed
above. For terminal blocks, the toral sampling principle of
Section~\ref{sec:volume} reduces the problem to the orbit
\[
p\mathbf v\;(\mathrm{mod}\ \Z^d),
\qquad
\mathbf v=\tau(s_0),
\]
since \eqref{eq:toral-congruence-goal} gives
\[
\tau(\evalG(n))\equiv n\mathbf v\;(\mathrm{mod}\ \Z^d).
\]
The Rauzy torus of Section~\ref{sec:rauzy} also codes the substitution
fixed word itself through the orbit generated by $s_0$. For these two
basic prime-sampling statements, the only additional analytic input is
Vinogradov's equidistribution theorem for irrational multiples of the
primes. We first establish prime equidistribution on the relevant tori,
then derive prime genericity for the substitution fixed word, and finally
return to terminal blocks and their consequences.

\subsection{Prime equidistribution on the torus}

We begin with the toral orbit itself.

\begin{lemma}\label{lem:prime-toral-equidistribution}
The sequence
\[
p\mathbf v\;(\mathrm{mod}\ \Z^d)
\]
is equidistributed in $\R^d/\Z^d$ with respect to Haar measure.
\end{lemma}

\begin{proof}
By Proposition~\ref{prop:rational-independence}, the
numbers
\[
1,v_1,\ldots,v_d
\]
are linearly independent over $\Q$. Hence, for every nonzero
$\mathbf m\in\Z^d$, the number $\mathbf m\cdot\mathbf v$ is irrational.
Vinogradov's theorem \cite{Vinogradov1954} gives
\[
\frac{1}{\pi(X)}
\sum_{p\leqslant X}
 e^{2\pi i p\,\mathbf m\cdot\mathbf v}
\longrightarrow 0
\qquad (X\to+\infty).
\]
The multidimensional Weyl criterion therefore gives the asserted equidistribution on the torus.
\end{proof}

\subsection{Prime sampling of the substitution fixed word}

The same Rauzy torus also gives a prime-number theorem that does not
involve terminal blocks. Recall that
\[
A_\infty=u_0u_1u_2\cdots
\]
is the fixed word of the substitution, that $p_m$ is the letter-count
vector of the prefix $u_0\cdots u_{m-1}$, and that
\[
\pi_s(p_m)\equiv ms_0\;(\mathrm{mod}\ \mathcal L)
\]
by Lemma~\ref{lem:prefix-toral-congruence}. Put
\[
a:=q_{\mathcal L}(s_0)\in \mathbb T_{\mathcal L}.
\]
Proposition~\ref{prop:torus-equidistribution} says that the orbit $(ma)_{m\geqslant0}$
is Haar equidistributed. Notice that this codes the original position $m$ in
$A_\infty$ directly: it does not pass through the expanded-word position
$c_m$ from Theorem~\ref{thm:fixed-word-bridge}. The same toral orbit
can therefore be sampled directly at prime indices.

\begin{lemma}\label{lem:prime-rauzy-torus}
Suppose Assumption~\ref{ass:pisot} holds. The sequence
\[
pa\qquad (p\ \mathrm{prime})
\]
is Haar equidistributed in $\mathbb T_{\mathcal L}$.
\end{lemma}

\begin{proof}
Let $\chi$ be a nontrivial character of the torus
$\mathbb T_{\mathcal L}$. Since the orbit $(ma)_{m\geqslant0}$ is dense by
Proposition~\ref{prop:torus-equidistribution}, its image
$(\chi(a)^m)_{m\geqslant0}$ is dense in the image of $\chi$. A nontrivial
character of a connected torus has image $\mathbb S^1$, so if
\[
\chi(a)=e^{2\pi i\alpha},
\]
then $\alpha$ is irrational. Vinogradov's theorem
\cite{Vinogradov1954} therefore gives
\[
\frac1{\pi(X)}\sum_{p\leqslant X}\chi(pa)
=
\frac1{\pi(X)}\sum_{p\leqslant X}e^{2\pi i p\alpha}
\longrightarrow0.
\]
The Weyl criterion on compact tori
\cite[Chapter~1]{KuipersNiederreiter1974} now proves the claim.
\end{proof}

For a finite word $W=A_{i_0}\cdots A_{i_{r-1}}$ over $\mathcal A$, write
\[
\freq(W)
:=
\lim_{N\to\infty}
\frac1N
\#\left\{0\leqslant m<N:
 u_mu_{m+1}\cdots u_{m+r-1}=W\right\},
\]
whenever the limit exists. We also write $\Lambda$ for the von Mangoldt
function,
\[
\Lambda(n):=
\begin{cases}
\log p,& n=p^k\ \text{for some prime }p\text{ and }k\geqslant1,\\
0,& \text{otherwise}.
\end{cases}
\]
The next theorem shows at the same time that these factor frequencies
exist and that prime starting positions see the same frequencies.

\begin{theorem}
\label{thm:prime-fixed-word}
Suppose Assumption~\ref{ass:pisot} holds. For every
finite word $W=A_{i_0}\cdots A_{i_{r-1}}$ over $\mathcal A$, the frequency
$\freq(W)$ exists and
\[
\frac1{\pi(X)}
\#\left\{p\leqslant X:
 u_pu_{p+1}\cdots u_{p+r-1}=W\right\}
\longrightarrow
\freq(W).
\]
Moreover,
\[
\frac1X
\sum_{\substack{n\leqslant X\\
 u_nu_{n+1}\cdots u_{n+r-1}=W}}
\Lambda(n)
\longrightarrow
\freq(W).
\]
\end{theorem}

\begin{proof}
Set
\[
E_\sigma
:=q_{\mathcal L}\!\left(
\partial R\cup\bigcup_{i=1}^L\partial R_i
\right)
\subset\mathbb T_{\mathcal L}.
\]
By Theorem~\ref{thm:rauzy-fundamental-domain}, the set inside
parentheses is compact and Lebesgue-null. Since the quotient map is locally
a linear isometry, $E_\sigma$ is compact and Haar-null.

For $y\notin E_\sigma$, let $x_y\in R\setminus\partial R$ be the unique
representative of $y$. Since $y$ also avoids the images of all
$\partial R_i$, the point $x_y$ belongs to exactly one of the sets
$R_i\setminus\partial R_i$. Define
\[
g_i(y):=
\begin{cases}
1,& y\notin E_\sigma\ \text{and }x_y\in R_i,\\
0,& \text{otherwise}.
\end{cases}
\]
The inverse representative map is continuous away from
$q_{\mathcal L}(\partial R)$, as in the proof of
Proposition~\ref{prop:prefix-rauzy-equidistribution}. Hence each $g_i$ is
continuous outside $E_\sigma$ and is therefore Riemann integrable.

The proof of Proposition~\ref{prop:base-class-equidistribution}, together
with Lemma~\ref{lem:prefix-toral-congruence}, gives
\[
\mathbf 1_{\{u_m=A_i\}}
=
g_i(ma)
\]
whenever $ma\notin E_\sigma$. Now let
$W=A_{i_0}\cdots A_{i_{r-1}}$ and define
\[
G_W(y):=\prod_{j=0}^{r-1}g_{i_j}(y+ja).
\]
This is Riemann integrable, since its discontinuity set is contained in
\[
\bigcup_{j=0}^{r-1}(E_\sigma-ja),
\]
a finite union of Haar-null sets. Outside this exceptional set one has
\[
G_W(ma)
=
\mathbf 1_{\{u_mu_{m+1}\cdots u_{m+r-1}=W\}}.
\]
Proposition~\ref{prop:torus-equidistribution} shows that the exceptional indices
$m$ have density zero. Therefore
\[
\freq(W)
=
\int_{\mathbb T_{\mathcal L}}G_W\,d\nu_{\mathcal L}.
\]
In particular, the frequency exists.

Lemma~\ref{lem:prime-rauzy-torus} gives the same equidistribution along
the primes. The prime indices for which $pa$ belongs to the exceptional
set again have relative density zero, and hence
\[
\frac1{\pi(X)}
\#\left\{p\leqslant X:
 u_pu_{p+1}\cdots u_{p+r-1}=W\right\}
\longrightarrow
\int_{\mathbb T_{\mathcal L}}G_W\,d\nu_{\mathcal L}
=
\freq(W).
\]

For the von Mangoldt formulation, let $\delta=\freq(W)$.
The first part gives
\[
\#\{p\leqslant X:u_p\cdots u_{p+r-1}=W\}
=\delta\pi(X)+o(\pi(X)).
\]
Partial summation and the prime number theorem therefore yield
\[
\sum_{\substack{p\leqslant X\\u_p\cdots u_{p+r-1}=W}}\log p
=\delta X+o(X).
\]
The total von Mangoldt weight of prime powers $p^k\leqslant X$ with
$k\geqslant2$ is $o(X)$, so adjoining them proves the final assertion.
\end{proof}

For single letters the limiting constants are explicit. The normalized
Perron--Frobenius frequency vector of the primitive substitution satisfies
$f_{i+1}=\omega f_i$ for $1\leqslant i<L$. Hence
\[
\freq(A_i)
=
\frac{\omega^{i-1}}{1+\omega+\cdots+\omega^{L-1}}
=
\frac{\omega^i}{\omega+\omega^2+\cdots+\omega^L}.
\]
Consequently,
\[
\#\{p\leqslant X:u_p=A_i\}
\sim
\frac{\omega^i}{\omega+\omega^2+\cdots+\omega^L}\,\pi(X).
\]
For the Tribonacci substitution
\[
A_1\mapsto A_1A_2,\qquad
A_2\mapsto A_1A_3,\qquad
A_3\mapsto A_1,
\]
one has $h_1=h_2=h_3=1$. If $\beta$ denotes the Tribonacci constant,
then $\omega=\beta^{-1}$ and
$\omega+\omega^2+\omega^3=1$, so
\[
\#\{p\leqslant X:u_p=A_i\}
\sim \beta^{-i}\pi(X)
\qquad (i=1,2,3).
\]

\emph{In particular, this proves the prime-number theorem for the
Tribonacci fixed word suggested in
\cite[Section~9.2]{DrmotaMuellnerSpiegelhofer2025}.}

The same toral coding also gives polynomial sampling of the fixed word
with essentially no additional geometric input, including polynomial
values at prime arguments.

\begin{corollary}
\label{cor:polynomial-fixed-word-sampling}
Suppose Assumption~\ref{ass:pisot} holds. Let
$P\in\Z[t]$ be nonconstant and suppose that $P(n)\in\N_0$ for all
sufficiently large $n$. Choose $n_0$ such that this holds for every
$n\geqslant n_0$. Then, for every finite word
$W=A_{i_0}\cdots A_{i_{r-1}}$ over $\mathcal A$,
\[
\frac1N
\#\left\{n_0\leqslant n\leqslant N:
 u_{P(n)}u_{P(n)+1}\cdots u_{P(n)+r-1}=W\right\}
\longrightarrow
\freq(W).
\]
Moreover,
\[
\frac1{\pi(X)}
\#\left\{n_0\leqslant p\leqslant X:
 u_{P(p)}u_{P(p)+1}\cdots u_{P(p)+r-1}=W\right\}
\longrightarrow
\freq(W).
\]
\end{corollary}

\begin{proof}
Let $\chi$ be a nontrivial character of $\mathbb T_{\mathcal L}$ and
write $\chi(a)=e^{2\pi i\alpha}$. As in the proof of
Lemma~\ref{lem:prime-rauzy-torus}, one has $\alpha\notin\Q$. Hence the
nonconstant polynomial $\alpha P(t)$ has an irrational nonconstant
coefficient. Weyl's polynomial equidistribution theorem gives Haar
equidistribution of $P(n)a$ on $\mathbb T_{\mathcal L}$, while Rhin's
theorem \cite{Rhin1973} gives the same conclusion for $P(p)a$ with $p$
prime. Removing the finitely many integer or prime arguments below $n_0$
does not affect either limit. Applying the Riemann-integrable coding
function $G_W$ from the proof of Theorem~\ref{thm:prime-fixed-word} gives
both limits.
\end{proof}

\subsection{Prime terminal blocks and their distribution}

We return from the toral orbit to terminal block words. The general
sampling principle of Proposition~\ref{prop:toral-sampling-principle}
already packages the passage from toral equidistribution to the
terminal-cylinder law, so only the prime equidistribution input remains.

\begin{theorem}\label{thm:prime-terminal-blocks}
Suppose Assumption~\ref{ass:pisot} holds. Let
$e\geqslant1$ and let $B\in\mathcal B_e$. Then
\[
\#\left\{p\leqslant X:\TBe{e}(p)=B\right\}
\sim \omega^{\len(B)}\pi(X).
\]
More generally, if $\mathcal C\subset\mathcal B_e$, then
\[
\#\left\{p\leqslant X:
\TBe{e}(p)\in\mathcal C\right\}
\sim
\left(\sum_{B\in\mathcal C}\omega^{\len(B)}\right)\pi(X).
\]
In particular, taking $e=1$ gives, for every
$1\leqslant i\leqslant L$ and $0\leqslant j\leqslant h_i-1$,
\[
\#\left\{p\leqslant X:\TB(p)=b_{i,j}\right\}
\sim \omega^i\pi(X).
\]
\end{theorem}

\begin{proof}
Enumerate the primes increasingly as $p_1<p_2<\cdots$. By
Lemma~\ref{lem:prime-toral-equidistribution},
\[
p_n\mathbf v\;(\mathrm{mod}\ \Z^d)
\]
is Haar equidistributed. Proposition~\ref{prop:toral-sampling-principle}
therefore gives, for every $B\in\mathcal B_e$,
\[
\frac1N
\#\left\{1\leqslant n\leqslant N:\TBe{e}(p_n)=B\right\}
\longrightarrow
\omega^{\len(B)}.
\]
Taking $N=\pi(X)$ gives the first assertion. The statement for
$\mathcal C\subset\mathcal B_e$ follows by summing over
$B\in\mathcal C$, and the case $e=1$ gives the final assertion.
\end{proof}

It is useful to compare Theorem~\ref{thm:prime-fixed-word},
which describes factors of the substitution fixed word at prime positions,
with Theorem~\ref{thm:prime-terminal-blocks},
which gives the distribution of terminal blocks of primes.
In the Tribonacci case, the two codings can be matched explicitly.
This comparison shows first that the two theorems agree for single
letters, and then gives explicit prime frequencies for factors of
length two.

Since $h_1=h_2=h_3=1$, each Tribonacci letter has a single expanded
block,
\[
\mathcal E(A_i)=b_{i,0},
\]
and therefore
\[
u_n=A_i
\quad\Longleftrightarrow\quad
\TB(n)=b_{i,0}.
\]
Thus the single-letter prime-number theorem of
Theorem~\ref{thm:prime-fixed-word} is exactly the $e=1$ case of
Theorem~\ref{thm:prime-terminal-blocks} after this relabeling.

For factors of length two, the correspondence is still explicit, but
carries begin to appear.  Writing
\[
b_1=(0),\qquad b_2=(0,1),\qquad b_3=(0,1,1),
\]
one checks that
\[
\begin{aligned}
 u_nu_{n+1}=A_1A_1 &\iff \TBe{2}(n)=b_1b_3,\\
 u_nu_{n+1}=A_1A_2 &\iff \TBe{2}(n)=b_1b_1,\\
 u_nu_{n+1}=A_1A_3 &\iff \TBe{2}(n)=b_1b_2,
\end{aligned}
\]
while $A_2A_1$ corresponds to the three cylinders $b_2b_j$ and
$A_3A_1$ to the three cylinders $b_3b_j$, $1\leqslant j\leqslant3$.
Since Theorem~\ref{thm:prime-fixed-word} identifies the ordinary
factor frequencies with the corresponding frequencies at prime starting
positions, Theorem~\ref{thm:prime-terminal-blocks} makes these
common values explicit:
\[
\freq(A_1A_1)=\omega^4,\qquad
\freq(A_1A_2)=\omega^2,\qquad
\freq(A_1A_3)=\omega^3,
\]
and
\[
\freq(A_2A_1)=\omega^2,\qquad
\freq(A_3A_1)=\omega^3,
\]
where the last two identities use
\[
\omega+\omega^2+\omega^3=1.
\]

This comparison also clarifies why the direct prefix-torus proof of
Theorem~\ref{thm:prime-fixed-word} is preferable for general factors. Already at length two, the translation from fixed-word factors to
terminal cylinders involves carries; for longer factors it requires
increasingly more bookkeeping, and for a nontrivial block expansion
such as $g=(1,3,2)$ it also involves the time change $c_m$ from
Theorem~\ref{thm:fixed-word-bridge}. The prefix-torus argument avoids these complications and treats arbitrary
fixed factors directly. In the Tribonacci case, the single-letter
specialization is the fixed-word statement suggested in
\cite[Section~9.2]{DrmotaMuellnerSpiegelhofer2025}; it is distinct from the
broader Tribonacci sum-of-digits problem posed in
\cite[Problem~8]{DrmotaMuellnerSpiegelhofer2025}.

A different use of Theorem~\ref{thm:prime-terminal-blocks} is to
compute frequencies of concrete patterns in the expanded word
$\widehat T_\infty=\mathcal E(A_\infty)$.  We illustrate this for the
recurring example $g=(1,3,2)$.  By
Theorem~\ref{thm:terminal-coding},
\[
\widehat T_\infty(n+1)=\TB(n).
\]
A direct check of the block coding gives
\[
\TB(n)\TB(n+1)\TB(n+2)\TB(n+3)
=(0)(1)(0)(1)
\]
if and only if
\[
\TBe{2}(n)\in
\bigl\{
(0)(0),\ (0)(0,2),\ (0)(1,2),\ (0)(0,3,2)
\bigr\}.
\]
The four cylinders have total digit-lengths $2,3,3,4$.  Hence both the
natural frequency of this pattern and its relative frequency at prime
starting positions are
\[
\omega^2+2\omega^3+\omega^4.
\]
Thus a finite pattern in the expanded word can, when its carry structure
is explicit, be reduced to a finite union of terminal cylinders.

The factor $\omega^{\len(B)}$ in
Theorem~\ref{thm:prime-terminal-blocks} has a natural
probabilistic interpretation. Define a probability measure $\nu$ on the
finite set $\mathcal B$ of semi-proper blocks by
\[
\nu(b_{i,j}):=\omega^i
\qquad
(1\leqslant i\leqslant L,\ 0\leqslant j\leqslant h_i-1).
\]
This is indeed a probability measure, since
\[
\sum_{b\in\mathcal B}\nu(b)
=
\sum_{i=1}^L h_i\omega^i
=1.
\]

\begin{corollary}
\label{cor:bernoulli-prime-terminal-blocks}
Suppose Assumption~\ref{ass:pisot} holds. Fix
$e\geqslant1$ and semi-proper blocks $b_1,\ldots,b_e\in\mathcal B$.
Then
\[
\frac{1}{\pi(X)}
\#\left\{
p\leqslant X:
\TBe{e}(p)=b_1\cdots b_e
\right\}
\longrightarrow
\prod_{r=1}^e\nu(b_r).
\]
Equivalently, if a prime $p\leqslant X$ is chosen uniformly from the
primes up to $X$, then, for every fixed $e$, the joint distribution of
its first $e$ terminal blocks converges, as $X\to+\infty$, to that of
$e$ independent $\nu$-distributed blocks.
\end{corollary}

\begin{proof}
By additivity of digit-length,
\[
\len(b_1\cdots b_e)=\sum_{r=1}^e\len(b_r).
\]
Theorem~\ref{thm:prime-terminal-blocks} therefore gives
\[
\omega^{\len(b_1\cdots b_e)}
=
\prod_{r=1}^e\omega^{\len(b_r)}
=
\prod_{r=1}^e\nu(b_r),
\]
which is the asserted product law.
\end{proof}

Thus, in the sense of finite-dimensional distributions, the process of
successive terminal blocks of a uniformly sampled prime is
asymptotically Bernoulli. In particular, the successive terminal blocks
become asymptotically independent, even though they arise from a
deterministic numeration system.

The main theorem also determines the distribution of the coefficient
of the smallest weight $H_1=1$, which is the closest analogue here of
fixing the final digit in an ordinary integer base.

\begin{corollary}
\label{cor:prime-small-weight-digit}
Suppose Assumption~\ref{ass:pisot} holds. Put
\[
h_{\max}:=\max_{1\leqslant i\leqslant L}h_i.
\]
For every integer $j$ with $0\leqslant j<h_{\max}$, one has
\[
\frac{1}{\pi(X)}\#\left\{p\leqslant X:a_1(p)=j\right\}
\longrightarrow \sum_{\substack{1\leqslant i\leqslant L\\ h_i>j}}
\omega^i,
\]
where $a_1(p)$ denotes the coefficient of the smallest weight $H_1=1$
in the reduced $H$-expansion of $p$.
\end{corollary}

\begin{proof}
The first digit of the terminal block $b_{i,j}$ is exactly $j$.
Therefore the condition $a_1(p)=j$ is the disjoint union of the
terminal-block conditions
\[
\TB(p)=b_{i,j}
\qquad
\left(1\leqslant i\leqslant L,\ h_i>j\right).
\]
Summing Theorem~\ref{thm:prime-terminal-blocks} over these
indices gives the result. The limiting constants are normalized by
\[
\sum_{i=1}^{L} h_i\omega^i=1.
\]
\end{proof}

\subsection{Direct consequences of toral equidistribution}

We first collect the consequences which follow by combining the toral
sampling principle with standard one-parameter equidistribution or
orthogonality inputs. These arguments do not require any new
higher-order prime-pattern machinery beyond the results explicitly
cited below.

The same transfer principle applies when the toral orbit is
sampled at polynomial values of the primes. Thus it remains to establish
the required toral equidistribution for $P(p)\mathbf v$. This is provided
by Rhin's theorem \cite{Rhin1973}: if a real polynomial has at least one
irrational nonconstant coefficient, then its values at the primes are
uniformly distributed modulo one.

\begin{corollary}
\label{cor:prime-polynomial-terminal-sampling}
Suppose Assumption~\ref{ass:pisot} holds. Let $P\in\Z[t]$ be nonconstant and suppose that there exists $p_0$ such that
\(
P(p)\in\N_0
\)
for every prime $p\geqslant p_0$. Then, for every fixed $e\geqslant1$ and
$B\in\mathcal B_e$,
\[
\frac1{\pi(X)}
\#\left\{p_0\leqslant p\leqslant X:\TBe{e}(P(p))=B\right\}
\longrightarrow
\omega^{\len(B)}.
\]
\end{corollary}

\begin{proof}
For every nonzero $\mathbf m\in\Z^d$, the polynomial
\[
(\mathbf m\cdot\mathbf v)P(t)
\]
has an irrational nonconstant coefficient by
Proposition~\ref{prop:rational-independence}. Rhin's
equidistribution theorem for polynomial values along the primes
\cite{Rhin1973} therefore gives
\[
(\mathbf m\cdot\mathbf v)P(p)\;(\mathrm{mod}\ \Z)
\]
equidistributed in $\R/\Z$. By the multidimensional Weyl criterion,
\[
P(p)\mathbf v\;(\mathrm{mod}\ \Z^d)
\]
is Haar equidistributed.

Let $r_1<r_2<\cdots$ be the primes at least $p_0$. Then
\[
P(r_n)\mathbf v\;(\mathrm{mod}\ \Z^d)
\]
is Haar equidistributed, so
Proposition~\ref{prop:toral-sampling-principle}, applied with
$K_n=P(r_n)$, gives
\[
\frac1N
\#\left\{1\leqslant n\leqslant N:
\TBe{e}(P(r_n))=B\right\}
\longrightarrow
\omega^{\len(B)}.
\]
Taking $N=\#\{p_0\leqslant p\leqslant X\}$ and using
\[
\frac{\#\{p_0\leqslant p\leqslant X\}}{\pi(X)}\longrightarrow1
\]
yields the claimed result.
\end{proof}

Polynomial sampling in recurrence numeration has substantial precedent.
Drmota--Steiner \cite{DrmotaSteiner2002} treat Zeckendorf expansions
along polynomial sequences, while Steiner \cite{Steiner2002Parry}
studies Parry expansions of polynomial sequences and polynomial values
at primes using the associated toral Markov partitions. The new point
of Corollary~\ref{cor:prime-polynomial-terminal-sampling} is that the toral coding developed in
Sections~\ref{sec:general-data}--\ref{sec:volume} transfers this
equidistribution input to terminal blocks, yielding the exact local law
\[
\TBe{e}(P(p))=B
\quad\text{with limiting frequency}\quad
\omega^{\len(B)},
\]
including for systems in which the canonical expansion need not be greedy.

A second consequence of the same toral mechanism concerns congruence
conditions on the primes. Restricting the primes to a fixed reduced
residue class preserves the required toral equidistribution, and hence
the terminal-block distribution. In particular, terminal-block
statistics are asymptotically independent of any fixed congruence
condition.

\begin{corollary}\label{cor:congruence-terminal-independence}
Suppose Assumption~\ref{ass:pisot} holds. Fix
$q\geqslant2$ and $a\in\Z$ with $(a,q)=1$. Let $e\geqslant1$ and
$B\in\mathcal B_e$. Then
\[
\#\left\{
p\leqslant X:
p\equiv a\;(\mathrm{mod}\ q),\
\TBe{e}(p)=B
\right\}
\sim
\frac{\omega^{\len(B)}}{\phi(q)}\,\pi(X).
\]
Equivalently, among primes in any fixed reduced residue class modulo
$q$, the limiting distribution of terminal block words is the same as
among all primes.
\end{corollary}

\begin{proof}
Write
\[
\pi(X;q,a):=
\#\{p\leqslant X:p\equiv a\;(\mathrm{mod}\ q)\}.
\]
We first show that the restricted toral orbit
\[
p\mathbf v\;(\mathrm{mod}\ \Z^d),
\qquad
p\equiv a\;(\mathrm{mod}\ q),
\]
is Haar equidistributed. Let $\mathbf m\in\Z^d\setminus\{0\}$ and put
$\alpha:=\mathbf m\cdot\mathbf v$. By
Proposition~\ref{prop:rational-independence}, $\alpha$ is
irrational. Additive orthogonality gives
\[
\sum_{\substack{p\leqslant X\\ p\equiv a\;(\mathrm{mod}\ q)}}
e^{2\pi i p\alpha}
=
\frac1q\sum_{h=0}^{q-1}
e^{-2\pi i ha/q}
\sum_{p\leqslant X}
e^{2\pi i p(\alpha+h/q)}.
\]
Every number $\alpha+h/q$ is irrational, so Vinogradov's theorem
\cite{Vinogradov1954} makes each inner sum $o(\pi(X))$. Hence the
left-hand side is $o(\pi(X))$. The prime number theorem in arithmetic
progressions gives
\[
\pi(X;q,a)\sim\frac{\pi(X)}{\phi(q)}.
\]
Therefore
\[
\frac{1}{\pi(X;q,a)}
\sum_{\substack{p\leqslant X\\ p\equiv a\;(\mathrm{mod}\ q)}}
e^{2\pi i p\alpha}
\longrightarrow 0,
\]
and the multidimensional Weyl criterion gives Haar equidistribution of
the restricted toral orbit.

Let $r_1<r_2<\cdots$ denote the primes satisfying
\(
r_n\equiv a \pmod q.
\)
For this sequence, the toral equidistribution established above allows us
to apply Proposition~\ref{prop:toral-sampling-principle} with
$K_n=r_n$. Hence
\[
\frac1N
\#\left\{1\leqslant n\leqslant N:\TBe{e}(r_n)=B\right\}
\longrightarrow
\omega^{\len(B)}.
\]
Taking $N=\pi(X;q,a)$ and using the prime number theorem in arithmetic
progressions gives the asserted asymptotic.
\end{proof}

For comparison, the integer analogue follows directly from
Proposition~\ref{prop:toral-sampling-principle}. If $q\geqslant1$, then
$q\mathbf v$ is still totally irrational, and hence
\[
(qn+a)\mathbf v
=
n(q\mathbf v)+a\mathbf v
\pmod{\Z^d}
\]
is Haar equidistributed. Therefore, for every fixed $q\geqslant1$,
$a\in\N_0$, $e\geqslant1$, and $B\in\mathcal B_e$,
\[
\frac1N\#\{1\leqslant n\leqslant N:
\TBe{e}(qn+a)=B\}
\longrightarrow \omega^{\len(B)}.
\]
Thus every fixed arithmetic progression of the integers inherits the
same terminal-cylinder law.

The same toral coding also records correlations between terminal
conditions at finitely many fixed shifts. To formulate this precisely,
let
\[
F=\operatorname{rem}\circ\,\tau:
\widetilde R\longrightarrow \R^d/\Z^d
\]
be the quotient map from Lemma~\ref{lem:regular-toral-coding}, and write
$\mu_{\mathrm{Haar}}$ for normalized Haar measure on $\R^d/\Z^d$.
For $B\in\mathcal B_e$, let $\Omega_B$ denote, up to a Haar-null set,
the set of torus points whose unique representative in $\widetilde R$
belongs to $R_B$. Then
\[
\mu_{\mathrm{Haar}}(\Omega_B)
=
\frac{\vol(R_B)}{\vol(\widetilde R)}
=
\omega^{\len(B)}.
\]
By Lemma~\ref{lem:regular-toral-coding} and the null-boundary property
of $R_B$, the sets $\Omega_B$ may be chosen to be Haar-continuity
sets.

\begin{corollary}
\label{cor:fixed-shift-terminal-correlations}
Suppose Assumption~\ref{ass:pisot} holds. Fix
$r\geqslant1$, shifts $t_1,\ldots,t_r\in\N_0$, integers
$e_1,\ldots,e_r\geqslant1$, and words
$B_j\in\mathcal B_{e_j}$. Then
\[
\frac1N
\#\left\{1\leqslant n\leqslant N:
\TBe{e_j}(n+t_j)=B_j\ \text{for all }j\right\}
\longrightarrow
\mu_{\mathrm{Haar}}\!\left(
\bigcap_{j=1}^r(\Omega_{B_j}-t_j\mathbf v)
\right),
\]
and also
\[
\frac1{\pi(X)}
\#\left\{p\leqslant X:
\TBe{e_j}(p+t_j)=B_j\ \text{for all }j\right\}
\longrightarrow
\mu_{\mathrm{Haar}}\!\left(
\bigcap_{j=1}^r(\Omega_{B_j}-t_j\mathbf v)
\right).
\]
Thus the integer and prime fixed-shift correlation laws are governed by
the same finite intersection of translated terminal cylinders on the
torus.
\end{corollary}

\begin{proof}
Set
\[
\mathcal E:=
F(\partial\widetilde R)
\cup
\bigcup_{j=1}^r
\bigcup_{C\in\mathcal B_{e_j}}
F(\partial R_C).
\]
By Lemma~\ref{lem:regular-toral-coding} and
Proposition~\ref{prop:terminal-cylinder-window}, the set $\mathcal E$ is compact
and Haar-null, since the families $\mathcal B_{e_j}$ are finite.
Hence $\mathcal E$ is a Haar-continuity set.

If
\[
m\mathbf v\;(\mathrm{mod}\ \Z^d)\notin\mathcal E,
\]
then the torus point has a unique representative in $\widetilde R$,
and this representative belongs to exactly one cylinder of each level
$e_j$. Therefore
\[
\TBe{e_j}(m)=B_j
\quad\Longleftrightarrow\quad
m\mathbf v\;(\mathrm{mod}\ \Z^d)\in\Omega_{B_j}.
\]
Since
\[
(m+t_j)\mathbf v
=
m\mathbf v+t_j\mathbf v
\pmod{\Z^d},
\]
the simultaneous conditions
\[
\TBe{e_j}(m+t_j)=B_j
\qquad (1\leqslant j\leqslant r)
\]
are equivalent to
\[
m\mathbf v\;(\mathrm{mod}\ \Z^d)
\in
D:=
\bigcap_{j=1}^r(\Omega_{B_j}-t_j\mathbf v),
\]
unless
\[
m\mathbf v\;(\mathrm{mod}\ \Z^d)
\in
\bigcup_{j=1}^r(\mathcal E-t_j\mathbf v).
\]
The latter set is a finite union of translates of the compact
Haar-null set $\mathcal E$, hence is itself a Haar-continuity set of
measure zero.

Moreover, $D$ is a Haar-continuity set. Indeed, each
$\Omega_{B_j}$ is one, and
\[
\partial D
\subseteq
\bigcup_{j=1}^r
\bigl(\partial\Omega_{B_j}-t_j\mathbf v\bigr),
\]
whose Haar measure is zero.

The orbit $n\mathbf v$ is Haar equidistributed by
Proposition~\ref{prop:rational-independence} and the
multidimensional Weyl criterion, while the prime orbit $p\mathbf v$ is
Haar equidistributed by
Lemma~\ref{lem:prime-toral-equidistribution}. Applying these two
equidistribution statements to $D$, and using that the exceptional set
above has Haar measure zero, gives both limits.
\end{proof}

A different consequence of the toral coding comes from combining it
with classical exponential-sum estimates. The indicator of a fixed
terminal cylinder can be represented, outside a null exceptional set,
by a Riemann-integrable function on the torus. Fourier approximation
then allows one to use Davenport's estimate for the M\"obius function.
This gives the following local-cylinder analogue, in the present Pisot
family, of the M\"obius phenomena studied for Zeckendorf digit
functions by Drmota, M\"ullner, and Spiegelhofer
\cite{DrmotaMuellnerSpiegelhofer2018}.

\begin{corollary}
\label{cor:mobius-terminal-cylinders}
Suppose Assumption~\ref{ass:pisot} holds. For every
fixed $e\geqslant1$ and $B\in\mathcal B_e$,
\[
\sum_{n\leqslant X}\mu(n)\,
\mathbf 1_{\{\TBe{e}(n)=B\}}
=o(X),
\qquad X\to+\infty,
\]
where $\mu$ is the M\"obius function.
\end{corollary}

\begin{proof}
Set
\[
g_B:=\mathbf 1_{\Omega_B}.
\]
Since $\Omega_B$ is a Haar-continuity set, $g_B$ is bounded and
Riemann integrable on $\R^d/\Z^d$.

Let
\[
E_e
:=
F(\partial\widetilde R)
\cup
\bigcup_{C\in\mathcal B_e}F(\partial R_C).
\]
By Proposition~\ref{prop:terminal-cylinder-partition},
Proposition~\ref{prop:toral-congruence}, and
Lemma~\ref{lem:regular-toral-coding}, one has
\[
\mathbf 1_{\{\TBe{e}(n)=B\}}
=
g_B(n\mathbf v)
\]
whenever $n\mathbf v\notin E_e$. Indeed, outside
$F(\partial\widetilde R)$ the torus point $n\mathbf v$ has the unique
representative $\evalG(n)$ in $\widetilde R$, while outside the
cylinder boundaries the measurable partition
$(R_C)_{C\in\mathcal B_e}$ determines its unique terminal class.

The set $E_e$ is compact and Haar-null: the family $\mathcal B_e$ is
finite, every $\partial R_C$ has Lebesgue measure zero by
Proposition~\ref{prop:terminal-cylinder-window}, and $F$ is Lipschitz.
Since $n\mathbf v$ is Haar equidistributed, the set of integers $n$
for which $n\mathbf v\in E_e$ has density zero. As
$|\mu(n)|\leqslant1$, its M\"obius-weighted contribution is therefore
$o(X)$.

It remains to show that
\[
\sum_{n\leqslant X}\mu(n)g_B(n\mathbf v)=o(X).
\]
Let $\varepsilon>0$. Since $g_B$ is Riemann integrable, choose a
continuous function $f$ on $\R^d/\Z^d$ such that
\[
\int_{\R^d/\Z^d}|g_B-f|\,d\mu_{\mathrm{Haar}}<\varepsilon.
\]
The equidistribution of $n\mathbf v$ then gives
\[
\limsup_{X\to+\infty}
\frac1X
\sum_{n\leqslant X}
|g_B(n\mathbf v)-f(n\mathbf v)|
\leqslant\varepsilon.
\]

Next choose a trigonometric polynomial $T$ such that
\[
\|f-T\|_\infty<\varepsilon.
\]
Each Fourier mode of $T(n\mathbf v)$ has the form
\[
e^{2\pi i n\alpha},
\qquad
\alpha=\mathbf m\cdot\mathbf v,
\]
and Davenport's uniform estimate \cite{Davenport1937},
\[
\sup_{\alpha\in\R}
\left|
\sum_{n\leqslant X}\mu(n)e^{2\pi i n\alpha}
\right|
\ll_A
\frac{X}{(\log X)^A}
\]
for every $A>0$, gives
\[
\sum_{n\leqslant X}\mu(n)T(n\mathbf v)=o(X).
\]
Consequently,
\[
\limsup_{X\to+\infty}
\frac1X
\left|
\sum_{n\leqslant X}\mu(n)g_B(n\mathbf v)
\right|
\leqslant 2\varepsilon.
\]
Since $\varepsilon>0$ is arbitrary, the desired estimate follows.
\end{proof}

\subsection{Selected transfers from established prime theorems}

We finish with two representative transfers from established prime
theorems. Chebotarev shows that the terminal condition is compatible
with a genuinely arithmetic Frobenius constraint, while the
Green--Tao theorem shows that every fixed terminal prime class retains
long additive structure.

For the Chebotarev transfer, the qualitative irrationality supplied by
Proposition~\ref{prop:rational-independence} is no longer
sufficient: we need a quantitative Diophantine bound for the linear
forms $\mathbf m\cdot\mathbf v$. This follows from the algebraicity of
the coordinates of $\mathbf v$.

\begin{proposition}\label{prop:algebraic-toral-direction}
Each coordinate $v_r$ is algebraic. Moreover, there exist constants
$c_{\mathbf v}>0$ and $\kappa_{\mathbf v}\geqslant1$, depending only on
the numeration system, such that
\[
\bigl\|\mathbf m\cdot\mathbf v\bigr\|_{\R/\Z}
\geqslant
c_{\mathbf v}
\bigl(1+\|\mathbf m\|_\infty\bigr)^{-\kappa_{\mathbf v}}
\qquad
(\mathbf m\in\Z^d\setminus\{\mathbf0\}).
\]
In particular, for every fixed $\mathbf m\neq\mathbf0$, the number
$\alpha_{\mathbf m}:=\mathbf m\cdot\mathbf v$ is algebraic and
irrational, and there exists $c_{\mathbf m}>0$ such that
\[
\|q\alpha_{\mathbf m}\|_{\R/\Z}
\geqslant
c_{\mathbf m}q^{-\kappa_{\mathbf v}}
\qquad (q\geqslant1).
\]
Thus $\alpha_{\mathbf m}$ has finite Diophantine type.
\end{proposition}

\begin{proof}
Recall that the arithmetic lattice $\mathcal T$ introduced in
Section~\ref{sec:volume} is generated by
\[
Q_1(\Gamma),\ldots,Q_d(\Gamma),
\]
where
\[
Q_r(\Gamma)=\bigl(Q_r(\eta_\nu)\bigr)_\nu
\]
and the $\eta_\nu$ are the contracting conjugates of $\varphi$.
Since each $Q_r$ is a Laurent polynomial with integer coefficients,
its values at the $\eta_\nu$ are algebraic; in the real coordinates of
$E_s=\R^a\times\C^b$, their real and imaginary parts are therefore
algebraic. The vector $s_0$ also has algebraic coordinates.

By Lemma~\ref{lem:Qdet},
$Q_1(\Gamma),\ldots,Q_d(\Gamma)$ form a real basis of $E_s$. Hence the
identity
\[
s_0=\sum_{r=1}^d v_rQ_r(\Gamma)
\]
and Cramer's rule show that $v_1,\ldots,v_d$ are algebraic.

Let
\[
\mathbb K:=\Q(v_1,\ldots,v_d),
\]
and choose a positive integer $D$ such that $Dv_r$ is an algebraic
integer for every $r$. Given $\mathbf m\neq\mathbf0$, choose
$a_0\in\Z$ so that
\[
|a_0+\mathbf m\cdot\mathbf v|
=
\|\mathbf m\cdot\mathbf v\|_{\R/\Z}.
\]
By Proposition~\ref{prop:rational-independence},
$a_0+\mathbf m\cdot\mathbf v\neq0$. Therefore
\[
\alpha:=D(a_0+\mathbf m\cdot\mathbf v)
\]
is a nonzero algebraic integer, and hence
\[
|N_{\mathbb K/\Q}(\alpha)|\geqslant1.
\]

The choice of $a_0$ gives
\[
|a_0|\ll_{\mathbf v}1+\|\mathbf m\|_\infty.
\]
Consequently, for every embedding
$\sigma:\mathbb K\hookrightarrow\C$ distinct from the identity,
\[
|\sigma(\alpha)|
\ll_{\mathbf v,\mathbb K}
1+\|\mathbf m\|_\infty.
\]
Taking the product over all embeddings of $\mathbb K$ yields
\[
|a_0+\mathbf m\cdot\mathbf v|
\gg_{\mathbf v,\mathbb K}
\bigl(1+\|\mathbf m\|_\infty\bigr)^{1-[\mathbb K:\Q]}.
\]
Thus the asserted estimate holds with
\[
\kappa_{\mathbf v}=[\mathbb K:\Q]-1.
\]
By Proposition~\ref{prop:rational-independence},
$\mathbb K\neq\Q$, and hence $\kappa_{\mathbf v}\geqslant1$.

Finally, fix $\mathbf m\neq\mathbf0$ and put
\[
\alpha_{\mathbf m}:=\mathbf m\cdot\mathbf v.
\]
The algebraicity of the coordinates $v_r$ shows that
$\alpha_{\mathbf m}$ is algebraic, while
Proposition~\ref{prop:rational-independence} shows that
it is irrational. Applying the preceding estimate to the nonzero
integer vector $q\mathbf m$, with $q\geqslant1$, gives
\[
\|q\alpha_{\mathbf m}\|_{\R/\Z}
=
\|(q\mathbf m)\cdot\mathbf v\|_{\R/\Z}
\geqslant
c_{\mathbf v}
\bigl(1+q\|\mathbf m\|_\infty\bigr)^{-\kappa_{\mathbf v}}
\gg_{\mathbf m,\mathbf v}
q^{-\kappa_{\mathbf v}}.
\]
Thus $\alpha_{\mathbf m}$ has finite Diophantine type.
\end{proof}

The asymptotic independence from fixed congruence classes obtained above
has a Chebotarev analogue. Ji, Kazdan, and McDonald
\cite{JiKazdanMcDonald2020} studied primes subject simultaneously to
one-dimensional Beatty restrictions and Chebotarev conditions, and
proved the exponential-sum estimate that we use as the analytic input
below. Combined with the multidimensional toral coding developed here,
this allows us to impose a prescribed Frobenius conjugacy class on the
terminal-block problem.

Let $E/\Q$ be a finite Galois extension with Galois group $G$, and let
$C\subseteq G$ be a conjugacy class. For a prime $p$ unramified in
$E$, let
\[
\operatorname{Frob}_p(E/\Q)
\]
denote its Frobenius conjugacy class in $G$. We write
\[
\operatorname{Li}(X):=\int_2^X\frac{dt}{\log t}.
\]

\begin{theorem}\label{thm:terminal-chebotarev}
Suppose Assumption~\ref{ass:pisot} holds. Let
$E/\Q$ and $C$ be as above. For every fixed $e\geqslant1$ and
$B\in\mathcal B_e$,
\[
\#\left\{
p\leqslant X:
\operatorname{Frob}_p(E/\Q)=C,\ 
\TBe{e}(p)=B
\right\}
\sim
\frac{|C|}{|G|}
\omega^{\len(B)}
\operatorname{Li}(X).
\]
Equivalently, the terminal block and the Frobenius conjugacy class are
asymptotically independent. The same statement holds for any fixed
union of conjugacy classes.
\end{theorem}

\begin{proof}
Let $\mathbf m\in\Z^d\setminus\{\mathbf0\}$ and put
\[
\alpha_{\mathbf m}:=\mathbf m\cdot\mathbf v.
\]
By Proposition~\ref{prop:algebraic-toral-direction},
$\alpha_{\mathbf m}$ is an algebraic irrational of finite
Diophantine type. The exponential-sum estimate of Ji, Kazdan, and
McDonald \cite[Lemma~2.6]{JiKazdanMcDonald2020}, applied with
$\gamma=\alpha_{\mathbf m}$, modulus $q=1$, and Fourier parameter
$k=1$, gives some $\eta_{\mathbf m}>0$ such that
\[
\sum_{\substack{p\leqslant X\\
\operatorname{Frob}_p(E/\Q)=C}}
(\log p)\,
e^{2\pi i p\alpha_{\mathbf m}}
\ll_{E,\mathbf m}
X^{1-\eta_{\mathbf m}}.
\]
Partial summation therefore gives
\[
\sum_{\substack{p\leqslant X\\
\operatorname{Frob}_p(E/\Q)=C}}
e^{2\pi i p\alpha_{\mathbf m}}
=
o\!\left(\operatorname{Li}(X)\right).
\]

By the Chebotarev density theorem,
\[
\#\left\{
p\leqslant X:
\operatorname{Frob}_p(E/\Q)=C
\right\}
\sim
\frac{|C|}{|G|}\operatorname{Li}(X).
\]
Hence, by the multidimensional Weyl criterion, the sequence
\[
p\mathbf v\;(\mathrm{mod}\ \Z^d),
\qquad
\operatorname{Frob}_p(E/\Q)=C,
\]
is Haar equidistributed on $\R^d/\Z^d$.

Let $r_1<r_2<\cdots$ denote the unramified primes satisfying
\[
\operatorname{Frob}_{r_n}(E/\Q)=C.
\]
For this sequence, the toral equidistribution established above allows
us to apply Proposition~\ref{prop:toral-sampling-principle} with
$K_n=r_n$. Hence
\[
\frac1N
\#\left\{
1\leqslant n\leqslant N:
\TBe{e}(r_n)=B
\right\}
\longrightarrow
\omega^{\len(B)}.
\]
Taking
\[
N=
\#\left\{
p\leqslant X:
\operatorname{Frob}_p(E/\Q)=C
\right\}
\]
and using the Chebotarev asymptotic above gives
\[
\#\left\{
p\leqslant X:
\operatorname{Frob}_p(E/\Q)=C,\ 
\TBe{e}(p)=B
\right\}
\sim
\frac{|C|}{|G|}
\omega^{\len(B)}
\operatorname{Li}(X).
\]
The statement for a fixed union of conjugacy classes follows by
summing over its finitely many disjoint classes.
\end{proof}

In particular, let $f\in\Z[x]$ be a fixed separable polynomial. For all
but finitely many primes $p$, the factorization type of $f$ modulo $p$
is determined by the Frobenius conjugacy class in the Galois group of
its splitting field. Since each prescribed factorization type that
occurs corresponds to a fixed union of conjugacy classes,
Theorem~\ref{thm:terminal-chebotarev} shows that it is asymptotically
independent of the terminal block of $p$.

We finally record a combinatorial consequence of the positive relative
density of each terminal prime class. Let
\[
\mathcal P_B
:=
\left\{
p:\TBe{e}(p)=B
\right\}.
\]
By Theorem~\ref{thm:prime-terminal-blocks},
$\mathcal P_B$ has positive relative density inside the primes.
Shao's refinement of the Green--Tao theorem
\cite{ShaoNarrow2017} then gives not only arbitrarily long arithmetic
progressions in $\mathcal P_B$, but also a quantitative bound for their
common difference.

\begin{corollary}
\label{cor:green-tao-terminal-cylinders}
Suppose Assumption~\ref{ass:pisot} holds. Let
$e\geqslant1$, $B\in\mathcal B_e$, and $k\geqslant3$. Put
\[
L_k:=(k-1)2^{k-2}.
\]
Then, for every sufficiently large $X$, there exist integers
$a\geqslant1$ and $r\geqslant1$ such that
\[
a,\ a+r,\ \ldots,\ a+(k-1)r
\]
all belong to $\mathcal P_B\cap[1,X]$ and
\[
r\ll_{B,k}(\log X)^{L_k}.
\]
In particular, every fixed terminal block cylinder contains arbitrarily
long arithmetic progressions of primes with polylogarithmically small
common difference.
\end{corollary}

\begin{proof}
By Theorem~\ref{thm:prime-terminal-blocks},
\[
\frac{\#(\mathcal P_B\cap[1,X])}{\pi(X)}
\longrightarrow
\omega^{\len(B)}>0.
\]
Hence, for all sufficiently large $X$,
$\mathcal P_B\cap[1,X]$ has relative density inside the primes bounded
below by a positive constant depending only on $B$. Shao's theorem~\cite{ShaoNarrow2017} on
narrow arithmetic progressions in positive-density subsets of the
primes therefore applies and gives the stated
bound with
\[
L_k=(k-1)2^{k-2}.
\]
\end{proof}

This completes the proof of the main prime terminal-block theorem and
of the selected consequences under Assumption~\ref{ass:pisot}.
The next section proves the geometric result stated in
Theorem~\ref{thm:rauzy-fundamental-domain} and then gives simple algebraic
conditions ensuring that Assumption~\ref{ass:pisot} holds.

%% file: sections/08-rauzy-verification.tex
\section{Automatic Rauzy geometry and algebraic criteria}\label{sec:rauzy-verification}

Section~\ref{sec:prime-blocks} derives the prime terminal-block theorem from
Theorem~\ref{thm:rauzy-fundamental-domain}. Here we prove that theorem
from Assumption~\ref{ass:pisot}. The key point is that, after reversing
the substitution words, the companion substitution falls under Barge's
pure-discreteness theorem~\cite[Theorem~3.13]{Barge2016}.

We then give simple coefficient conditions ensuring that
Assumption~\ref{ass:pisot} itself holds, together with two explicit infinite
families illustrating phenomena beyond the classical $\beta$-substitution
setting. The first contains the recurring example $g=(1,3,2)$ and is non-$\beta$ for $L\geqslant3$; the second consists of Pisot units failing the finiteness property~(F). Finally, we translate these criteria back to the original
recurrence coefficients and obtain unconditional cases of the results of
Sections~\ref{sec:volume} and~\ref{sec:prime-blocks}. The finite
$\beta$-expansion interpretation and the balanced-pair criterion provide
complementary perspectives on these algebraic cases.

\subsection{Automatic Rauzy geometry for the companion family}

We now prove the geometric statement announced in
Theorem~\ref{thm:rauzy-fundamental-domain}.

\begin{proposition}\label{prop:automatic-rauzy-hypothesis}
Under Assumption~\ref{ass:pisot}, the central Rauzy window is a
multiplicity-one fundamental domain for the associated lattice, with
the null-boundary and disjointness properties stated in
Theorem~\ref{thm:rauzy-fundamental-domain}.
\end{proposition}

\begin{proof}
Under Assumption~\ref{ass:pisot}, Section~\ref{sec:general-data} shows
that $\sigma$ is primitive, unimodular, and irreducible Pisot. Let
$\revsig$ be obtained by reversing every substitution word. Thus
\[
\revsig(A_i)=A_{i+1}A_1^{h_i}
\quad(1\leqslant i<L),
\qquad
\revsig(A_L)=A_1.
\]
Reversal does not change letter counts, so $\revsig$ has the
same incidence matrix as $\sigma$. In particular, it is primitive and has
Pisot inflation factor $\varphi$. Its initial letters are
\[
A_2,A_3,\ldots,A_L,A_1,
\]
so $\revsig$ is injective on initial letters, while every image
ends in $A_1$, so it is constant on final letters. Barge's
theorem~\cite[Theorem~3.13]{Barge2016} therefore gives pure discrete
spectrum for the one-dimensional tiling dynamical system
\(
(\Omega_{\revsig},\R).
\)

The transfer back to $\sigma$ uses both dynamical systems associated with a
substitution. Given a substitution $\tau$, we use the standard notation
$(X_\tau,s)$ for its symbolic substitution system, with its $\Z$-action
generated by the shift, and $(\Omega_\tau,\R)$ for its associated
one-dimensional tiling system, with its translation action. For an irreducible Pisot substitution, pure discrete
spectrum of these two systems is equivalent. Sirvent--Solomyak
\cite[Corollary~5.2]{SirventSolomyak2002} proved the implication from the
tiling $\R$-action to the symbolic $\Z$-action. In their Added in Proof, they
refer to Clark--Sadun \cite[Theorem~3.1]{ClarkSadun2003} for the converse
implication, and hence for the full equivalence. The same equivalence is
stated in the discussion of the symbolic and tiling versions of the Pisot
Substitution Conjecture in
\cite[Section~4, discussion after Conjecture~4.2]{AkiyamaBargeBertheLeeSiegel2015}.

Since $\revsig$ has the same irreducible Pisot incidence matrix as
$\sigma$, the symbolic--tiling equivalence transfers Barge's conclusion to the
symbolic system $(X_{\revsig},s)$. By construction,
\[
(\revsig)^n(A_i)
=
\operatorname{rev}\!\bigl(\sigma^n(A_i)\bigr)
\qquad(n\geqslant0).
\]
Hence the languages of $\sigma$ and $\revsig$ are exchanged by
word reversal. Reversal of a bi-infinite word conjugates the shift on one
system with the inverse shift on the other. Pure discrete spectrum is
unchanged when the acting transformation is replaced by its inverse, so
$(X_\sigma,s)$ has pure discrete spectrum. Applying the same
symbolic--tiling equivalence once more, $(\Omega_\sigma,\R)$ has pure
discrete spectrum.

We have thus obtained pure discrete spectrum for the original substitution
$\sigma$. To deduce Theorem~\ref{thm:rauzy-fundamental-domain},
it remains to translate the standard Rauzy-fractal consequences of pure
discreteness into the specific stable coordinates, Rauzy pieces, and lattice
introduced in Section~\ref{sec:rauzy}. With the usual column-incidence
convention, the substitution matrix is $M^T$, and the abelianization of a
prefix $W$ is exactly the vector $\ell(W)$ used in
Section~\ref{sec:general-data}. The standard contracting space and $E_s$ are
two realizations of the stable space associated with the non-dominant
spectrum of $M^T$. Hence there is an invertible real-linear map $A$ from the
standard contracting space onto $E_s$ such that
\[
A\bigl(\pi_s^{\mathrm{std}}(v)\bigr)=\pi_s(v)
\qquad(v\in\R^L).
\]
Thus, for every prefix-count vector, the standard stable projection and our
projection $\pi_s$ represent the same point after the change of coordinates
$A$. It follows directly from the prefix definitions that $A$ sends the
standard prefix Rauzy fractal and its letter-indexed subtiles onto the sets
$R$ and $R_i$ of Section~\ref{sec:rauzy}. Likewise, it sends the standard
difference lattice generated by the vectors
$\pi_s^{\mathrm{std}}(e_i-e_L)$ onto the lattice generated by
$\pi_s(e_i-e_L)$. Therefore the standard Rauzy-fractal statements transport
to our coordinates without changing tiling multiplicities or
measure-theoretic disjointness.

Lemma~\ref{lem:rauzy-hypotheses} gives strong coincidence for
$\sigma$. For an irreducible unimodular Pisot substitution satisfying strong
coincidence, pure discrete spectrum is equivalent to the associated Rauzy
fractal generating a lattice tiling of the contracting space
\cite[Theorem~4.8]{BaratBertheLiardetThuswaldner2006}; see also
\cite[Section~3.4]{AkiyamaBargeBertheLeeSiegel2015}. By
Lemma~\ref{lem:rauzy-lattice-full-rank}, the transported lattice in our
coordinates is precisely
\[
\mathcal L
=
\left\langle
\pi_s(e_i-e_L):1\leqslant i<L
\right\rangle_{\Z}.
\]
Consequently,
\[
R+\mathcal L=E_s,
\]
and distinct lattice translates of $R$ overlap only on a set of Lebesgue
measure zero.

It remains to verify the topological and boundary properties required in
Theorem~\ref{thm:rauzy-fundamental-domain}. For a primitive unit
Pisot substitution, the central tile and all its subtiles are compact and
equal to the closures of their interiors
\cite[Theorem~2.6]{SiegelThuswaldner2009}. Moreover, strong coincidence
makes the letter subtiles pairwise disjoint in measure; see
\cite[Section~5.1]{AkiyamaBargeBertheLeeSiegel2015}. Since $\sigma$ is
primitive, every $R_i$ is nonempty; hence each $R_i$ has nonempty interior
and therefore positive Lebesgue measure. Thus
$R=R_1\cup\cdots\cup R_L$ is compact and has positive Lebesgue measure.

The multiplicity-one tiling also yields the null-boundary assertions
directly. Since $R$ is compact, the lattice family
$(R+\ell)_{\ell\in\mathcal L}$ is locally finite. If
$x\in\partial R$, choose points $x_n\notin R$ with $x_n\to x$. Since
$R+\mathcal L=E_s$, each $x_n$ belongs to some translate
$R+\ell_n$ with $\ell_n\neq0$. By local finiteness, after passing to a
subsequence we may assume that $\ell_n=\ell$ is fixed. Since
$R+\ell$ is closed,
\[
x\in R\cap(R+\ell),
\qquad
\ell\in\mathcal L\setminus\{0\}.
\]
Only finitely many translates $R+\ell$ meet $R$, and therefore
\[
\partial R
\subseteq
\bigcup_{\substack{\ell\in\mathcal L\setminus\{0\}\\
R\cap(R+\ell)\neq\varnothing}}
R\cap(R+\ell).
\]
Each intersection in this finite union has Lebesgue measure zero by the
multiplicity-one tiling. Hence
\[
\lambda_{L-1}(\partial R)=0.
\]

Similarly, fix $1\leqslant i\leqslant L$ and let
$x\in\partial R_i\setminus\partial R$. Then
$x\in\operatorname{int}R$. Since $R_i$ is closed, we have $x\in R_i$.
Moreover, $x$ is not an interior point of $R_i$, so points of
$R\setminus R_i$ approach $x$. As the finitely many closed sets
$R_1,\ldots,R_L$ cover $R$, after passing to a subsequence these points
belong to one fixed $R_j$, with $j\neq i$. Closedness of $R_j$ then gives
$x\in R_j$. Hence
\[
\partial R_i
\subseteq
\partial R\cup
\bigcup_{j\neq i}(R_i\cap R_j).
\]
The letter subtiles are pairwise disjoint in measure, so together with
$\lambda_{L-1}(\partial R)=0$ this yields
\[
\lambda_{L-1}(\partial R_i)=0
\qquad(1\leqslant i\leqslant L).
\]
The same measure-disjointness also implies that the interiors of
$R_1,\ldots,R_L$ are pairwise disjoint, since any nonempty intersection of
two such interiors would be a nonempty open set and hence would have
positive Lebesgue measure.

It remains only to match the measure-theoretic lattice tiling with the
pointwise formulation of
Theorem~\ref{thm:rauzy-fundamental-domain}. The covering
$R+\mathcal L=E_s$ implies that
\[
q_{\mathcal L}:R\longrightarrow E_s/\mathcal L
\]
is onto. If two distinct points
$x,y\in R\setminus\partial R=\operatorname{int}R$ had the same image,
then
\[
\ell:=x-y\in\mathcal L\setminus\{0\},
\]
and therefore
\[
x\in
\operatorname{int}R\cap
\bigl(\operatorname{int}R+\ell\bigr).
\]
This is a nonempty open intersection and hence has positive Lebesgue
measure, contradicting the multiplicity-one lattice tiling. Thus
$q_{\mathcal L}$ is one-to-one on $R\setminus\partial R$.

The same argument shows that the interiors of the translates
$R+\ell$, $\ell\in\mathcal L$, are pairwise disjoint. Together with the
properties established above, these are precisely the assertions of
Theorem~\ref{thm:rauzy-fundamental-domain}.
\end{proof}

Proposition~\ref{prop:automatic-rauzy-hypothesis} proves
Theorem~\ref{thm:rauzy-fundamental-domain}. Consequently, the results
of Sections~\ref{sec:rauzy}--\ref{sec:prime-blocks} require no
additional geometric assumption beyond Assumption~\ref{ass:pisot}.

It remains to determine when Assumption~\ref{ass:pisot} itself follows
directly from the recurrence coefficients. We therefore give simple sufficient conditions and two explicit infinite families: a family that is non-$\beta$ for $L\geqslant3$ and contains the recurring example $g=(1,3,2)$, and a family of Pisot units failing the finiteness property~(F).

\subsection{Algebraic sufficient conditions}

We record two classical coefficient criteria ensuring
Assumption~\ref{ass:pisot}, followed by an infinite family lying outside
these two branches when $L\geqslant3$. No new geometric argument is needed
here.

\begin{proposition}\label{prop:automatic-pisot-beta-branches}
Assume either the monotone condition
\[
 h_1\geqslant h_2\geqslant\cdots\geqslant h_L=1,
 \tag{M}
\]
or the dominant-coefficient condition
\[
 h_1>h_2+\cdots+h_L.
 \tag{D}
\]
Then the dominant root $\varphi$ of $P_h$ is a Pisot number and
$P_h$ is its minimal polynomial. In particular,
Assumption~\ref{ass:pisot} holds automatically.
\end{proposition}

\begin{proof}
In case~(M), Brauer's classical root-location theorem~\cite{Brauer1950} for
\[
X^L-h_1X^{L-1}-\cdots-h_L,
\qquad
h_1\geqslant\cdots\geqslant h_L>0,
\]
shows that all but one root lie strictly inside the unit disk. In case~(D), Berwald's criterion~\cite{Berwald1933} gives the same
root-location conclusion: with $a_r=-h_r$,
\[
|1+a_2|+|a_3|+\cdots+|a_L|
=
h_2+\cdots+h_L-1
<
h_1
=
|a_1|,
\]
so all but one root lie strictly inside the unit disk. In either case, the remaining root is the positive
dominant root $\varphi>1$.

It remains only to prove irreducibility. By Gauss's lemma, if $P_h$ were
reducible over $\Q$, we could write
\[
P_h=FG
\]
with monic nonconstant $F,G\in\Z[X]$. Exactly one of the two factors
contains the root $\varphi$; every root of the other factor therefore has
modulus strictly less than $1$. Its constant term is nevertheless a nonzero
integer, since the product of the constant terms of $F$ and $G$ is
$P_h(0)=-1$. On the other hand, since the factor is monic, the absolute
value of its constant term is the product of the moduli of its roots, and
is therefore strictly less than $1$, a contradiction. Hence $P_h$ is
irreducible over $\Q$. Thus $P_h$ is the minimal polynomial of the Pisot
number $\varphi$, and Assumption~\ref{ass:pisot} follows.
\end{proof}

The two coefficient criteria above cover broad families, but they do not
include the recurring example $g=(1,3,2)$. In fact, this example belongs to
an explicit infinite family for which the algebraic hypothesis can be
verified directly.

\begin{proposition}\label{prop:nonbeta-family-pisot}
Let $c\geqslant1$ and $L\geqslant2$, and set
\[
g=\left(1,\underbrace{c+1,\ldots,c+1}_{L-2\text{ terms}},c\right),
\]
with the middle string omitted when $L=2$. Equivalently,
\[
h=\left(c,\underbrace{c+1,\ldots,c+1}_{L-2\text{ terms}},1\right).
\]
Then Assumption~\ref{ass:pisot} holds.
\end{proposition}

\begin{proof}
The characteristic polynomial is
\[
P(X)=X^L-cX^{L-1}-(c+1)\sum_{k=1}^{L-2}X^k-1,
\]
where the sum is empty when $L=2$. For $0\leqslant t\leqslant1$, put
\[
f_t(X)
=
X^L-t\left(cX^{L-1}+(c+1)\sum_{k=1}^{L-2}X^k+1\right).
\]
For sufficiently small $t>0$, all zeros of $f_t$ lie strictly inside the
unit disk. By Descartes' rule of signs, for every $t>0$ the polynomial
$f_t$ has a unique positive zero $\varphi_t$, counted with multiplicity.

Suppose that $f_t(\alpha)=0$ for some $0<t\leqslant1$ and
$|\alpha|=1$. Writing $\alpha=e^{i\theta}$, multiplying the equation
$f_t(\alpha)=0$ by $\alpha-1$, and taking real parts gives
\[
t\bigl(c\cos((L-1)\theta)+\cos(L\theta)\bigr)
=
tc+1-(1-t)\cos\theta.
\]
The left-hand side is at most $t(c+1)$, while the right-hand side is at
least $t(c+1)$. Hence equality holds throughout. If $t<1$, this forces
$\cos\theta=1$; if $t=1$, equality on the left forces
$\cos((L-1)\theta)=\cos(L\theta)=1$, and again $\alpha=1$. Moreover,
\[
f_t(1)=0
\quad\Longleftrightarrow\quad
 t=\frac{1}{(c+1)(L-1)}.
\]
At this value, $1$ is the unique positive zero and, by the preceding Descartes argument, it is simple. Therefore no
zero other than $\varphi_t$ can meet the unit circle as $t$ varies from a
small positive value to $1$. By continuity of the zeros, all zeros of
$P=f_1$ other than its unique positive zero $\varphi$ have modulus
strictly less than $1$. Since
\[
P(1)=1-(c+1)(L-1)<0,
\]
we have $\varphi>1$, so $\varphi$ is a Pisot number.

It remains to prove irreducibility. If $P=FG$ with monic nonconstant
$F,G\in\mathbb Z[X]$, exactly one factor contains $\varphi$, while every
root of the other factor has modulus strictly less than $1$. But the
constant term of that factor is a nonzero integer, since $P(0)=-1$, whereas
its absolute value is the product of the moduli of its roots and is
therefore strictly less than $1$, a contradiction. Thus $P$ is
irreducible over $\mathbb Q$, and Assumption~\ref{ass:pisot} follows.
\end{proof}

The recurring example $g=(1,3,2)$ is the case $L=3$, $c=2$ of
Proposition~\ref{prop:nonbeta-family-pisot}.

It is useful to compare these algebraic cases with the classical
$\beta$-substitution setting. In the monotone and dominant branches, the
characteristic equation gives
\[
1=h_1\varphi^{-1}+h_2\varphi^{-2}+\cdots+h_L\varphi^{-L}.
\]
Since $h_L=1$, the corresponding candidate Parry reference word is the
periodic infinite word
\[
(h_1h_2\cdots h_{L-1}0)^\infty.
\]
In either case one also has
\[
h_1<\varphi<h_1+1,
\]
so $\lfloor\varphi\rfloor=h_1$. Under~(M), each shift of the reference
word by $1,\ldots,L-1$ positions is lexicographically smaller than the
original word by monotonicity; shifts by multiples of $L$ give equality.
Under~(D), the first digit after any shift by $1,\ldots,L-1$ positions is
already strictly smaller than $h_1$. Parry's
criterion~\cite{Parry1960} therefore shows that the greedy
$\varphi$-expansion of $1$ is
\[
h_1h_2\cdots h_L.
\]
Thus, in both branches, the companion substitution is the standard
simple-Parry $\beta$-substitution associated with this finite expansion;
see \cite[Section~2.4]{AkiyamaBargeBertheLeeSiegel2015}, and also
\cite{Hollander1996,ChangZeckendorfShift} for the relation with linear
numeration systems.

The family in Proposition~\ref{prop:nonbeta-family-pisot} shows that this
interpretation does not cover the whole companion family. If $L\geqslant3$,
direct evaluation gives
\[
P(c+1)<0<P(c+2),
\]
so $c+1<\varphi<c+2$ and hence $\lfloor\varphi\rfloor=c+1$. On the
other hand, the characteristic identity is
\[
1=c\varphi^{-1}+(c+1)\varphi^{-2}+\cdots
 +(c+1)\varphi^{-(L-1)}+\varphi^{-L},
\]
whose first digit is $c$. It therefore cannot be the greedy
$\varphi$-expansion of $1$. Thus every member of this family with
$L\geqslant3$ is a non-$\beta$ example in the present sense. In
particular, $g=(1,3,2)$ corresponds to $L=3$, $c=2$.

For the recurring example $g=(1,3,2)$, the balanced-pair algorithm provides
a classical alternative route to pure discreteness; for irreducible Pisot
substitutions, termination with coincidence is equivalent to pure discrete
spectrum
\cite[Theorem~5.3]{AkiyamaBargeBertheLeeSiegel2015}; see also
\cite[Theorem~5.1]{SirventSolomyak2002}.
The reversal argument of
Proposition~\ref{prop:automatic-rauzy-hypothesis}, however, yields the same
geometric conclusion uniformly throughout the companion family once
Assumption~\ref{ass:pisot} is known.

\subsection{A family without the finiteness property}
\label{sec:nonfinite-beta-family}

For $\beta>1$, let $\operatorname{Fin}(\beta)$ denote the set of
non-negative real numbers having a finite greedy $\beta$-expansion. Recall
that $\beta$ has the finiteness property~(F) if
\[
\operatorname{Fin}(\beta)
=
\Z[\beta^{-1}]\cap\R_{\geqslant0}.
\]
The next family satisfies the algebraic hypothesis of the main theorem but
fails this finiteness property.

\begin{proposition}\label{prop:nonfinite-family-pisot}
Let $L\geqslant4$ and $c\geqslant3$, and set
\[
g=
\left(1,c+1,\underbrace{c,\ldots,c}_{L-2\text{ terms}}\right).
\]
Equivalently,
\[
h=
\left(\underbrace{c,\ldots,c}_{L-2\text{ terms}},c+1,1\right).
\]
Then Assumption~\ref{ass:pisot} holds. Moreover, if $\varphi$ denotes the
dominant root of the characteristic polynomial, then the Pisot unit
$\varphi$ does not satisfy property~(F).
\end{proposition}

\begin{proof}
The characteristic polynomial is
\[
P_{L,c}(X)
=
X^L-c\sum_{k=2}^{L-1}X^k-(c+1)X-1.
\]
We first verify Assumption~\ref{ass:pisot}. Put
\[
F(z):=(z-1)P_{L,c}(z),\qquad
G(z):=z^L\bigl(z-(c+1)\bigr),\qquad
H(z):=z^2-cz-1.
\]
Then $F=G-H$. On $|z|=r$, writing $z=re^{i\theta}$ and
$x=\cos\theta$, a direct calculation gives
\[
\begin{aligned}
D_r(x)
&:=|G(z)|^2-|H(z)|^2\\
&=4r^2x^2
 -2r\bigl((c+1)r^{2L}+c(1-r^2)\bigr)x\\
&\quad
 +r^{2L}\bigl(r^2+(c+1)^2\bigr)
 -r^4-(c^2+2)r^2-1.
\end{aligned}
\]
At $r=1$,
\[
D_1(x)=2(x-1)\bigl(2x-(c-1)\bigr)\geqslant0
\qquad(-1\leqslant x\leqslant1),
\]
with equality only at $x=1$, while
\[
\left.\frac{\partial D_r(1)}{\partial r}\right|_{r=1}
=2c(Lc-c+1)>0.
\]
Hence, for $r>1$ sufficiently close to $1$, one has
$|G(z)|>|H(z)|$ on $|z|=r$. We may also assume $r<c+1$.
Rouch\'e's theorem then shows that $F$ and $G$ have the same number of
zeros in $|z|<r$, namely $L$. Since $P_{L,c}(1)\neq0$, the zero $z=1$ of $F$ is simple. It follows,
by taking such radii arbitrarily close to $1$, that $P_{L,c}$ has at least
$L-1$ zeros in the closed unit disk.

No zero of $P_{L,c}$ lies on the unit circle. Indeed, if
$P_{L,c}(z)=0$ and $|z|=1$, then $|G(z)|=|H(z)|$, so the preceding formula
forces $z=1$, whereas
\[
P_{L,c}(1)=-c(L-1)-1\neq0.
\]
By Descartes' rule of signs, $P_{L,c}$ has a unique positive root
$\varphi$. Direct evaluation gives
\[
P_{L,c}(c+1)=-1<0<P_{L,c}(c+2),
\]
so
\[
c+1<\varphi<c+2,
\qquad
\lfloor\varphi\rfloor=c+1.
\]
Consequently every other root has modulus strictly less than $1$, and
$\varphi$ is a Pisot number.

The same constant-term argument as in
Proposition~\ref{prop:automatic-pisot-beta-branches} gives irreducibility.
Indeed, if $P_{L,c}=UV$ with monic nonconstant $U,V\in\Z[X]$, the factor
not containing $\varphi$ has all its roots strictly inside the unit disk,
while its constant term is a nonzero integer because $P_{L,c}(0)=-1$.
This is impossible. Thus $P_{L,c}$ is the minimal polynomial of
$\varphi$, and Assumption~\ref{ass:pisot} holds.

It remains to show that property~(F) fails. We will construct an
eventually periodic greedy $\varphi$-expansion of $1$ of the form
\[
1=0.a_1a_2\cdots a_\ell\overline{a_{\ell+1}\cdots a_N},
\]
with $a_1=c+1$ and $a_N=1$, where
\[
\ell:=L^2-4L+3,\qquad n:=2L-1,\qquad N:=\ell+n=L^2-2L+2.
\]
The polynomial relation associated with this expansion uses a
cyclotomic product. Let
\[
\mathcal D_L
:=
\{d:d\mid L(L-1),\ d\nmid L,\ d\nmid(L-1)\}.
\]
The cyclotomic product
\[
T(X)
:=
\prod_{d\in\mathcal D_L}\Phi_d(X)
=
\frac{(X^{L(L-1)}-1)(X-1)}{(X^L-1)(X^{L-1}-1)}
\]
is a polynomial. We define integers $a_1,\ldots,a_N$ by
\[
a_1=c+1,\qquad
a_2=\cdots=a_L=0,\qquad
a_{L+1}=c,\qquad
a_{L+2}=\cdots=a_n=0,
\]
and, for $1\leqslant k\leqslant\ell$, by
\[
a_{n+k}=a_k-\delta_k,
\]
where
\[
\delta_k=
\begin{cases}
 1,& \text{if  } k=qL+1,\quad 0\leqslant q\leqslant L-4,\\
 -1,& \text{if  } k=q(L-1),\quad 1\leqslant q\leqslant L-3,\\
 -c,& \text{if  } k=qL+2,\quad 0\leqslant q\leqslant L-4,\\
 0,&\text{otherwise}.
\end{cases}
\]
Expanding the above cyclotomic product gives the polynomial identity
\[
P_{L,c}(X)T(X)
=
X^N-\sum_{k=1}^{n}a_kX^{N-k}-X^\ell
 +\sum_{k=1}^{\ell}(a_k-a_{n+k})X^{\ell-k}.
\]
Since $P_{L,c}(\varphi)=0$, this identity is equivalent to the eventually
periodic representation announced above, in base $\varphi$.

We record the digit bound needed for Parry's criterion. For $L=4$ and $L=5$, the coefficient sequences $(a_k)_{1\leqslant k\leqslant N}$ are, respectively,
\[
(c+1,0,0,0,c,0,0,c,c,1)
\]
and
\[
(c+1,0,0,0,0,c,0,0,0,c,c,0,1,0,c-1,c,1),
\]
so the required digit bound holds directly, and in both cases
the periodic part is nonzero. We may therefore assume $L\geqslant6$.

The recurrence $a_{n+k}=a_k-\delta_k$ partitions the coefficients
according to their indices modulo $n$. For $1\leqslant j\leqslant n$,
put
\[
m_j:=\left\lfloor\frac{\ell-j}{n}\right\rfloor,
\qquad
S_j(p):=j+pn\quad(0\leqslant p\leqslant m_j+1).
\]
Thus $a_{S_j(p+1)}=a_{S_j(p)}-\delta_{S_j(p)}$ for
$0\leqslant p\leqslant m_j$. We show below that all these coefficients,
except $a_1=c+1$, lie between $0$ and $c$.

Notice that
\[
m_j\leqslant\frac{\ell-1}{n}<\frac{L-3}{2}.
\]
Since
\[
n\equiv-1\pmod L,\qquad n\equiv1\pmod{L-1},
\]
the nonzero changes can occur only when the current index is
congruent to $1$ or $2$ modulo $L$, or to $0$ modulo $L-1$.
Each of these congruence events can occur at most once for a fixed
$j$: if $S_j(p)$ and $S_j(q)$ have the same residue modulo $L$,
or modulo $L-1$, then $p-q$ is divisible by $L$, or by $L-1$,
respectively, while $|p-q|<L/2$.

If both residues $2$ and $1$ modulo $L$ occur for a fixed $j$,
the residue $2$ occurs first, and the two changes are respectively
$+c$ and $-1$. Moreover, a congruence $0$ modulo $L-1$ cannot occur together with either of these events for
the same $j$. Indeed, suppose that
\[
S_j(u)\equiv0\pmod{L-1},
\qquad 0\leqslant u\leqslant m_j.
\]
Since $j+u\equiv0\pmod{L-1}$ and $m_j<(L-3)/2$, one has either
\[
j=L-1-u
\qquad\text{or}\qquad
j=2L-2-u.
\]
In the first case, $u\leqslant m_j$ implies
$u\leqslant(L-4)/2$. If $S_j(p)\equiv r\pmod L$ with
$0\leqslant p\leqslant m_j$ and $r\in\{1,2\}$, then
\[
p+u+r+1\equiv0\pmod L,
\]
whereas
\[
0<p+u+r+1<L,
\]
a contradiction. In the second case, the inequalities
$u,p\leqslant m_j\leqslant(\ell-j)/n$, with
$j=2L-2-u$, imply
\[
u,p\leqslant\frac{L-5}{2}.
\]
Hence $S_j(p)\equiv r\pmod L$, with
$0\leqslant p\leqslant m_j$ and $r\in\{1,2\}$, would give
\[
p+u+r+2\equiv0\pmod L,
\]
again with
\[
0<p+u+r+2<L,
\]
which is impossible. Thus, an event modulo $L-1$ cannot occur
together with either of the two events modulo $L$.

The exceptional starting indices
\[
j\in\{1,2,L-1,L,L+1,L+2,2L-2,2L-1\}
\]
are checked directly. The corresponding tuples
$(a_{S_j(p)})_{0\leqslant p\leqslant m_j+1}$ are, respectively,
with possible truncation,
\[
\begin{gathered}
(c+1,c,c,\ldots),\qquad
(0,c,c-1,c-1,\ldots),\qquad
(0,1,1,\ldots),\qquad
(0,0,0,\ldots),\\
(c,c-1,c-1,\ldots),\qquad
(0,c,c-1,c-1,\ldots),\qquad
(0,1,1,\ldots),\qquad
(0,0,0,\ldots).
\end{gathered}
\]
For all other starting indices $j$, the tuples have, again up
to truncation, one of the forms
\[
(0,0,\ldots),\qquad
(0,\ldots,0,1,1,\ldots),\qquad
(0,\ldots,0,c,c,\ldots),\qquad
(0,\ldots,0,c,c-1,c-1,\ldots).
\]
Consequently
\[
0\leqslant a_k\leqslant c
\qquad(2\leqslant k\leqslant N).
\]
The same residue check gives $a_\ell=0$; since
$\ell=(L-1)(L-3)$, one has $\delta_\ell=-1$, and therefore
\[
a_N=a_\ell-\delta_\ell=1.
\]
In particular, the periodic part is nonzero.

Since $a_1=c+1$ and every later digit is at most $c$,
every nontrivial shift of the digit sequence is
lexicographically smaller than the sequence itself. Parry's criterion therefore shows that the displayed eventually
periodic representation is the greedy $\varphi$-expansion of $1$.
Finally, put $\omega:=\varphi^{-1}$. Since $\varphi>c+1$,
\[
x:=1-(c+1)\omega>0,
\qquad
x\in\Z[\varphi^{-1}].
\]
The digit sequence $0.0a_2a_3\cdots$ represents $x$.
Since every shift of its fractional digit sequence begins
with a digit at most $c<a_1$, it is admissible with respect
to the greedy expansion of $1$ established above.
Hence it is the greedy $\varphi$-expansion of $x$. As $a_N=1$, its periodic part is nonzero, so the expansion
is infinite. Hence
$x\notin\operatorname{Fin}(\varphi)$, and property~(F) fails.
\end{proof}

\subsection{Unconditional consequences}

Translating the coefficient conditions of
Proposition~\ref{prop:automatic-pisot-beta-branches} back to the original
coefficients $g$, condition~(M) is equivalent to
\[
g_2\leqslant\cdots\leqslant g_L,
\]
while condition~(D) is equivalent to
\[
\sum_{k=1}^{L-1}g_k<g_L.
\]

\begin{corollary}\label{cor:unconditional-prime-terminal-blocks}
The conclusions of Theorem~\ref{thm:prime-terminal-blocks} hold
unconditionally if either of the coefficient conditions
\[
g_2\leqslant\cdots\leqslant g_L,
\qquad
\sum_{k=1}^{L-1}g_k<g_L
\]
holds, or if $g$ belongs to one of the two families
\[
g=\left(1,\underbrace{c+1,\ldots,c+1}_{L-2\text{ terms}},c\right),
\qquad c\geqslant1,
\]
and
\[
g=\left(1,c+1,\underbrace{c,\ldots,c}_{L-2\text{ terms}}\right),
\qquad L\geqslant4,\quad c\geqslant3.
\]
\end{corollary}

\begin{proof}
In the first two cases,
Proposition~\ref{prop:automatic-pisot-beta-branches} gives
Assumption~\ref{ass:pisot}; the two remaining families are covered by
Propositions~\ref{prop:nonbeta-family-pisot} and
\ref{prop:nonfinite-family-pisot}, respectively.
Proposition~\ref{prop:automatic-rauzy-hypothesis} then gives
Theorem~\ref{thm:rauzy-fundamental-domain}, so
Theorem~\ref{thm:prime-terminal-blocks} applies.
\end{proof}

This completes the proofs of Theorems~\ref{thm:intro-main}, \ref{thm:intro-several-blocks}, and \ref{thm:intro-unconditional-cases}. In particular, in all these cases the results of
Sections~\ref{sec:volume} and~\ref{sec:prime-blocks} whose only additional
geometric input is
Theorem~\ref{thm:rauzy-fundamental-domain} hold unconditionally.
This includes polynomial sampling of terminal blocks, prime genericity of
the substitution fixed word, prime polynomial sampling, independence from
fixed congruence classes, the asymptotic Bernoulli law, fixed-shift
correlations, Möbius orthogonality, terminal Chebotarev, and the narrow
Green--Tao consequence.

The section therefore separates the geometric input from its algebraic
verification. Barge's theorem supplies pure discrete spectrum for the
reversed tiling system; the symbolic--tiling equivalence together with word
reversal transfers this to the original substitution; and
Proposition~\ref{prop:automatic-rauzy-hypothesis} translates the resulting
Rauzy-fractal structure into the windows and lattice used in the arithmetic
argument, thereby establishing
Theorem~\ref{thm:rauzy-fundamental-domain}.

%% file: sections/09-outlook-and-further-directions.tex
\section{Outlook and open problems}\label{sec:conclusion}

The results above separate the problem into a structural--geometric
coding and a classical equidistribution input. For a fixed terminal word
$B$, the exact cylinder law
\[
\frac{\vol(R_B)}{\vol(\widetilde R)}=\omega^{\len(B)}
\]
turns the terminal condition into membership of a measurable window for
a linear orbit on a compact torus. The argument uses two closely related
toral codings. The prefix torus $R/\mathcal L$ gives
prime genericity and polynomial sampling for the substitution fixed word,
directly at the original word positions. The full terminal torus gives
polynomial sampling for terminal cylinders, congruence and fixed-shift laws,
M\"obius orthogonality, terminal Chebotarev, and arithmetic progressions
inside fixed terminal prime classes.

Many other classical prime questions admit analogous terminal refinements,
but merely listing them would add little. We therefore record only a few
directions in which the present structure points to a specific new
difficulty, notably quantitative shrinking targets and extensions beyond
the unimodular irreducible Pisot setting considered in this paper.

\subsection{The terminal torus as an auxiliary arithmetic space}
\label{sec:terminal-auxiliary-space}

The toral coding developed above suggests a general principle for adding
terminal conditions to arithmetic problems. Suppose that one starts with
an arithmetic average over integers, primes, or another arithmetic set,
and wishes to restrict it further by imposing a fixed terminal condition
\[
\TBe{e}(n)=B.
\]
The point of the terminal torus is that this extra condition can be replaced,
up to a negligible exceptional set, by membership of $n\mathbf v$ in a
fixed toral window.

Indeed, for every fixed $e$ one has the exact partition
\[
1=\sum_{B\in\mathcal B_e}
\mathbf 1_{\{\TBe{e}(n)=B\}}
\qquad(n\in\N_0).
\]
By Propositions~\ref{prop:terminal-cylinder-partition},
\ref{prop:toral-congruence}, and
\ref{prop:rational-independence}, together with
Lemma~\ref{lem:regular-toral-coding}, each fixed terminal indicator
agrees, outside a natural density-zero exceptional set, with the toral
window indicator
\[
\mathbf 1_{\Omega_B}(n\mathbf v).
\]
The corresponding toral window has Haar measure
\[
\mu_{\mathrm{Haar}}(\Omega_B)=\omega^{\len(B)}.
\]

Approximating $\mathbf 1_{\Omega_B}$ by a finite Fourier sum reduces the
inserted terminal condition to a constant term together with nonconstant
modes of the form
\[
e^{2\pi i n\,\mathbf m\cdot\mathbf v},
\qquad
\mathbf m\in\Z^d\setminus\{0\}.
\]
Thus a terminal refinement of an arithmetic theorem is possible whenever
the underlying arithmetic average remains controllable after multiplication
by these irrational toral phases, and when the boundary and approximation
errors can also be controlled against the arithmetic weight.

This gives a useful test for possible extensions of the method. The toral
coding does not solve a new arithmetic problem by itself; rather, it allows
a fixed terminal condition to be inserted into an existing one whenever
the required twisted estimates are available. If the untwisted problem is
already open, the toral lift does not remove that obstruction.

\subsection{Shrinking cylinders and quantitative distribution}
\label{sec:quantitative-terminal-directions}

The theorems of this paper prescribe a fixed terminal word while
$X\to\infty$. A first genuinely new regime begins when the amount of
prescribed terminal information grows with $X$. Let
$B_X\in\mathcal B_{e(X)}$ and put
\[
m(X):=\len(B_X)\to\infty.
\]
The self-affine identity from Section~\ref{sec:recombination} gives
\[
R_{B_X}=v_{B_X}+K^{m(X)}\widetilde R,
\]
so one is counting primes in a Rauzy target whose scale shrinks with
$X$.

\begin{problem}\label{prob:shrinking-cylinders}
Determine a nontrivial range of $m(X)\to\infty$ for which
\[
\#\{p\leqslant X:\TBe{e(X)}(p)=B_X\}
\sim \omega^{m(X)}\pi(X)
\]
holds uniformly over admissible words $B_X$ of total digit-length
$m(X)$.
\end{problem}

For such a uniform asymptotic one should at least require the expected
main term $\omega^{m(X)}\pi(X)$ to tend to infinity. On the scale
\[
m(X)=c\log_\varphi X+O(1),
\]
one has
\[
\omega^{m(X)}\pi(X)\asymp \frac{X^{1-c}}{\log X},
\]
so this necessary condition forces $c<1$. The problem is a
recurrence-numeration analogue of prescribing a growing number of ordinary
digits of a prime; compare Bourgain~\cite{Bourgain2015} and
Swaenepoel~\cite{Swaenepoel2020}. Here, however, the shrinking sets are
fractal Rauzy cylinders rather than boxes in an independent digit
coordinate system.

Even for a fixed terminal word $B$, the present argument gives no uniform
error term or level of distribution. A natural quantitative next step is
therefore to seek Bombieri--Vinogradov-type estimates for primes restricted
by the exact terminal condition
\[
\TBe{e}(n)=B.
\]
Existing results for Beatty primes and for von Mangoldt sums twisted by
linear or more general nilsequence phases show that substantial analytic
input is available for smooth irrational restrictions
\cite{JiKazdanMcDonald2020,Matomaki2009,ShaoTeravainen2021}.
The new issue here is quantitative smoothing of the nonsmooth
Rauzy-cylinder indicator and control of its fractal boundary. The same
difficulty would enter short-interval versions of the prime terminal-block
theorem.

\subsection{Beyond the unimodular irreducible Pisot setting}
\label{sec:beyond-unimodular}

The standing normalization $g_1=1$ makes the companion substitution
unimodular. In this discussion we extend the definitions of the weights
and semi-proper blocks from Section~\ref{sec:general-data} to
$g_1=h_L\geqslant1$, keeping
$b_{i,j}=(j,h_{i-1},\ldots,h_1)$ for $0\leqslant j<h_i$.
The finite-level induction in Theorem~\ref{thm:block-language}
remains valid in this setting, now with $h_L$ semi-proper blocks of
digit-length $L$. If the normalization $g_1=1$ is dropped, the natural
companion extension of Definition~\ref{def:intro-associated-substitution}
is
\[
\sigma(A_i)=A_1^{h_i}A_{i+1}\quad(1\leqslant i<L),
\qquad
\sigma(A_L)=A_1^{h_L}.
\]
Its incidence matrix has characteristic polynomial $P_h$ and determinant
of absolute value $h_L$. It is natural to ask which parts of the present
argument survive in this non-unimodular setting. The simplest test family
is the order-two family
\[
g=(2,a),\qquad a\geqslant2,
\]
for which $h=(a,2)$ and
\[
P_a(X)=X^2-aX-2.
\]
The polynomial $P_a$ is irreducible for $a\geqslant2$: otherwise
$a^2+8=b^2$ for some integer $b$, so
\[
(b-a)(b+a)=8,
\]
which forces $a=1$. Its dominant root $\varphi_a$ is Pisot, since the
other root is $-2/\varphi_a$ and $\varphi_a>2$. Since
$N(\varphi_a)=-2$, these are non-unit Pisot systems. Rauzy fractals for
non-unit Pisot substitutions naturally live in a representation space
with non-Archimedean components; see Minervino--Thuswaldner
\cite{MinervinoThuswaldner2014}.

This family already exhibits a local obstruction invisible in the
unimodular theory. If $a$ is even, then
\[
H_1=1,
\qquad
H_k\equiv0\pmod2\quad(k\geqslant2).
\]
Hence an integer
\[
n=\sum_{k\geqslant1}a_kH_k
\]
satisfies
\(
n\equiv a_1\pmod2.
\)
Since the first digit of the terminal block $b_{i,j}$ is $j$, every
terminal class with even first digit contains no odd prime. In particular,
the positive law $\omega^{\len(B)}$ from the unimodular setting cannot
extend unchanged to all classes in this family; a correct statement must
incorporate the local factor at $2$.

The simplest case $g=(2,2)$ already makes the expected correction visible
numerically. Here $\varphi=1+\sqrt3$, $\omega=\varphi^{-1}$, and the four
one-block classes are
\[
(0),\qquad (1),\qquad (0,2),\qquad (1,2).
\]
Apart from the prime $2$, only the two classes with first digit $1$ can
occur at primes. Among the first $100000$ primes, numerical computation
gives $73225$ occurrences of $(1)$ and $26774$ occurrences of $(1,2)$.
Among the $99999$ odd primes these proportions are approximately
$0.732257$ and $0.267743$, compared with
\[
2\omega=0.732050\ldots,
\qquad
2\omega^2=0.267949\ldots.
\]
This suggests that in this example the local factor at $2$ simply
renormalizes the two admissible formal weights by a factor of $2$.

\begin{problem}\label{prob:nonunit-order-two}
Determine the prime terminal-block law for $g=(2,a)$, $a\geqslant2$,
including the local $2$-adic factors, and identify the appropriate
non-unit Rauzy representation space that replaces the torus used in
this paper.
\end{problem}

The adjacent case $g=(2,1)$ is different again. Here
\[
P(X)=X^2-X-2=(X-2)(X+1),
\qquad
H_n=\frac{2^n-(-1)^n}{3},
\]
so the characteristic polynomial is reducible. The weights $H_n$ are the Jacobsthal numbers. Greedy expansions
in this system may contain the digit $2$, whereas Carlitz,
Scoville, and Hoggatt~\cite{CarlitzScovilleHoggatt1972}
studied a different non-greedy representation system.

Although its dominant
root $2$ is itself a Pisot number, $P$ is not its minimal polynomial;
hence this system lies outside Assumption~\ref{ass:pisot} and the
irreducible Pisot setting of the paper. The three semi-proper blocks are
\[
(0),\qquad (0,1),\qquad (1,1),
\]
and numerical experiments suggest the frequencies
\[
\frac12,\qquad\frac14,\qquad\frac14
\]
among the primes, matching the formal weights $2^{-\len(b)}$.

\begin{problem}\label{prob:reducible-order-two}
Prove or disprove the prime terminal-block frequencies for
$g=(2,1)$. In particular, determine whether the result can be reduced
to explicit dyadic congruence classes and Dirichlet's theorem, or
whether a genuinely morphic prime-number theorem is required.
\end{problem}

These two order-two families indicate that leaving the unimodular
irreducible setting introduces genuinely new local arithmetic, rather
than merely a technical modification of the present proof.

The auxiliary-space principle also suggests many further arithmetic
refinements, but in several natural examples the main obstruction already
lies in the untwisted problem. Higher M\"obius correlations and
Bateman--Horn-type prime-value questions, for instance, lead back at the
zero frequency to deep open problems. In other settings, such as
Sato--Tate or prime counts in short intervals, one would instead have to
determine whether the available analytic estimates remain uniform after
the insertion of the irrational toral twists and the nonsmooth cylinder
cutoff.

A qualitatively different difficulty appears for lacunary sampling.
Replacing $n\mathbf v$ by, for example, $2^n\mathbf v$ leads to
equidistribution questions for lacunary orbits of algebraic coordinates.
Already in one dimension, equidistribution of
$2^n\alpha \pmod 1$ is equivalent to normality of $\alpha$ in base~$2$.
For algebraic irrational $\alpha$, this is a notoriously difficult
Diophantine problem. Such questions lie well beyond the equidistribution
input used in the present paper.

\section*{Use of generative AI}
Generative AI tools were used during the preparation of this manuscript
for bibliographic searches, as well as for English-language reformulation. They were not used to obtain mathematical results or
proofs. The authors take full responsibility for the mathematical content
and the final text.